\documentclass[11pt, letterpaper]{amsart}
\usepackage{amssymb}
\usepackage{enumitem}
\usepackage[colorlinks = true,
            linkcolor = blue,
            urlcolor  = blue,
            citecolor = blue,
            anchorcolor = blue]{hyperref}
\usepackage{caption}
\newtheorem{theorem}{Theorem}[section]
\newtheorem{proposition}{Proposition}[section]
\newtheorem{lemma}[theorem]{Lemma}
\newtheorem{corollary}[theorem]{Corollary}

\theoremstyle{definition}
\newtheorem{condition}[theorem]{Condition}
\newtheorem{step}{Step}
\newtheorem{definition}[theorem]{Definition}
\newtheorem{notation}[theorem]{Notation}

\usepackage{colortbl}
\usepackage{todonotes}

\theoremstyle{remark}
\newtheorem{remark}[theorem]{Remark}
\DeclareMathOperator{\re}{Re}
\DeclareMathOperator{\im}{Im}
\DeclareMathOperator{\sgn}{sgn}
\DeclareMathOperator{\sech}{sech}
\DeclareMathOperator{\Var}{Var}
\DeclareMathOperator{\Cov}{Cov}

\DeclareMathOperator{\erf}{erf}

\DeclareMathOperator{\supp}{supp}

\numberwithin{equation}{section}
\usepackage{graphicx}
\graphicspath{ {./images/} }
\usepackage{tikz}
\usepackage{tkz-tab}
\usepackage{amsrefs}
\def\MR#1{\href{https://mathscinet.ams.org/mathscinet-getitem?mr=#1}{MR#1}}

\begin{document}

\title{Real roots of non-centered random polynomials}      
   \author{Yen Q. Do}
\address{Department of Mathematics, University of Virginia,  Charlottesville, Virginia 22904, USA}
\email{yendo@virginia.edu}

\author{Nhan D. V. Nguyen}
\address{Department of Mathematics,
University of Colorado Boulder,
Campus Box 395,
Boulder, CO 80309-0395}
\email{Nhan.Nguyen-1@colorado.edu, nguyenduvinhan@qnu.edu.vn}

\author{Sean O'Rourke}
\address{Department of Mathematics,
University of Colorado Boulder,
Campus Box 395,
Boulder, CO 80309-0395}
\email{sean.d.orourke@colorado.edu}
\thanks{S. O'Rourke has been supported in part by NSF CAREER grant DMS-2143142.}


\date{\today}
\begin{abstract}
We study the fluctuations of the number of real roots of random polynomials with independent, nonzero-mean coefficients. Such non-centered ensembles arise naturally in signal-plus-noise models and in random perturbations of deterministic polynomials. While Ibragimov and Maslova (1971) established the leading asymptotics of the expected number of real roots for non-centered polynomials with i.i.d.\ coefficients, the corresponding variance asymptotics and central limit theorem have remained open for more than fifty years. This stands in sharp contrast to the centered case, where the fluctuation theory is now well understood across a wide range of ensembles.

We resolve this gap by developing novel comparison principles that reduce the fluctuation theory of a non-centered ensemble to that of its centered counterpart. These principles yield sharp variance asymptotics and central limit theorems for broad classes of ensembles, including Kac and hyperbolic polynomials, their derivatives, and related extensions. In particular, for both Kac and hyperbolic polynomials, the leading variance constant equals exactly one-half of that in the centered case, reflecting asymmetric suppression of fluctuations across the two regions where roots concentrate. Our results provide the first comprehensive fluctuation theory for the number of real roots of non-centered random polynomials.
\end{abstract} 
\maketitle
\tableofcontents
\section{Introduction} 
We consider random polynomials with independent real-valued coefficients,
\begin{equation} \label{e.pn}
    P_n(x)=\sum_{j=0}^n \omega_j x^j,
\end{equation}
and study the asymptotic distribution of the number of real roots $N_n$ as $n\to \infty$.

The behavior of $N_n$ has been a central subject in probability theory for nearly a century. Early work of Bloch and P\'olya \cite{BP32} showed that when the coefficients are uniformly distributed on $\{-1,0,1\}$, one has $N_n=O(\sqrt{n})$ with high probability. Littlewood and Offord \cites{LO43, LO45, LO48} subsequently refined this bound to polylogarithmic order. A decisive breakthrough came with Kac \cites{Kac43, Kac49}, who derived an explicit integral formula for the density of real roots and used it to show that for i.i.d.\ centered Gaussian or uniform coefficients\footnote{Throughout the paper, asymptotic notation is understood in the limit as the degree $n$ tends to infinity.  A precise description of the notation is given in Notation \ref{not:asymptotic}.},
\[
\mathbb{E}[N_n] =\left(\frac{2}{\pi}+o(1)\right) \log n.
\]  
Random polynomials with i.i.d.\ coefficients are now called \emph{Kac polynomials} in his honor. Kac's density formula is the prototype of the Kac--Rice method \cites{EK95, TV15, AW09}, which has since become one of the principal tools in the study of real zeros of Gaussian processes \cites{BD97, BS86, Das69, Das72, Dun68, Dun70, Dun72, F86, F98, FK20CCM, FK20JTP, LPX16, LPX18, Qua70, Sam79, SGF98, St69, W88}.

The logarithmic growth of $\mathbb{E}[N_n]$ is remarkably robust. Erd\H{o}s and Offord \cite{EO56} established this phenomenon for Rademacher coefficients, while Ibragimov and Maslova \cites{IM68, IM71, IM71I} extended it to coefficients whose distributions belong to the domain of attraction of the normal law. More recently, universality methods have further generalized these results to random polynomials with independent, non-identically distributed coefficients under mild moment assumptions \cites{DHNV15, KLN22, NNV16, NV22A, TV15}.

The study of \emph{second-order fluctuations} of $N_n$ is substantially more delicate. In a pair of seminal papers, Maslova \cites{M74V, M74D} pioneered this direction by establishing the variance asymptotics
\[
    \Var[N_n] =\left[\frac{4}{\pi}\left(1-\frac{2}{\pi}\right)+o(1)\right]\log n,
\]
and the asymptotic normality of $N_n$ (CLT)
\[
\frac{N_n -\mathbb{E}[N_n]}{\sqrt{\Var[N_n]}} \xrightarrow{d} \mathcal{N}(0,1),
\]  
for Kac polynomials with centered coefficients. Maslova's approach, based on approximating $N_n$ by a sum of weakly dependent local contributions over carefully chosen intervals, has become a foundational tool in the subject.

These results extend naturally to \emph{generalized Kac ensembles}, which include derivatives of classical Kac and hyperbolic polynomials and are closely connected to the study of critical points. Variance asymptotics in this setting were recently obtained in \cite{DN25}, while Nguyen and Vu \cite{NV22D} proved a CLT under the condition $\Var[N_n]\ge \epsilon\log n$ for some constant $\epsilon>0$, combining universality arguments with Maslova’s method.

More broadly, Maslova’s work has stimulated extensive efforts to understand fluctuations across the classical random polynomial ensembles, drawing on tools such as Kac--Rice formulas, Wiener chaos expansions, Edgeworth theory, and moment methods. By now, variance asymptotics and CLTs have been established for a wide range of \emph{centered models}. For \emph{Gaussian elliptic polynomials}, Bleher and Di \cite{BD97} and Dalmao \cite{Dal15} obtained the first results, with subsequent refinements by Ancona and Letendre \cite{AL21}, Gass \cite{Gas23}, and Nguyen \cite{N24}; extensions to non-Gaussian settings remain open. For \emph{trigonometric polynomials}, Do, Nguyen, and Nguyen \cite{DNN22} and Bally, Caramellino, and Poly \cite{BCP19} extended earlier Gaussian results of  Granville and Wigman \cite{GW11}, Aza\"is and Le\'on \cite{AL13}, and Aza\"is, Dalmao, and Le\'on \cite{ADL16} to non-Gaussian coefficients. For \emph{Weyl ensembles}, Schehr and Majumdar \cite{SM08} and Do and Vu \cite{DV20} established variance asymptotics and CLTs in the Gaussian case, while Aguirre, Nguyen, and Wang \cite{ANW25} recently extended the variance results to the sub-Gaussian regime, leaving the CLT open. For \emph{Gaussian orthogonal ensembles}, Do, Nguyen, Nguyen, and Pritsker \cite{DNNP21} proved a CLT that complements the variance asymptotics of Lubinsky and Pritsker \cites{LP21, LP22}, with the non-Gaussian case still unresolved. Altogether, these works provide a nearly complete picture of second-order behavior for centered models.

In sharp contrast, the \emph{non-centered} regime, in which a non-negligible proportion of the coefficients carry nonzero means, has resisted a complete second-order treatment. Such ensembles arise naturally in signal-plus-noise models and random perturbations of deterministic polynomials. Ibragimov and Maslova \cite{IM71II} determined the leading asymptotics of $\mathbb{E}[N_n]$ for non-centered Kac polynomials as early as 1971, showing that the logarithmic growth persists but with the leading constant reduced by a factor of two. However, their method is limited to first-order asymptotics and provides no information on fluctuations. For over fifty years, no variance asymptotics or CLT for $N_n$ were known for any non-centered ensemble, whether Kac, elliptic, Weyl, trigonometric, or orthogonal.

The present paper resolves these problems for the Kac polynomials and their generalizations, the hyperbolic polynomials. More precisely, we establish sharp variance asymptotics and CLTs for $N_n$ for Kac and hyperbolic polynomials with non-centered coefficients, their derivatives, and related generalizations, in both Gaussian and non-Gaussian settings. The key technical contribution is a set of novel comparison principles relating the real root count of a non-centered polynomial to that of its centered counterpart. As a byproduct, we derive an explicit formula for the two-point correlation function of real zeros of non-centered Gaussian processes, which may be of independent interest.

\subsection{Main results}
As a primary application of our framework, we present results on the variance and limiting distribution of the number of real roots of random polynomials with non-centered coefficients, directly resolving the open problems stemming from the foundational works of Ibragimov and Maslova \cites{IM71, IM71I, IM71II, M74V, M74D}. Although their methods yield the leading asymptotics of $\mathbb{E}[N_n]$ in the non-centered case \cite{IM71II}, the precise variance asymptotics and the asymptotic normality of $N_n$ are established here for the first time.

\begin{theorem}[Non-centered Kac polynomials] \label{t.kac} 
Let $P_n(x)=\sum_{j=0}^n\xi_j x^j$, where $(\xi_j)_{j=0}^n$ are real-valued independent random variables, with 
\[\mathbb E[\xi_j]=\mu,\quad \Var[\xi_j]=1,\quad \mbox{and }\quad \mathbb E[|\xi_j|^{2+\varepsilon_0}]<C_0\]
for all $j$ and some constants $\mu\ne 0$, $\varepsilon_0>0$, and $C_0>0$. Then, as $n\to \infty$, 
    \[
    \Var[N_{n}]=\left[\frac{2}{\pi}\left(1-\frac{2}{\pi}\right) + o(1) \right]\log n.
    \]
Furthermore, $N_{n}$ satisfies the CLT; that is, as $n\to \infty$, 
    \[
    \frac{N_{n}-\mathbb E[N_{n}]}{\sqrt{\Var[N_{n}]}} \xrightarrow{d} \mathcal N(0,1).
    \]
\end{theorem}
The variance constant $\frac{2}{\pi}\left(1-\frac{2}{\pi}\right)$ is exactly one-half of that obtained by Maslova \cites{M74V} for centered Kac polynomials. This reduction reflects a genuine asymptotic effect of the nonzero mean. In the centered case, real roots accumulate symmetrically near both $\pm 1$, and both regions contribute to the leading-order variance. In the non-centered setting, the deterministic component suppresses roots on one side, leaving a single dominant concentration region that governs the leading-order variance.


Figures \ref{fig1}, \ref{fig2}, and \ref{fig3} present numerical simulations supporting Theorem \ref{t.kac}. For each degree $2\le n\le 2000$, we generate $1500$ independent realizations with i.i.d.\ coefficients of mean $2$ and variance $1$.

\begin{figure}[htbp]
\centering
\includegraphics{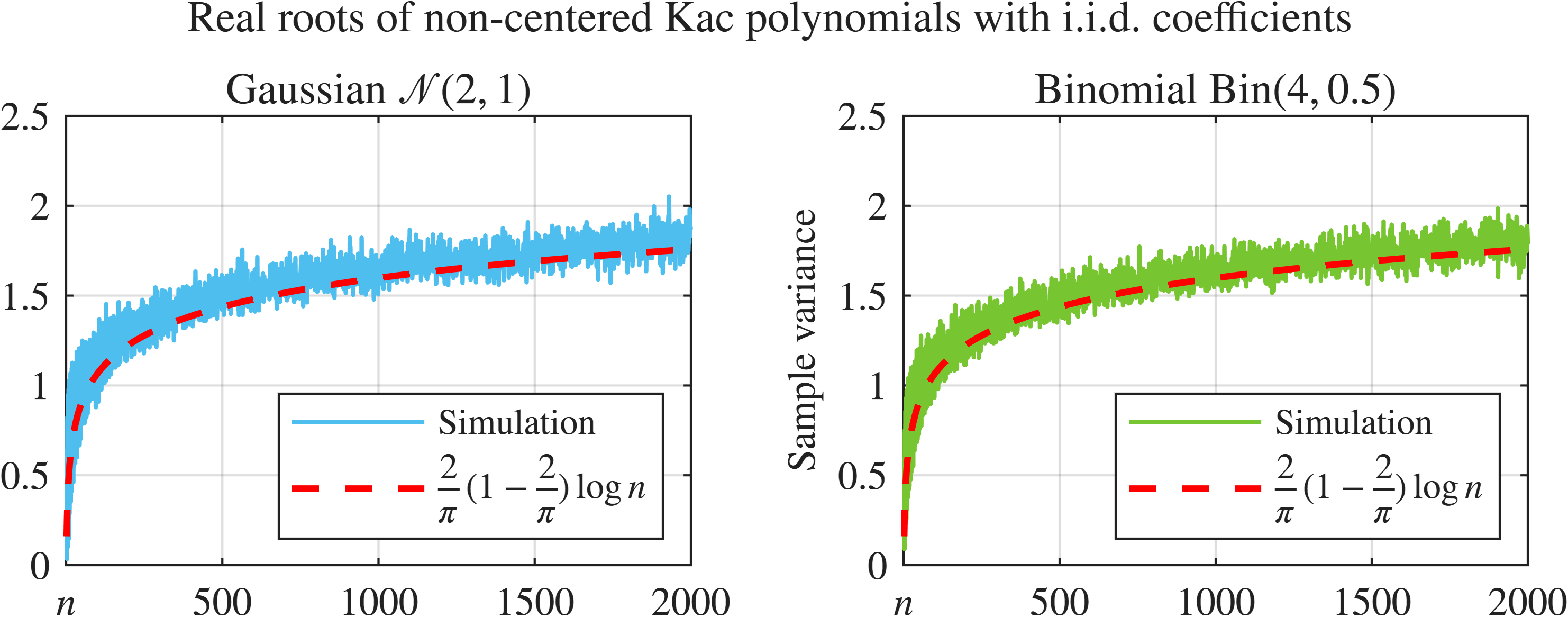}
\caption{Sample variance of $N_n$ as a function of $n$.}
\label{fig1}
\end{figure}

\begin{figure}[htbp]
\centering
\includegraphics{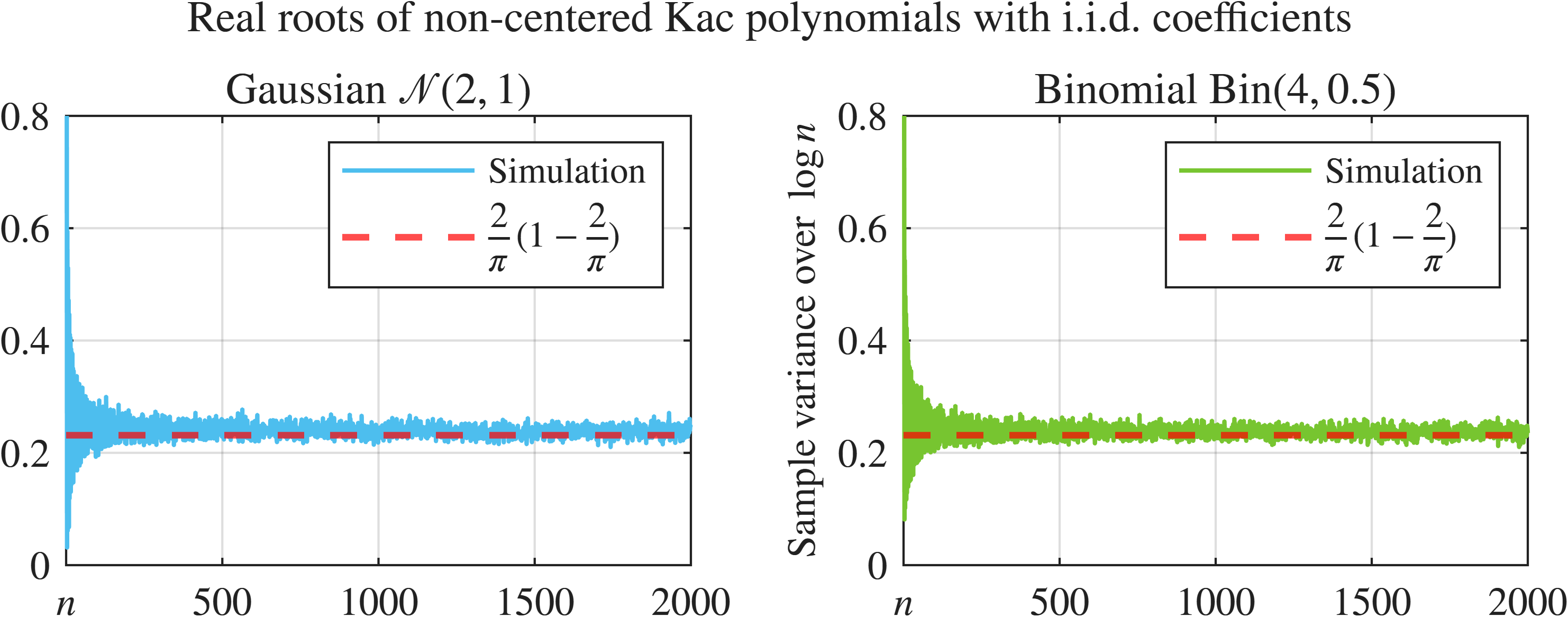}
\caption{Sample variance of $N_n$, normalized by $\log n$, converging to the theoretical limit $\frac{2}{\pi}\left(1-\frac{2}{\pi}\right)$.}
\label{fig2}
\end{figure}

\begin{figure}[htbp]
\centering
\includegraphics{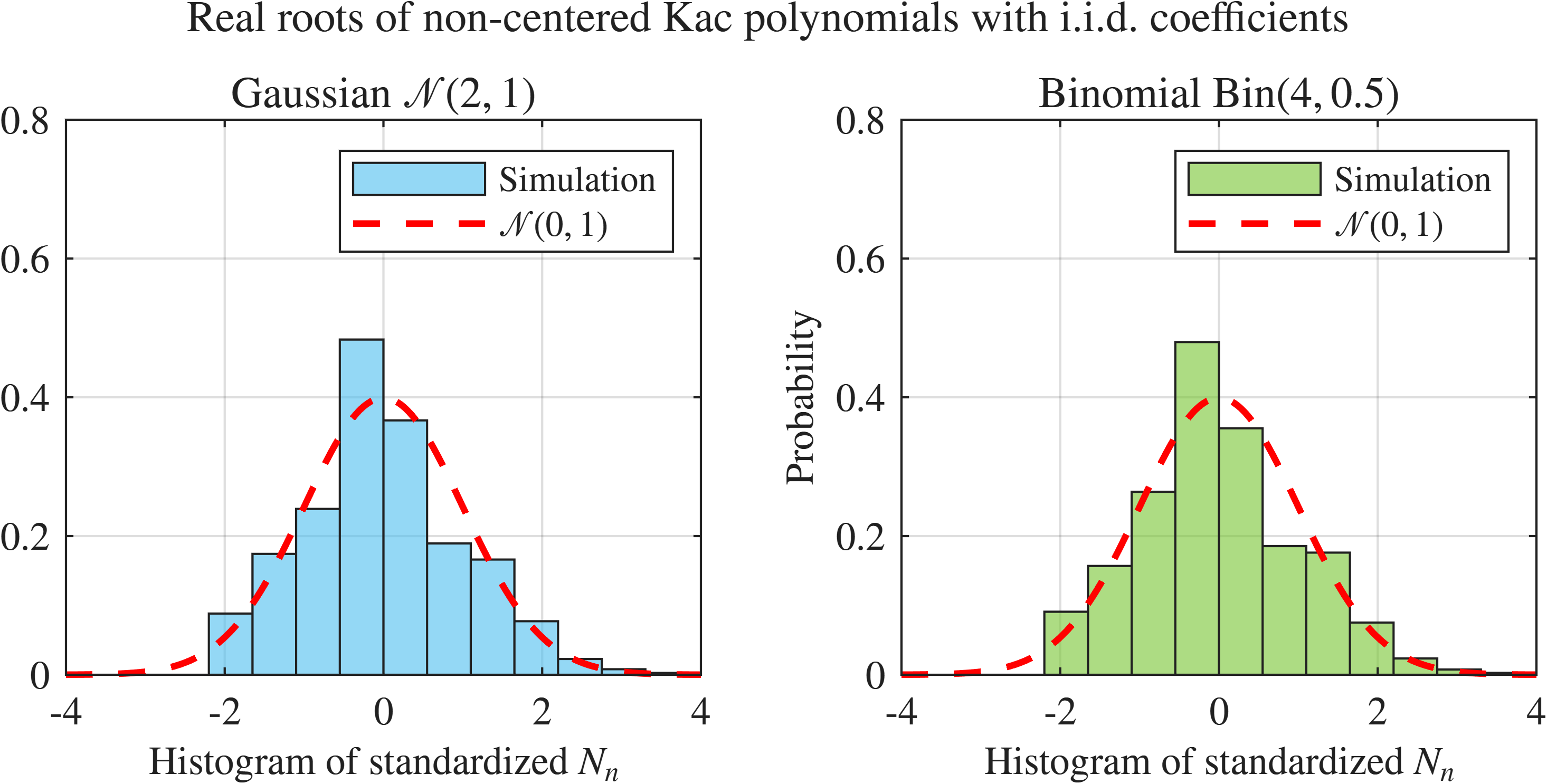}
\caption{Empirical distributions of the standardized $N_n$ compared with the standard normal density.}
\label{fig3}
\end{figure}

Theorem \ref{t.kac} extends naturally to derivatives of non-centered Kac polynomials. To state this precisely, we adopt the notation of \cite{DN25}. For any real number $\tau>-1/2$, define
    \[
    f_{\tau}(u):=\left(\sqrt{1-\Delta_{\tau}^2(u)}+\Delta_{\tau}(u) \arcsin \Delta_{\tau}(u)\right)\Sigma_{\tau}(u)-1,
    \]
where 
\begin{align*}
    \Delta_{\tau}(u)&=u^ {\tau+1/2}\frac{u\left(1-u^{2\tau+1}\right)-(2\tau+1)\left(1-u\right)}{1-u^{2\tau+1}-(2\tau+1)u^{2\tau+1}\left(1-u\right)},\\
    \Sigma_{\tau}(u)&=\frac{1-u^{2\tau+1}-(2\tau+1)\left(1-u\right)u^{2\tau+1}}{\left(1-u^{2\tau+1}\right)^{3/2}},
\end{align*}
and let
\begin{equation}
    \label{kappa}
    \kappa_\tau:=\frac{1}{\pi}\left(\frac{2\tau+1}{\pi}\int_{0}^{\infty}f_{\tau}\left(\sech^2v\right)dv+\frac{\sqrt{2\tau+1}}{2}\right).
\end{equation}
A direct computation (see \cite{DN25}) yields $\kappa_0=\frac 1{\pi}\left(1-\frac 2 {\pi}\right)$, so that the variance constant in Theorem \ref{t.kac} equals $2\kappa_0$, which is consistent with the following theorem.

\begin{theorem}[Derivatives of non-centered Kac polynomials]\label{t.kac-der} 
Let $P_n(x)=\sum_{j=0}^n\xi_j x^j$, where $\left(\xi_j\right)_{j=0}^n$ are real-valued independent random variables, with 
\[\mathbb E[\xi_j]=\mu,\quad \Var[\xi_j]=1,\quad \mbox{and }\quad \mathbb E[|\xi_j|^{2+\varepsilon_0}]<C_0\]
for all $j$ and some constants $\mu\ne 0$, $\varepsilon_0>0$, and $C_0>0$. For any integer $\ell\ge 0$, let $N_{n,\ell}$ denote the number of real roots of the $\ell$th derivative $P_{n}^{(\ell)}$\footnote{The case $\ell=0$ corresponds to $P_n$.}. Then as $n\to \infty$, 
    \[
    \Var[N_{n,\ell}]=\left[\kappa_\ell+\frac 1\pi \left(1-\frac 2\pi\right)+o(1)\right]\log n.   
    \]
Moreover, $N_{n,\ell}$ satisfies the CLT. 
\end{theorem}

Theorems \ref{t.kac} and \ref{t.kac-der} are special cases of a more general result for non-centered hyperbolic polynomials. For $L > 0$, define
\begin{equation}\label{e.hyperbolic}
    P_{n,L}(x):= \xi_0 + \sqrt L \xi_1 x + \dots + \sqrt{\frac{L\left(L+1\right)\cdots \left(L+n-1\right)}{n!}}\xi_n x^n,
\end{equation}
where $(\xi_j)_{j=0}^n$ are independent random variables of comparable size. While the coefficients are often assumed to be i.i.d., our results allow for more general assumptions. For a historical overview of this polynomial, see \cite{HKPV09}. When $L=1$, the polynomial $P_{n, L}$ simplifies to the standard Kac polynomial.

\begin{theorem}[Non-centered hyperbolic polynomials] \label{t.hyperkac} 
Fix $L>0$ and let $P_{n,L}$ be the random hyperbolic polynomial defined in \eqref{e.hyperbolic}, where $(\xi_j)_{j=0}^n$ are independent real-valued random variables satisfying
\[\mathbb E[\xi_j]=\mu,\quad \Var[\xi_j]=1,\quad \mbox{and }\quad \mathbb E[|\xi_j|^{2+\varepsilon_0}]<C_0\]
for all $j$, with constants $\mu\ne 0$, $\varepsilon_0>0$, and $C_0>0$. For any integer $\ell\ge 0$, let $N_{n,\ell}$ represent the number of real roots of the $\ell$th derivative $P_{n,L}^{(\ell)}$. Then, with $\tau=\ell+\frac{L-1}{2}$ and $\kappa_\tau$ defined in \eqref{kappa}, we have
\begin{equation} \label{e.varhyper}
    \Var[N_{n,\ell}]=\left[\kappa_\tau+\frac 1\pi \left(1-\frac 2\pi\right)+o(1)\right]\log n.   
\end{equation}
Moreover, $N_{n,\ell}$ satisfies the CLT. 
\end{theorem}

Theorem \ref{t.hyperkac} is a direct consequence of the broad comparison principles developed in the next subsection. These principles form the technical core of this paper, providing a systematic framework for studying the fluctuations of real roots in non-centered random polynomials.

 \subsection{Comparison principles}
Our comparison method relies on decomposing the random polynomial $P_n$, defined in \eqref{e.pn}, into its deterministic and stochastic components. Specifically, we isolate the deterministic part
 \[M_n(x):=\mathbb E [P_n(x)]\]
and define the centered fluctuation by
 \[R_n(x):=P_n(x)-M_n(x).\]
Building on earlier work by the first author \cite{Do21}, we develop novel comparison principles for random polynomials whose coefficients exhibit polynomial growth, commonly referred to as \emph{generalized Kac polynomials} \cite{DONV18}. A central object in our analysis is the normalized deterministic profile 
 \[
 m(x):=\frac{M_n(x)}{\sqrt{\Var[R_n(x)]}},
 \]
which establishes the natural scale for comparing the relative influence of the deterministic and random components of $P_n$.

Before stating the comparison principles, we introduce some notation and conventions that will be used throughout.

\begin{condition}[Polynomial growth] \label{c.A} 
We say that the coefficients $(\omega_j)_{j=0}^n$ of the polynomial $P_n$ in \eqref{e.pn} exhibit polynomial growth of order $\tau$ if, 
\[\omega_j=m_j+v_j\xi_j,\quad j\ge 0,\]
where $(m_j)_{j=0}^n$ and $(v_j)_{j=0}^n$ are deterministic, and $(\xi_j)_{j=0}^n$ are independent real-valued random variables with zero mean and unit variance, such that the following conditions hold for some constants $\varepsilon_0>0$, $C_0>0$, and $N_0>0$:
\begin{enumerate}[label={\rm(A\arabic*)}]
    \item \label{A1}  $\mathbb E[|\xi_j|^{2+\varepsilon_0}]<C_0$ for all $0\le j\le n$;
    \item \label{A2} $|m_j|, |v_j| \le C_0 (1+j)^{\tau}$ for all $0\le j\le n$; and 
    \item \label{A3} $|v_j| \ge\frac{1}{C_0} (1+j)^{\tau}$ for $N_0\le j\le n$.
\end{enumerate} 
\end{condition}
These assumptions ensure that for $j$ sufficiently large, the magnitude of $\omega_j$ is, with high probability, comparable to $j^{\tau}$. Under this decomposition, we may write
\[
M_n(x)=\sum_{j=0}^n m_j x^j,\quad R_n(x)=\sum_{j=0}^n v_j \xi_j x^j,
\]
yielding the representation
\[
P_n(x)=M_n(x)+R_n(x).
\]
This places $P_n$ in a natural signal-plus-noise framework, which is central to our analysis. Such models arise, for example, when a deterministic polynomial is perturbed by measurement or sampling noise, as in polynomial regression \cite{DS98}. In particular, if $M_n$ is a true deterministic polynomial regression function and the errors are centered Gaussian, then the least-squares estimator $P_n$ is a random polynomial satisfying $\mathbb{E}[P_n(x)]=M_n(x)$, and hence fits within the present framework.

\begin{notation} \label{not:asymptotic}
Unless otherwise specified, all asymptotic statements are understood in the limit $n \to \infty$. For real quantities $X$ and $Y$, we write $X=O(Y)$, $X\ll Y$, or equivalently  $Y\gg X$, if there exists a constant $C>0$, independent of $Y$, such that $|X| \le CY$. When needed, subscripts (e.g., $\ll_t$) indicate the dependence of $C$ on specified parameters. If both $X\ll Y$ and $X\gg Y$ hold, we write $X\asymp Y$ or $X=\Theta(Y)$. For real sequences $(\alpha_n)$ and $(\beta_n)$, we write $\alpha_n \sim  \beta_n$ if $\alpha_n / \beta_n \to 1$ as $n\to \infty$, and $\alpha_n = o(\beta_n)$ if  $\alpha_n / \beta_n \to 0$.

For a nonempty set $I\subset\mathbb R$, we adopt the notation 
\[-I:=\{-x: x\in I\}\quad\text{and}\quad I^{-1}:=\{x^{-1}: x\in I\}.\]

For a nonempty set $S\subset \mathbb C$ and a function $P: \mathbb C\to \mathbb C$, we denote by $N_P(S)$ the number of zeros of $P$ in $S$, counted with multiplicity.

Finally, let $P^*_n(x):=x^n P_n\left(1/x\right)$ be the reciprocal polynomial associated with $P_n(x)$. We define $M_n^*(x)=\mathbb E[P_n^*(x)]$, $R_n^*(x)=P_n^*(x)-M_n^*(x)$, and the normalized deterministic profile
\[
m^*(x)=\frac{M_n^*(x)}{\sqrt{\Var[R_n^*(x)]}}.
\]
\end{notation}

\begin{definition}[Enlargement of an interval]
Let $I \subset \mathbb{R}$ be an interval (open, closed, or half-open) with endpoints $a<b$. An interval $J\supset I$ is called an \emph{enlargement} of $I$ if it is obtained by extending $I$ beyond $a$ by $\Theta(|1-|a||+n^{-1})$ and beyond $b$ by $\Theta(|1-|b||+n^{-1})$.
\end{definition}

\begin{remark}
    Since our results impose conditions on $J$ and draw conclusions about $I$, the enlargement should be chosen as small as possible. In certain cases, the enlargement can be relaxed. Specifically, no left extension is required if $|1-|a||$ is bounded below by a positive absolute constant. An analogous statement holds for the right endpoint $b$. In particular, for $I\in \{(-\infty,0), (0,\infty)\}$, one may take $J=I$.
\end{remark}

\begin{definition}[Domination]
    Let $k\ge 0$ be an integer. We say that $M_n$ \emph{dominates} $R_n$ on a set $J\subset \mathbb R$ up to order $k$ with factor function $\phi$ if, for all $i=0,1,\dots, k$,
    \[
    |M_n^{(i)}(x)| \ge  \phi(1-|x|+1/n) \sqrt{\Var[R_n^{(i)}(x)]},\quad x\in J\cap [-1, 1],
    \]
    and 
    \[
    |{M_n^*}^{(i)}(1/x)| \ge  \phi(1-|1/x|+1/n) \sqrt{\Var[{R_n^*}^{(i)}(1/x)]},\quad x\in J\backslash [-1, 1].
    \]
If the above inequalities are reversed, we say that $M_n$ is \emph{dominated by} $R_n$ on $J$ up to order $k$ with factor function $\phi$. 
\end{definition}

For brevity, when $k=0$ we omit the phrase ``up to order $k$''. Roughly speaking, when $M_n$ dominates $R_n$, the deterministic polynomial is large compared to the size of the random fluctuations, and the same comparison holds for their derivatives up to the prescribed order.

\begin{remark}
The domination conditions admit an equivalent formulation in terms of the normalized profile $m$. Indeed, under Assumptions~\ref{A2} and~\ref{A3}, Lemma~\ref{l.estvar} shows that for any integer $i\ge 0$, there exist constants $C_i$ and $C_i'$ such that, uniformly for $x\in [-1,1]$,
    \[
    \Var[R_n^{(i)}(x)]\asymp \frac{C_i\Var[R_n(x)]}{(1-|x|+1/n)^{2i}}\quad \text{and}\quad 
    \frac{d^i}{dx^i}\Var[R_n(x)] \asymp\frac{C_i'\Var[R_n(x)]}{(1-|x|+1/n)^{i}}.
    \]
Consequently, the condition
    \[
    |M_n^{(i)}(x)| \gg  \phi(1-|x|+1/n) \sqrt{\Var[R_n^{(i)}(x)]},\quad x\in [-1, 1],
    \]
is equivalent to
    \[
    |m^{(i)}(x)| \gg  \frac{\phi(1-|x|+1/n)}{(1-|x|+1/n)^i},\quad x\in [-1, 1].
    \]
An analogous equivalence holds for $m^*$. Although the formulation in terms of $m$ is often less convenient to verify in applications, it yields a more concise statement.

\end{remark}

We now state the comparison principles, which are summarized in the following theorems.

\begin{theorem}[Variance comparison]\label{t.varcomp} Assume that the coefficients of $P_n$ in \eqref{e.pn} have polynomial growth of order $\tau>-1/2$. Let $I\subset \mathbb R$ be an interval (possibly depending on $n$), and let $J$ be an enlargement of $I$. Fix $d\in(0,1)$.
\begin{enumerate}
    \item There exists a constant $C>0$ such that, if $M_n$ dominates $R_n$ on $J$ with factor function $C|\log x|^{1/2}$, then as $n\to \infty$,
\begin{equation*} 
   \Var[N_{P_n}(I)]=O(\log^dn).
\end{equation*}
\item  Let $C,\theta>0$ be constants. If $M_n$ is dominated by $R_n$ on $J$ up to order 2 with factor function $C|x|^{\theta}$, then as $n\to \infty$,
    \[
   \Var[N_{P_n}(I)]=\Var[N_{R_n}(I)]+O(\log^d n).
   \]
\end{enumerate}
\end{theorem}
\begin{remark}
    To apply the comparison principles in the proofs of the main theorems, we will show in Section~\ref{s.proof.main} that $M_n$ dominates $R_n$ on $(0,\infty)$ with factor function $C|x|^{-1/2}$, whereas $M_n$ is dominated by $R_n$ on $(-\infty,0)$ with factor function $C|x|^{\theta}$ for some constants $C,\theta>0$. 
\end{remark}
\begin{theorem}[CLT comparison] \label{t.cltcomp} Assume that the coefficients of $P_n$ in \eqref{e.pn} have polynomial growth of order $\tau>-1/2$. Let $I\subset \mathbb R$ be an interval (possibly depending on $n$), and let $J$ be an enlargement of $I$. Fix constants $C, \theta,\epsilon>0$. Suppose that $M_n$ is dominated by $R_n$ on $J$ up to order $2$ with factor function $C|x|^{\theta}$, and that
\[
\Var[N_{R_n}(I)] \ge \epsilon \log n.
\]
Then $N_{P_n}(I)$ satisfies the CLT; namely, as $n\to\infty$,
    \[
    \frac{N_{P_n}(I)-\mathbb E[N_{P_n}(I)]}{\sqrt{\Var[N_{P_n}(I)]}} \xrightarrow{d} \mathcal N(0,1).
    \] 
\end{theorem}

\begin{remark}
Generalized Kac polynomials have also received considerable attention in mathematical physics, primarily in the centered setting. Schehr and Majumdar \cites{SM07, SM08} related persistence exponents for the diffusion equation with random initial data to the probability that a centered generalized Kac polynomial has no real zeros in a prescribed interval. Later, Dembo and Mukherjee \cite{DM15} established general continuity criteria for persistence exponents of centered Gaussian processes, with applications to random polynomials and zero-crossing problems for the heat equation; these results were subsequently extended to non-Gaussian settings by Ghosal and Mukherjee \cite{GM24}. More recently, \cites{KLN22, Lu23} connected the emergence of limit cycles in random vector fields to positive real zeros of centered random polynomials.

By contrast, the present work focuses exclusively on the non-centered regime. Our results suggest that analogous phenomena may persist beyond the centered framework, opening the door to corresponding applications. Non-centered generalized Kac polynomials also arise naturally in several applied settings, including biology \cite{Ham56} and communications and signal processing \cites{SG01, SG02}.
\end{remark}

\subsection{Outline of the paper} In Section \ref{s.local.estimate}, we recall key results from \cite{Do21} regarding the correlation functions of real roots. In Section \ref{s.multiple-root}, we establish probabilistic bounds for the event that the random polynomials have multiple real roots within small intervals. In Sections \ref{s.reduce.Gaussian}, \ref{s.universality}, and \ref{s.moment}, we reduce the proof of the comparison principles (Theorems \ref{t.varcomp} and \ref{t.cltcomp}) to the Gaussian case. The primary ingredients for this reduction are a universality estimate for the number of real roots (Section~\ref{s.universality}) and estimates for their moments (Section~\ref{s.moment}). In Section \ref{s.Kac-Rice}, we derive, via Kac--Rice formulas, expressions for the two-point correlation functions of real roots of non-centered Gaussian processes, which may be of independent interest. Section \ref{s.asym.est} establishes asymptotic estimates for the terms appearing in these expressions. Section \ref{s.Gaussian.case} applies these results to prove the Gaussian case of the comparison principles. Finally, in Section \ref{s.proof.main}, we combine the comparison principles to prove Theorems \ref{t.kac}, \ref{t.kac-der}, and \ref{t.hyperkac}.

\section{Local estimates for  real roots} \label{s.local.estimate}
In this section, we recall the relevant estimates from \cite{Do21}. 

We first recall a standard result in the area, which indicates that the majority of real roots of $P_n$ tend to concentrate around $\pm 1$.  
\begin{lemma} \label{l.im}  Fix $x_0\in (0,1)$ and let $A_{x_0}=\{z\in \mathbb R: ||z|-1| \ge x_0\}$. Then, for every integer $k\ge 1$, we have
    \[
    \mathbb E[N_{P_n}^k\left(A_{x_0}\right)] =O_k(1).
    \]
\end{lemma}
A proof of this estimate, using an argument of Ibragimov--Maslova, can be found in \cite{Do21}*{Lemma 2.2}.

The next result concerns mixed real-complex correlation measures for the roots of $P_n$, defined as follows. Let $k \ge 1$ and $k'\ge 0$ be integers, $Z$ be the multi-set of roots of $P_n$, and let $d\sigma$ be a measure on $\mathbb R^k \times \mathbb C_+^{k'}$, where $\mathbb C_+:=\mathbb C\backslash \mathbb R$. We say that $d\sigma$ is the $\left(k,k'\right)$-point correlation measure for $Z$ if the following two conditions hold: 
\begin{enumerate}
    \item The measure $d\sigma$ is symmetric under complex conjugations; that is, for any measurable set $A\subset \mathbb R^k \times \mathbb C_+^{k'}$, we have $d\sigma\left(A\right)=d\sigma\left(A'\right)$, where $A'$ is one of the sets $k'$ obtained from $A$ by taking the conjugate in one fixed coordinate.
    \item For any compactly supported continuous function $f:\mathbb R^k \times \mathbb C^{k'} \to \mathbb C$, we have
    \[
    \mathbb E [\sum_{ \alpha_i \in Z\cap \mathbb R} \sum_{\beta_j\in Z\cap \mathbb C_+} f\left(\alpha_1,\dots,\alpha_k,\beta_1,\dots,\beta_{k'}\right)] = \int_{\mathbb R^k \times \mathbb C_+^{k'}} f\left(w,z\right)d\sigma\left(w,z\right),
    \]
    where the summations on the left-hand side are taken over ordered tuples of different elements of $Z$. 
\end{enumerate}
If $d\sigma$ has a density with respect to the Lebesgue measure, such a density is classically called the $\left(k, k'\right)$-point correlation function (see \cite{TV15}), which will then be invariant under taking complex conjugation of any variable. 

In what follows, let $d\sigma$ and $d\sigma^*$ represent the $\left(k,k'\right)$-point correlation measures of the roots of $P_n$ and $P^*_n$, respectively. We write $d\sigma_G$ and $d\sigma^*_G$ for the corresponding measures associated with their Gaussian analogs, i.e., the random polynomials obtained by replacing the random coefficients with standard Gaussian variables. Given $\delta>0$, possibly dependent on $n$, define
    \[
    I(\delta) = \begin{cases}\{z\in \mathbb C: 1-2\delta \le |z| < 1-\delta\}&\text{if}\quad \delta \ge \frac 1 {10n},\\
    \{z\in \mathbb C: 1-\frac 1{2n}\le |z|< 1+\frac 1 {2n}\} &\text{if}\quad \delta <\frac 1 {10n}.
    \end{cases}
    \]
Let $I_{\mathbb R}(\delta) = I(\delta) \cap \mathbb R$ and $I_{\mathbb C_+}(\delta)=I(\delta)\cap \mathbb C_+$. Let $B(z, \delta):=\{w\in \mathbb C: |w-z|<\delta\}$ denote the open disk of radius $\delta$ centered at $z\in \mathbb C$.

\begin{theorem}[\cite{Do21}*{Theorem 3.1}] \label{t.corr} 
Let $k\ge1$ and $k'\ge0$ be integers, and let $0<c<\widetilde c<1$. Then there exist constants $C_1,\alpha_1>0$ such that the following holds for any $\frac1n\ll\delta\le C_1^{-1}$ and $(x,z)=\left(z_1,\dots,z_k,z_{k+1},\dots,z_{k+k'}\right) \in I_{\mathbb R}(\delta)^{k} \times I_{\mathbb C_+}(\delta)^{k'}$. Let $\phi_\delta$ be supported on $\left(-c\delta,c\delta\right)^k \times B\left(0,c\delta\right)^{k'}$ and suppose $\phi_\delta\in C^{3k'+2}(\mathbb R^{k+2k'})$ satisfies $\|\partial^\beta \phi_\delta\|_\infty \le \delta^{-|\beta|}$ up to order $|\beta|\le 3k'+2$. Let $J \subset I_{\mathbb R}(\delta)+\left(-\widetilde c \delta, \widetilde c\delta\right)$ be such that, for each $1\le j\le k+k'$, if $|\im z_j|\le \widetilde c\delta$, then $(\sgn(\re z_j)|z_j|-\widetilde c \delta,\sgn(\re z_j)|z_j|+\widetilde c \delta)\subset J$.
\begin{enumerate}
    \item Assume that  
    \[
    |M_n(x)|>C_1|\log\left(1- |x|+1/n\right)|^{1/2} \sqrt{\Var[R_n(x)]},\quad  x\in J,
    \]
    or 
    \[
    |M''_n(x)|\ll  \sqrt{\Var[R''_n(x)]}  \text{ uniformly on $J$}.
    \]
    Then
    \[
    \int_{\mathbb R^k\times \mathbb C^{k'}_+} \phi_\delta\left(y-x, w-z\right) [d\sigma (y,w)-d\sigma_G(y,w)] = O\left(\delta^{\alpha_1}\right).
    \]
\item Assume that  
    \[
    |M^*_n(x)|>C_1|\log\left(1- |x|+1/n\right)|^{1/2} \sqrt{\Var[R^*_n(x)]},\quad x\in J,
    \]
    or 
    \[
    |{M^*_n}''(x)|\ll  \sqrt{\Var[{R^*_n}''(x)]} \text{ uniformly on $J$}.
    \]
    Then
    \[
    \int_{\mathbb R^k\times \mathbb C^{k'}_+} \phi_\delta\left(y-x, w-z\right) [d\sigma^* (y,w)-d\sigma^*_G(y,w)] = O\left(\delta^{\alpha_1}\right).
    \]
\end{enumerate}
\end{theorem}

We also recall an estimate concerning the local integrability of $\log|P_n|$.

\begin{theorem}[\cite{Do21}*{Theorem 5.7}] \label{t.locallogint} Let $c',\widetilde c\in[0,1)$ satisfy $c'+\widetilde c<1$, and let $C_1>0$ be sufficiently large depending on $c'$ and $\widetilde c$. Then, for any $\varepsilon\in(0,\tfrac12)$, $\frac1n\ll\delta\le C_1^{-1}$, and $z \in I(\delta) + \left(-c'\delta, c'\delta\right)$, there exists an event $\mathcal T$ with probability $O(\delta^\varepsilon)$ such that the following estimate holds uniformly for $1\le p<\infty$:
    \[
    \pmb{1}_{\mathcal{T}^c} \int_{B\left(z, \widetilde{c}\delta\right)} |\log |P_n(w)||^p dw \le \left(Cp\right)^p \delta^2 |\log \delta|^{2p}.
    \]
An analogous estimate holds for $Q_n = \left(n + 1\right)^{-\tau} P^*_n$.
\end{theorem}

The next set of results provides local estimates for the number of real roots of $P_n$.

\begin{lemma}[\cite{Do21}*{Lemmas 9.1 and 9.2}] \label{l.nearR}  Let $\varepsilon>0$ be sufficiently small and $c\in[0,1)$. Then, for sufficiently large constants $C,C'>0$, the following holds for any $\frac1 n \ll  \delta \le \frac 1 C$, $\eta:=\delta^{1+\varepsilon}$, $\alpha<2$, and any $x\in  I_{\mathbb R}(\delta)+ (-c\delta, c\delta)$.  
\begin{enumerate}
    \item If either $|M_n| > C' |\log \delta|^{1/2}\sqrt{\Var[R_n]}$ or  $|M''_n| \ll  \sqrt{\Var[R''_n]}$ uniformly on $B\left(x,2\eta\right)$, then 
    \[
    \mathbb P\left(N_{P_n}\left(B\left(x, \eta\right)\right)\ge 2\right) \ll _{\varepsilon, \alpha}  \delta^{\alpha \varepsilon}.
    \]
    \item If either $|M_n^*|>C' |\log \delta|^{1/2}\sqrt{\Var[R^*_n]}$ or $|{M^*}''_n|\ll  \sqrt{\Var[{R_n^*}'']}$ uniformly on $B\left(x,2\eta\right)$, then 
    \[
    \mathbb P\left(N_{P^*_n}\left(B\left(x, \eta\right)\right)\ge 2\right) \ll _{\varepsilon, \alpha}   \delta^{\alpha \varepsilon}.
    \]
\end{enumerate}
\end{lemma}

\begin{theorem}[\cite{Do21}*{Theorem 6.1}] \label{t.localcount} 
Let $0\le c,c'<1$ satisfy $c+c'<1$. Then there exist constants $C_1,C_2,C_3>0$ such that, for any $\frac1n\ll \delta\le C_1^{-1}$, $|z|\in I(\delta)+(-c\delta,c\delta)$, $k\ge1$, $M>0$, and any event $E$,
    \[
    \mathbb E [N_{P_n}^k\left(B\left(z, c'\delta\right)\right)\pmb 1_{E}] \ll _{k,M} \delta^M + |\log \delta|^{C_2 k} \mathbb P\left(E\right).
    \]
Moreover, if $\delta\ge C_3\log n/n$, one may take $C_2=1$. The same estimate holds for $N_{P_n^*}$.
\end{theorem}

Below, we prove a variant of \cite{Do21}*{Lemma 6.3}.

\begin{lemma}\label{l.sublevelN} Let $0\le c,c'<1$ satisfy $c+c'<1$, and let $c_0\in (0,\tfrac12)$. Then, for sufficiently large $C>0$, the following holds. If $\frac 1 n\ll \delta <\frac 1 2$ and  $|z|\in I(\delta)+(-c\delta,c\delta)$, then
\begin{align*}
    \mathbb P(N_{P_n}(B(z,c'\delta))>C|\log \delta|) &\ll \delta^{c_0}.
\end{align*}
\end{lemma}
\begin{proof}
We will essentially follow the same argument as in the proof of \cite{Do21}*{Lemma 6.3}. Without loss of generality, we may assume $\delta<1/C_1$ for some large $C_1>1$. Using Jensen's formula, we have
\begin{equation}
    \label{eq.croots}
    N_{P_n}\left(B\left(z,c'\delta\right)\right) \ll \sup_{w\in B\left(z,c'\delta\right)}\log |P_n(w)|+\log \frac {1}{|P_n(z)|}.
\end{equation}
We now recall some results from \cite{Do21}.
\begin{lemma}[\cite{Do21}*{Lemma 5.8}]\label{l.L58} Let $0 \le c<1$.  For $\frac 1 n \ll \delta < \frac 1 5$ it holds for any $\varepsilon>0$ and $s\in \mathbb R$ that
\[
    \mathbb P\bigg(\sup_{|w|\in I(\delta)+\left(-c\delta,c\delta\right)}\log |P_n(w)|>s\bigg) \ll_\varepsilon e^{-2s}\delta^{-2\left(\rho+1+\varepsilon\right)}.
\]
\end{lemma}
\begin{lemma}[\cite{Do21}*{Corollary 4.3}]\label{l.C43} Let $0 \le c<1$.  Then there exists a constant $C_1>0$ such that the following holds for any $\frac 1 n\ll \delta \le \frac 1{C_1}$ and any $|z|\in I(\delta)+\left(-c\delta,c\delta\right)$: for any $0<\alpha_2<\frac 1 2$ there is a constant $C_2$ such that
\[
\mathbb P\left(\log |P_n(z)|\le -C_2 |\log \delta|\right)\ll \delta^{\alpha_2}.
\]
\end{lemma}
Using the triangle inequality, the desired estimate follows from \eqref{eq.croots}, Lemma \ref{l.L58} (with $s=\frac 1 2 C|\log \delta|$), and Lemma \ref{l.C43}.
\end{proof}

\section{Probability of multiple roots}\label{s.multiple-root}
In this section, we estimate the probability that $P_n$ has many real roots in small neighborhoods of $\pm 1$. These estimates will be used in subsequent sections to reduce the proof to the Gaussian setting.

\begin{theorem}\label{t.multi.roots} 
Assume that the coefficients of $P_n$ in \eqref{e.pn} have polynomial growth of order $\tau>-1/2$ and that $(\xi_j)_{j=0}^n$ are independent standard normal random variables. Fix constants $C, \theta, B>0$, and set $B_n=[1-B/n,1)$. Then there exist positive constants $C_1$ and $\beta$ such that the following holds. For every integer $k\ge 2$, every constant $\delta\in (0,1/\beta)$, and every interval $I=[a,b)\subset B_n$ satisfying
\[
\log \frac{1-a}{1-b}=\delta,
\]
if $M_n$ is dominated by $R_n$ on an enlargement $J\supset I$ with factor function  $Cx^{\theta}$, then
\begin{equation}\label{e.sublevel}
\mathbb P (N_{P_n}(I)\ge k) \le C_1 (\beta\delta)^{2k/3}.
\end{equation}
Analogous results hold for $N_{P_n}(-I)$ and $N_{P_n^*}(\pm I)$.
\end{theorem}

\subsection{Auxiliary results} 
To prove Theorem~\ref{t.multi.roots}, we first establish some preliminary results. 

\begin{lemma}\label{l.estvar}
If Condition~\ref{c.A} holds, then for $x_0\in(0,1)$ and uniformly over $x\in[x_0,1]$,
\begin{equation}
    \label{e.varRn}
    \Var[R_n(x)]\asymp \frac{1}{(1-x+1/n)^{2\tau+1}}
\end{equation}
and for $c=(2\tau+3)/x_0$ and any integer $k\ge 1$,
\begin{equation}
    \label{e.kder-var}
    \Var[R_n^{(k)}(x)]\le \frac{(c^k k!)^2}{(1-x+1/n)^{2k}}\Var[R_n(x)].
\end{equation}
\end{lemma}
\begin{proof}
    We first prove \eqref{e.varRn}. By Assumptions \ref{A2} and \ref{A3}, 
    \[
     \Var[R_n(x)]=\sum_{j=0}^n v_j^2x^{2j} \asymp\sum_{j=0}^n (1+j)^{2\tau}x^{2j}.
    \]
    
    Assume first that $1-1/n\le x\le 1$. Clearly, 
    \[
    \sum_{j=0}^n (1+j)^{2\tau}x^{2j} \le \sum_{j=0}^n(1+j)^{2\tau}\ll n^{2\tau+1} \asymp \frac{1}{(1-x+1/n)^{2\tau+1}}.
    \]
    To prove the lower bound, note that $x^{2j}\ge \left(1-1/n\right)^{2j}\gg 1$, so
    \[
    \sum_{j=0}^n (1+j)^{2\tau}x^{2j} \gg  \sum_{j=0}^n (1+j)^{2\tau}\gg  n^{2\tau+1} \asymp \frac{1}{(1-x+1/n)^{2\tau+1}}.
    \]
    
    Assume now that $x_0\le x\le 1-1/n$. By \cite{DN25}*{Lemma 3.1},
    \[ 
    \sum_{j=0}^n(1+j)^{2\tau}x^{2j}\sim  \frac{\Gamma(2\tau+1)}{(1-x^2)^{2\tau+1}}\asymp \frac{1}{(1-x+1/n)^{2\tau+1}},
    \]
    from which \eqref{e.varRn} follows.

   Similarly, for each $k=1,2, \ldots,$ it holds uniformly for $x\in [x_0,1]$ that 
    \[
    \Var[R_n^{(k)}(x)] \le \frac{C_0}{x_0^{2k}} \sum_{j=0}^n (1+j)^{2\tau+2k}x^{2j} \sim  \frac{C_0}{x_0^{2k}}\frac{\Gamma(2\tau+1+2k)}{(1-x+1/n)^{2\tau+1+2k}}.
    \]
    To prove \eqref{e.kder-var}, it suffices to show that 
    \[
    \Gamma(2\tau+1+2k) \le (x_0c)^{2k}(k!)^2 \Gamma(2\tau+1),
    \]
    or equivalently, 
    \[
    \prod_{j=0}^{2k-1}(2\tau+1+j)\le (x_0c)^{2k}(k!)^2.
    \]
    With $x_0c=2\tau+3$, this follows easily by induction on $k$.
\end{proof}
\begin{lemma}\label{l.derMn} Under the assumptions of Theorem \ref{t.multi.roots}, there exists a constant $c>0$ such that for all integer $k\ge 0$,
\begin{align*}
    \sup_{u\in [a,b]}|M_n^{(k)}(u)| 
    \le  \frac{c^kk!}{(1-b+1/n)^k} \sqrt{\Var[R_n(b)]}.
\end{align*}
\end{lemma}
\begin{proof}
We may assume $k\ge 1$, since the case $k=0$ is trivial. We follow the approach of Borwein and Erd\'elyi \cite{BE97}. 

Let $I=[a,b)\subset B_n$ be as in Theorem \ref{t.multi.roots} and let $J\supset I$ be an enlargement. Denote by $\ell$ the length of $J$ and by $w$ its midpoint. Consider the conformal map $f: \mathbb C\backslash \{0\}\to \mathbb C$ defined by
\[
f(z)=\frac{\ell}{4}\left(z+\frac 1z\right)+w.
\]
The map $f$ sends the unit circle $\{|z|=1\}$ onto the interval $J$. For $\eta>1$, let $\mathcal E_\eta$ denote the ellipse centered at $w$ with semiaxes
\[
\frac{\ell}{4}\left(\eta+\frac{1}{\eta}\right) \quad\text{and}\quad \frac{\ell}{4}\left(\eta-\frac{1}{\eta}\right).
\]
Then $f$ maps the circle $\{|z|=\eta\}$ onto $\mathcal E_\eta$.

Let $\eta_1=1+\epsilon$ and $\eta_2=3/2$, where $\epsilon>0$ is a sufficiently small constant. By Hadamard's three-circle theorem (see, e.g., \cite{BE97}),
\begin{align*}
    \sup_{z\in \mathcal E_{\eta_1}}\log|M_n(z)| &\le \frac{\log(\eta_2/\eta_1)}{\log \eta_2}\sup_{u\in J} \log|M_n(u)|+\frac{\log\eta_1}{\log \eta_2} \sup_{z\in \mathcal E_{\eta_2}}\log|M_n(z)|\\
    &\le \sup_{u\in J} \log|M_n(u)|+\frac{\epsilon}{\log(3/2)} \sup_{z\in \mathcal E_{\eta_2}}\log|M_n(z)|.
\end{align*}

Since $M_n$ is dominated by $R_n$ on $J$ with factor function $Cx^{\theta}$ and 
\[
\Var[R_n(u)]\asymp \frac{1}{(1-u+1/n)^{2\tau+1}} \asymp n^{2\tau+1},\quad u\in J,
\]
we have
\[
\sup_{u\in J} \log|M_n(u)| \le \left(\tau+1/2-\theta\right)\log n.
\]

By the definition of an enlargement, 
\[
\ell = \Theta(1-a+n^{-1}) + |I| + \Theta(1-b+n^{-1}). 
\]
In particular, there exist constants $c_1, c_2>0$, independent of $I$ and $\epsilon$, such that 
\[
c_1 n^{-1}+|I| \le \ell \le c_2 n^{-1}+|I|.
\]
Consequently, there exist constants $C_1, C_2>0$, also independent of $I$ and $\epsilon$, for which
\[
|z-u|\ge C_1 n^{-1},\quad \text{for all } z\in \mathcal E_{\eta_1},\; u\in [a,b],
\]
and 
\[
|z|\le 1+C_2n^{-1}, \quad \text{for all } z\in \mathcal E_{\eta_2}.
\]
By Condition~\ref{A2}, this implies that for some constant $C_\tau>0$ (independent of $\epsilon$),
\begin{align*}
    \sup_{z\in \mathcal E_{\eta_2}}\log|M_n(z)| &\le \log\bigg(C_0\sum_{j=0}^n (1+j)^\tau (1+C_2 /n)^j \bigg)\le C_\tau \log n.
\end{align*}

Choosing $\epsilon>0$ sufficiently small (depending on $C_\tau$ and $\theta$), we obtain
\begin{align*}
    \sup_{z\in \mathcal E_{\eta_1}}\log|M_n(z)| \le (\tau+1/2)\log n,
\end{align*}
and hence
\[
\sup_{z\in \mathcal E_{\eta_1}} |M_n(z)| \le n^{\tau+\frac 12} \asymp \sqrt{\Var[R_n(b)]}.
\]

Finally, by Cauchy's integral formula, for every $u\in [a,b]$,
\begin{align*}
    |M_n^{(k)}(u)| &\le \frac{k!}{2\pi}\bigg|\oint_{\mathcal E_{\eta_1}}\frac{M_n(z)}{(z-u)^{k+1}}dz\bigg|\\
    & \le \frac{k! }{2\pi} \frac{n^{k}}{C_1^k} n^{\tau+\frac 12}\\
    &\le \frac{c^kk!}{(1-b+1/n)^k} \sqrt{\Var[R_n(b)]},
\end{align*}
for some constant $c>0$. This completes the proof.
\end{proof}

\begin{lemma} 
    If $P_n$ has at least $k$ real roots in an interval $I=[a,b)$, then
\begin{equation}\label{e.rolle}
|P_n(b)| \le   \int_{x}^y \int_x^{y_1}\dots \int_{x}^{y_{k-1}}|P^{(k)}_n(y_k)|dy_k\dots dy_1=: Q_{n,k}^I.
\end{equation}
\end{lemma}
\begin{proof}
    Since $P_n$ has at least $k$ zeros in $I$, repeated applications of Rolle’s theorem yield $k$ points at which successive derivatives vanish. Iterating the fundamental theorem of calculus then gives the representation for $P_n(b)$, from which the bound follows.
\end{proof}
\subsection{Proof of Theorem~\ref{t.multi.roots}} We prove the result for $N_{P_n}(I)$, as the remaining cases follow by analogous arguments. 

For notational convenience, set $\epsilon:=(\delta \beta)^{2k/3}$ and 
\[
V(b):=\Var[P_n(b)]=\Var[R_n(b)].
\]
By the union bound,
\begin{align*}
    \mathbb P(N_{P_n}(I)\ge k) &\le \mathbb P(|P_n(b)| \le \epsilon  V(b)^{1/2})\\
    &\quad +\mathbb P(|P_n(b)| > \epsilon V(b)^{1/2},  N_{P_n}(I)\ge k).
\end{align*}
Since $P_n(b)$ is Gaussian with mean $M_n(b)$ and variance $V(b)$, a standard anti-concentration bound yields
\[
\mathbb P(|P_n(b)| \le \epsilon V(b)^{1/2}) \ll  \epsilon.
\]
Thus, to establish \eqref{e.sublevel}, it suffices to prove that
\[
\mathbb P(|P_n(b)| > \epsilon V(b)^{1/2}, N_{P_n}(I)\ge k) \ll \epsilon.
\]
By \eqref{e.rolle}, on the event $\{N_{P_n}(I)\ge k\}$ we have $|P_n(b)|\le Q_{k,n}^I$, and hence
\begin{align*}
\mathbb P(|P_n(b)| > \epsilon V(b)^{1/2}, N_{P_n}(I)\ge k) \le \mathbb P(Q_{k,n}^I \ge \epsilon V(b)^{1/2}).
\end{align*}
Using \eqref{e.rolle}, H\"older's inequality, and Fubini-Tonelli's theorem,  we have
\begin{align*}
    \mathbb E [|Q_{k,n}^I|^2] &\ll \frac{(b-a)^{2k}}{(k!)^2} \sup_{u\in [a,b]}\Big(|M^{(k)}_n(u)|^{2}+\mathbb E |R_n^{(k)}(u)|^{2}\Big).
\end{align*}
Applying Lemma~\ref{l.derMn}, \eqref{e.kder-var}, and Lemma~\ref{l.estvar}, there exists a constant $c>0$ such that 
\begin{align*}
    \sup_{u\in [a,b]}\Big(|M^{(k)}_n(u)|^{2}+\mathbb E |R_n^{(k)}(u)|^{2}\Big)
    \ll \frac{(c^k k!)^2}{(1-b+1/n)^{2k}} V(b).
\end{align*}
Since 
\[
\frac{b-a}{1-b+1/n}\le \frac{b-a}{1-b}=e^{\delta}-1 \le 2\delta,
\]
by choosing $\beta$ very large compared to $c$, we obtain
\begin{align*}
    \mathbb E [|Q_{k,n}^I|^2] &\ll  \frac{(b-a)^{2k}}{(k!)^2}\frac{(c^k k!)^2}{\left(1-b+1/n\right)^{2k}}V(b)\\
    &=\bigg(c\frac{b-a}{1-b+1/n}\bigg)^{2k}V(b)\\
    &\le  (\beta\delta)^{2k}V(b).
\end{align*}
Using Markov's inequality, we derive
\[
\mathbb P(Q_{k,n}^I \ge \epsilon V(b)^{1/2})\ll \frac{(\beta\delta)^{2k}}{\epsilon^2} = \epsilon.
\]
This completes the proof of \eqref{e.sublevel}.

\section{Comparison principles: Reduction to the Gaussian case} \label{s.reduce.Gaussian}
In this section, we reduce the comparison principles in Theorems \ref{t.varcomp} and \ref{t.cltcomp} to the Gaussian setting, where one can subsequently invoke the Kac--Rice formula together with the structural properties of Gaussian processes. The key input is a universality principle asserting that the asymptotic behavior of real roots is largely insensitive to the precise distribution of the coefficients.

We state the required universality results here and use them to reduce the problem to the Gaussian case, deferring their proofs to the next two sections. Since the statements are inherently local, we partition the real line into a core region $\mathcal{I}_n$, which captures the bulk of the real roots, and a complementary region whose contribution is negligible.

We now define $\mathcal{I}_n$. Let $0\le b_n<a_n<1$ be such that for any $A>0$, 
\begin{align*}
    a_n \ll_A \log^{-A}n.
\end{align*}
A convenient choice is $a_n = \exp(-\log^{d/4} n)$ for some fixed $d \in (0,1)$, though we retain flexibility for later applications. Set $I_n=[1-a_n, 1-b_n)$ and define
\begin{align*}
U_n =I_n\cup \left(-I_n\right),\quad 
U_n^* = I_n^{-1}\cup \left(-I_n^{-1}\right),\quad \text{and}\quad
\mathcal I_n=U_n\cup U_n^*. 
\end{align*}
We assume that the coefficients of $P_n$ exhibit polynomial growth of order $\tau > -1/2$ (see Condition~\ref{c.A}), and define the Gaussian analogue
\[
P_{n, G}(x)=\sum_{j=0}^n (m_j+v_j \widetilde{\xi}_j)x^j,
\]
where $(\widetilde{\xi}_j)_{j=0}^n$ are i.i.d.\ standard Gaussian random variables. 

The reduction proceeds by comparing the statistics of the number of real roots of $P_n$ and $P_{n,G}$ on the core region $\mathcal{I}_n$ via universality, and then showing that the contribution from $\mathbb{R}\backslash \mathcal{I}_n$ is negligible.
\subsection{Main ingredients for the reduction}
The first ingredient is a universality result that controls the contribution from the core region.

\begin{theorem}[Universality on the core region]
\label{t.general-u}  Let $I$ be an interval and let $J$ be an enlargement of $I$. Assume that either $M_n$ dominates $R_n$ on $J$ with factor function $C|\log x|^{1/2}$ for a sufficiently large constant $C$, or that $M_n$ is dominated by $R_n$ on $J$ up to order 2 with an arbitrary factor constant $C>0$. Then there exist positive constants $C_1$ and $c$ such that, for any function $\varphi: \mathbb R\to \mathbb R$ whose derivatives up to order 3 are bounded by $1$, and for all sufficiently large $n$, 
\begin{equation}
    \label{e.general-u}
    |\mathbb E[\varphi(N_{P_n}(I\cap \mathcal I_n))]-\mathbb E[\varphi(N_{P_{n, G}}(I\cap \mathcal I_n))]| \le C_1(a_n^{c}+n^{-c}).
\end{equation}
\end{theorem}
The second result provides a moment bound for the number of real roots outside the core region $\mathcal I_n$.
\begin{theorem}[Moment bound outside the core region]
\label{t.outside-In} 
Let $I$ be an interval and let $J$ be an enlargement of $I$. Assume that either $M_n$ dominates $R_n$ on $J$ with factor function $C|\log x|^{1/2}$ for a sufficiently large constant $C$, or that $M_n$ is dominated by $R_n$ on $J$ with factor function $C|x|^{\theta}$ for arbitrary constants $C, \theta>0$. Assume that $a_n = \exp(-\log^{d/4} n)$ for some fixed $d \in (0,1)$ and $b_n=B/n$ for some constant $B>0$. Then, for every integer $k\ge2$, there exists a constant $C_k>0$ such that 
    \[ 
    \mathbb E[N_{P_n}^k\left(I\backslash  \mathcal I_n\right)]\le C_k |\log a_n|^{2k}.
    \]
\end{theorem}

The proofs of these results will be given later. For the reduction, we will in fact use the following corollary of Theorem~\ref{t.general-u}, whose derivation will also be provided.

\begin{corollary}\label{c.kth-moment-u}
Let $k\ge 1$ be an integer. Under the assumptions of Theorem \ref{t.general-u}, there exist constants $C_1,c>0$ such that for all sufficiently large $n$, 
    \begin{equation}
    \label{e.kth-moment-u}
    |\mathbb E[N_{P_n}^k(I\cap \mathcal I_n)]-\mathbb E[N_{P_{n, G}}^k(I\cap \mathcal I_n)]| \le C_1(a_n^{c}+n^{-c}).
    \end{equation}
Consequently, for $k=2$,
    \begin{equation}
    \label{e.var-u}
    |\Var[N_{{P}_n}(I\cap \mathcal I_n)]-\Var[N_{P_{n,G}}(I\cap \mathcal I_n)]|\le C_1(a_n^{c}+n^{-c}).
    \end{equation}
\end{corollary}

\subsection{Reduction to the Gaussian case}
In this subsection, assuming that Theorem \ref{t.outside-In} and Corollary \ref{c.kth-moment-u} hold, we will reduce the proof of Theorems \ref{t.varcomp} and \ref{t.cltcomp} to the Gaussian setting.

Let $a_n=\exp(-\log^{d/4}n)$ for fixed $d\in (0,1)$ and let $b_n=B/n$ for some sufficiently large constant $B$. Applying Theorem \ref{t.outside-In} and Corollary \ref{c.kth-moment-u} with $k=2$, we obtain 
\begin{align*}
    \Var[N_{P_n}(I)] = \Var[N_{P_{n, G}}(I)] +O(\log^d n).
\end{align*}
Combined with the universality result for $\Var[N_{R_n}(I)]$ from \cite{NV22D}*{Corollary 2.2}, this reduces Theorem \ref{t.varcomp} to the Gaussian case.

For the reduction of Theorem \ref{t.cltcomp}, assume that $\Var[N_{R_n}(I)]\ge \epsilon \log n$. Then
\[
\Var[N_{P_{n,G}}(I\cap\mathcal I_n)]\gg \log n.
\]
This assumption, combined with estimate \eqref{e.var-u}, implies
\[
\Var[N_{P_n}(I\cap\mathcal I_n)]= \Var[N_{P_{n,G}}(I\cap\mathcal I_n)] (1+o(1)).
\]
Write 
\begin{align*}
    \frac{N_{P_n}(I)-\mathbb E[N_{P_n}(I)]}{\sqrt{\Var[N_{P_n}(I)]}} &=\frac{N_{P_n}(I\backslash \mathcal I_n)-\mathbb E[N_{P_n}(I\backslash \mathcal I_n)]}{\sqrt{\Var[N_{P_n}(I)]}}\\
    &+\frac{\sqrt{\Var[N_{P_n}(I\cap\mathcal I_n)]}}{\sqrt{\Var[N_{P_n}(I)]}}\frac{N_{P_n}(I\cap \mathcal I_n)-\mathbb E[N_{P_n}(I\cap \mathcal I_n)]}{\Var[N_{P_n}(I\cap\mathcal I_n)]}.
\end{align*}
Applying Theorem \ref{t.outside-In}, we deduce that 
\[
\Var[N_{P_n}\left(I\backslash \mathcal I_n\right)] = o(\log n)\quad \text{and}\quad \Var[N_{P_n}(I)]\gg \log n.
\]
This leads to
\[
\frac{N_{P_n}(I\backslash \mathcal I_n)-\mathbb E[N_{P_n}(I\backslash \mathcal I_n)]}{\sqrt{\Var[N_{P_n}(I)]}} \xrightarrow{d} 0 \quad \text{and}\quad \frac{\sqrt{\Var[N_{P_n}(I\cap\mathcal I_n)]}}{\sqrt{\Var[N_{P_n}(I)]}}=1+o(1).
\]
By Slutsky's theorem, $N_{P_n}(I)$ satisfies the CLT if and only if $N_{P_n}(I\cap\mathcal I_n)$ does. 

Fix $x\in \mathbb R$. If $x<0$, then 
    \[
    \mathbb P(N_{{P}_n}(I\cap \mathcal I_n)\le x)=0=\mathbb P(N_{P_{n,G}}(I\cap \mathcal I_n)\le x).
    \]
Now, suppose that $x\ge 0$. Since $N_{P_n}\left(I\cap \mathcal I_n\right)$ is always a non-negative integer, it follows that 
    \[
    \mathbb P(N_{P_n}\left(I\cap \mathcal I_n\right) \le x)=\mathbb P(N_{P_n}\left(I\cap \mathcal I_n\right) \le \lfloor x \rfloor) =\mathbb E[\varphi(N_{P_n}(I\cap \mathcal I_n))],
    \]
where $\varphi$ is any smooth function that takes values in $[0,1]$ and satisfies $\pmb 1_{[0,\lfloor x \rfloor]}\le \varphi \le \pmb 1_{[0,\lfloor x \rfloor+1/2]}$. By applying \eqref{e.general-u}, we can assert that
    \[
   |\mathbb E[\varphi(N_{P_n}(I\cap \mathcal I_n))]-\mathbb E[\varphi(N_{P_{n, G}}\left(I\cap \mathcal I_n\right))]|
   \le Ca_n^{c}, 
    \]
and so
    \[
   |\mathbb P(N_{{P}_n}(I\cap \mathcal I_n)\le x)-\mathbb P(N_{P_{n,G}}(I\cap \mathcal I_n)\le x)|\le Ca_n^{c}.
    \]
Together with Corollary \ref{c.kth-moment-u}, we deduce that 
    \begin{align*}
        &\frac{N_{P_{n}}(I\cap \mathcal I_n)-\mathbb E[N_{P_{n}}(I\cap \mathcal I_n)]}{\sqrt{\Var[N_{P_{n}}(I\cap \mathcal I_n)]}}\\
        &=\frac{N_{P_{n}}(I\cap \mathcal I_n)-\mathbb E[N_{P_{n,G}}(I\cap \mathcal I_n)]}{\sqrt{\Var[N_{P_{n,G}}(I\cap \mathcal I_n)]}}(1+o(1))+o(1)
    \end{align*}
and 
\begin{align*}
    &\mathbb P\bigg(\frac{N_{P_{n}}(I\cap \mathcal I_n)-\mathbb E[N_{P_{n,G}}(I\cap \mathcal I_n)]}{\sqrt{\Var[N_{P_{n,G}}(I\cap \mathcal I_n)]}}\le x\bigg)\\
    &=\mathbb P\bigg(\frac{N_{P_{n,G}}(I\cap \mathcal I_n)-\mathbb E[N_{P_{n,G}}(I\cap \mathcal I_n)]}{\sqrt{\Var[N_{P_{n,G}}(I\cap \mathcal I_n)]}}\le x\bigg)+o(1).
\end{align*}
Therefore, by Slutsky's theorem, if $N_{P_{n,G}}(I\cap \mathcal I_n)$ satisfies the CLT, then so does $N_{{P}_n}(I\cap \mathcal I_n)$.

\section{Universality for real roots, revisit}\label{s.universality}
In this section, we establish the universality results stated in Theorem~\ref{t.general-u} and Corollary~\ref{c.kth-moment-u}. Our approach relies on a replacement principle introduced by Tao and Vu~\cite{TV15} and further refined by Do, Nguyen, and Vu \cite{DONV18} and Nguyen and Vu \cites{NV22D, NV22A}. This principle allows one to compare local statistics, such as root densities and correlation functions, of zeros of random functions whose logarithmic magnitudes are close in distribution and satisfy suitable anti-concentration bounds.

As the method is inherently local and most effective when the expected number of real roots is $O(1)$, an additional argument is required to lift local estimates to global conclusions. For simplicity of exposition, we restrict to the case $I\subset[0,1]$; the general case follows by a straightforward modification.

To implement the local comparison, we partition $I_n=[1-a_n,1-b_n)$ into dyadic subintervals. Let $T\ll \log n$   be the smallest positive integer such that 
\[
\frac{a_n}{2^T} \le \max \left\{\frac 1n, b_n \right\}.
\]
Define
\begin{align*}
    \delta_j&=
    \begin{cases}
    2^{-j}a_n &\text{if}\quad 0\le j\le T-1,\\
    \max\left\{\frac 1n, b_n\right\} &\text{if}\quad j=T,
    \end{cases}\\
    J_{j} &=
    \begin{cases}
    [1-\delta_{j-1}, 1-\delta_j) &\text{if}\quad 1\le j\le T-1,\\
    [1-\delta_{j-1}, 1-b_n)&\text{if}\quad j=T. 
    \end{cases}
\end{align*}
Although $\delta_0>\dots>\delta_{T-1}$, it may happen that $\delta_N >\delta_{T-1}$ (e.g., if $a_n<1/n$, then $T=1$ and $\delta_1=\frac 1n >\delta_0=a_n$). Nevertheless, for all $1\le j\le T$,
\[
\delta_j \gg \frac{1}{n}\quad \text{and}\quad \delta_j \ge \frac{\delta_{j-1}}{2}.
\]
Thus, $\mathcal{P}:=\{J_j\}_{j=1}^T$ forms a partition of $I_n$ into intervals on which the expected number of real roots is uniformly bounded. For $j=1,...,T$, let
\[
N_{j}=N_{P_n}\left(I\cap J_{j}\right),
\]
and let $G_{j}$ denote the corresponding quantity for the Gaussian polynomial.
\subsection{Proof of Theorem~\ref{t.general-u}}

It suffices to prove the following more general theorem.

\begin{theorem}[General universality on the core region]\label{t.Psi-u}
There exist positive constants $C_1$ and $c$ such that for every function $F: \mathbb R^{T}\to \mathbb R$ with all partial derivatives up to order $3$ bounded by $1$,
    \[
    |\mathbb E[F(N_{1}, \dots, N_{T})]-\mathbb E[F(G_{1}, \dots, G_{T})]| \le C_1(a_n^{c}+n^{-c}).
    \]
\end{theorem}

\begin{proof}
Let $\varepsilon>0$  be sufficiently small. Let $\left(\zeta_\ell\right)_{\ell=1}^n$ be the (complex) roots of $P_n$. We will approximate $N_{j}$ by
    \[
    \varphi_{j}:=\sum_{\ell=1}^n\varphi_j\left(\zeta_\ell\right),
    \]
where, for each $1\le j\le T$, $\varphi_j$ is a suitable bump function supported on a small neighborhood of $J_j$ in the complex plane: 
    \begin{equation}\label{e.varphi-j}         \varphi_j(x+iy):=\psi_j(x)\varphi\bigg(\frac{y}{\delta_j^{1+\varepsilon}}\bigg),
    \end{equation}
\begin{itemize}
    \item $\varphi$ is a smooth function supported on $[-1,1]$ that equals $1$ at $0$, and 
    \item $\psi_j$ is smooth and supported on the neighborhood $J_j+[-\delta_j^{1+\varepsilon},\delta_j^{1+\varepsilon}]$ of $J_j$ on the real line,
    and equals to $1$ on $J_j$.
\end{itemize}
By a standard construction, we can ensure that
\[
\|\partial_x^\ell \partial_y^k\varphi_j(x+iy)\|_{\sup} = O\left(\delta_j^{-(k+\ell)(1+\varepsilon)}\right)
\]
for any $k,\ell\ge 0$ such that $k+\ell\le 3$.

We will prove the following lemmas.
\begin{lemma}\label{l.estimate-Nj}
We have 
\begin{align*}
    \mathbb E[F\left(N_{1}, \dots, N_{T}\right)]-\mathbb E[F\left(\varphi_{1}, \dots, \varphi_{T}\right)] \ll  a_n^{\varepsilon/4}.    
\end{align*}
\end{lemma}
\begin{lemma}\label{l.estimate-varphi} Let $\widetilde \varphi_j$ be the Gaussian analogue of $\varphi_j$.  It holds that
\begin{align*}
    \mathbb E[F\left(\varphi_{1}, \dots, \varphi_{T}\right)]-\mathbb E[F\left(\widetilde\varphi_{1}, \dots, \widetilde\varphi_{T}\right)]\ll  a_n^{\varepsilon}.
\end{align*}
\end{lemma}
Theorem~\ref{t.Psi-u} follows directly from Lemmas \ref{l.estimate-Nj} and \ref{l.estimate-varphi} via the triangle inequality.
\end{proof}
The proofs of Lemmas~\ref{l.estimate-Nj} and~\ref{l.estimate-varphi} are obtained by adapting the arguments of Nguyen and Vu~\cite{NV22D} to the non‑centered setting. For completeness, the details are deferred to Appendix \ref{a.l5.2} and Appendix \ref{a.l5.3}.

\subsection{Proof of Corollary \ref{c.kth-moment-u}}  

Fix $k\ge 1$ and recall that $T\asymp \log n$. Let $\varphi$ be a smooth function supported on the interval $[-\frac 1 2, \lfloor\log^4 n\rfloor+\frac 1 2]$, with $\varphi(x)=x^k$ for all $x\in [0, \lfloor\log^4 n\rfloor]$. This function $\varphi$ can be constructed so that all of its derivatives up to order 3 are bounded by $A=O(k^3 \log^{4k}n)$. 

Let $\mathcal D:=\{N_{P_n}(I\cap \mathcal I_n) \le \lfloor\log^4 n\rfloor\}$ and $\mathcal D_G:=\{N_{P_{n, G}}(I\cap \mathcal I_n) \le \lfloor\log^4 n\rfloor\}$.
Since $N_{P_n}\left(I\cap \mathcal I_n\right)$ is always an integer, we must have
    \[
    N_{P_n}\left(I\cap \mathcal I_n\right)^k\pmb 1_{\mathcal D} =\varphi\left(N_{P_n}\left(I\cap \mathcal I_n\right)\right).
    \]
A similar conclusion holds for $N_{P_{n,G}}\left(I\cap \mathcal I_n\right)^k\pmb 1_{\mathcal D_G}$. Therefore, by applying Theorem \ref{t.general-u} to the rescaled function $\left(1/A\right)\varphi$, for some $c>0$ it holds that
\begin{align*} 
    &|\mathbb E[N_{P_n}\left(I\cap \mathcal I_n\right)^k]-\mathbb E[N_{P_{n, G}}\left(I\cap \mathcal I_n\right)^k]|\\
    &\ll \mathbb E[N_{P_n}(I\cap \mathcal I_n)^k\pmb 1_{\mathcal D^c}] + \mathbb E[N_{P_{n, G}}\left(I\cap \mathcal I_n\right)^k\pmb 1_{\mathcal D_G^c}]  +  a_n^{c/2}+n^{-c/2} .
\end{align*}

Let $\mathcal D_{j}=\{N_{j}\le \log^2(1/{\delta_j})\}$.  Clearly,  $\cap_{j=1}^T \mathcal D_{ij}\subset \mathcal D$. By Lemma~\ref{l.sublevelN}, for some $c>0$,
    \[
    \mathbb P\left(\mathcal D^c\right) \le \sum_{j=1}^T\mathbb P\left(\mathcal D_{j}^c\right)\ll  \sum_{j=1}^T\delta_j^{c} \ll  a_n^{c}+n^{-c}.
    \]
Combining this with Theorem \ref{t.localcount} and H\"older's inequality,  
\begin{align*}
    \mathbb E[N_{P_n}\left(I\cap \mathcal I_n\right)^k\pmb 1_{\mathcal D^c}] &\ll   T^{k-1}\sum_{j=1}^T\mathbb E[N_{j}^k\pmb 1_{\mathcal D^c}]\\
    &\ll   T^{k-1}\sum_{j=1}^T(\delta_j^{c}+|\log(\delta_j)|^{C_2k} \mathbb P(\mathcal D^c))\\
    &\ll a_n^{c/2}+n^{-c/2}.
\end{align*}
Since the analogous estimate holds for $N_{P_{n,G}}$, \eqref{e.kth-moment-u} is verified.

The estimate \eqref{e.var-u} follows by the triangle inequality and the case $k=2$ of \eqref{e.kth-moment-u}.

 \section{Moment bounds for the number of real roots}\label{s.moment}
 In this section, we will discuss the proof of  Theorem~\ref{t.outside-In}.

By the triangle inequality, we may assume without loss of generality that $I\subset [0,1]$. We then cover $I\backslash \mathcal I_n$ by the intervals
\begin{align*}
    S_1 &=[0,1-x_0],\\
    S_2 &=(1-x_0, 1-a_n),\\
    S_3 &=[1-\frac{B}{n}, 1-\frac {a_n}n),\\
    S_4 &=[1-\frac {a_n}n, 1],
\end{align*}
where $x_0\in (0,1)$ and $B>0$ are constants. Thus, it suffices to prove the desired estimate for the real roots inside each $S_j$, for $j=1,\dots, 4$. 

For $S_1$, by Lemma~\ref{l.im}, we have
\begin{align*}
    \mathbb E[N_{P_n}^k\left(S_1\right)] =O_k\left(1\right).
\end{align*}

For $S_2=(1-x_0, 1-a_n)$, we cover $S_2$ using dyadic intervals 
\[I_j=(1- x_02^{1-j}, 1 - x_02^{-j}],\quad 1\le j \le 1+\log_2(x_0/a_n).
\]
By the triangle inequality and Theorem~\ref{t.localcount}, 
\begin{align*}
    (\mathbb E[N_{P_n}^k \left(S_2\right)] )^{1/k} &\le \sum_{1\le j\le 1+\log_2(x_0/a_n)}  [\mathbb E N_{P_n}^k\left(I_j\right)]^{1/k}\\
    &\ll \sum_{1\le j\le 1+\log_2(x_0/a_n)} |\log(2^j/x_0)|\\
    &\ll |\log a_n|^{2},
\end{align*}
which implies
\[
\mathbb E[N_{P_n}^k \left(S_2\right)] \ll_k |\log a_n|^{2k}.
\]

For $S_3$ and $S_4$, we first use Corollary~\ref{c.kth-moment-u} to reduce the estimates to the Gaussian setting. Applying Corollary~\ref{c.kth-moment-u}  for the intervals $[1-\frac{B}{n}, 1-\frac {a_n}{n})$ and $[1-\frac {a_n}n,1]$, we have
\begin{align*}
|\mathbb E[N^k_{P_n}(S_3) ] - \mathbb E [ N^k_{P_{n,G}}(S_3)] | &\ll  \left(\frac{a_n}n\right)^c+n^{-c}  \ll 1, \\
|\mathbb E[ N^k_{P_n}(S_4) ] - \mathbb E [ N^k_{P_{n,G}}(S_4)]| &\ll   b_n^c+n^{-c} \ll 1.
\end{align*}
Hence, it suffices to work in the Gaussian setting.

For $S_4$, the Kac-Rice formula (see Lemma \ref{l.mean.Gauss}) yields
\begin{equation*}
    \mathbb E[N_{P_{n,G}}(S_4)] = O(a_n).
\end{equation*}
Combining this with Theorem~\ref{t.localcount} and applying H\"older's inequality, we get
\begin{align*}
    \mathbb E[N_{P_{n,G}}^k(S_4)] &\le (\mathbb E [N_{P_{n,G}}(S_4)])^{1/2}(\mathbb E [N_{P_{n,G}}^{2k-1}(S_4)])^{1/2}\\
    &\ll_k a_n (\log n)^{k-\frac 12}  \\
    &\ll_k 1.
\end{align*}

We next consider $S_3$, distinguishing two cases. 

If $M_n$ dominates $R_n$ on $J$ with factor function $C|\log x|^{1/2}$ for a sufficiently large constant $C$, then the Kac-Rice formula (see Lemma \ref{l.mean.Gauss}) yields
\[
    \mathbb E[N_{P_{n,G}}(S_3)] = O(a_n).
\]
Arguing as above, it follows that
\[
\mathbb E[N_{P_{n,G}}^k(S_3)]=O_k(1).
\]

Assume instead that $M_n$ is dominated by $R_n$ on $J$ with factor function $C|x|^{\theta}$ for some constants $C,\theta>0$. We then apply Theorem~\ref{t.multi.roots}. To this end, cover $S_3$ by intervals $I=[a,b)$ satisfying
\[
\log\left(\frac {1-a}{1-b}\right)=\delta,
\]
where $\delta>0$ is chosen sufficiently small so that Theorem~\ref{t.multi.roots} applies. By that theorem, 
\begin{align*}
\mathbb E [N^k_{P_{n,G}}(I)]  &= \sum_{j\ge 1}j^k \mathbb P(N_{P_{n,G}}(I)=j) \\
&\le \mathbb P(N_{P_{n,G}}(I)=1) + C_1\sum_{j\ge 2} j^k (\beta\delta)^{2j/3}  \\
&\le \mathbb E \left[N_{P_{n,G}}(I)\right]  + O_k(1).
\end{align*}
Note that the number of covering intervals is at most $\frac{1}{\delta}\log\frac{B}{a_n}$, and since these intervals have finite overlap, it follows that
\begin{align*}
\mathbb E[N^k_{P_{n,G}}(S_3)]    &\ll_k  \log^{k-1}(1/a_n) \mathbb E [N_{P_{n,G}}(S_3)]  + O(\log^{k}(1/a_n)).
\end{align*}
Then, we use Lemma~\ref{l.mean.Gauss} and obtain
\begin{align*}
\mathbb E [N^k_{P_{n,G}}(S_3)]    &\ll_k \log^k(1/a_n).
\end{align*}
This completes the proof of Theorem~\ref{t.outside-In}.

\section{Kac--Rice formulas for non-centered Gaussian processes} \label{s.Kac-Rice}
In this section, we derive the Kac--Rice formulas for the one-point and two-point correlation functions of real roots of non-centered Gaussian processes, which may be of independent interest. 
\begin{condition}\label{c.GP} Let $\mathcal G=\{G(x)\}_{x\in I}$ be a Gaussian process on an interval $I\subset\mathbb R$, with mean 
    $m(x) := \mathbb E[G(x)]$
and covariance kernel
    \begin{equation}\label{e.rxy}
    r(x,y) := \Cov[G(x), G(y)]. 
    \end{equation}  
Assume that $\mathcal G$ is normalized, namely $r(x,x)=\Var[G(x)]=1$ for every $x\in I$. We further impose the following assumptions:
\begin{enumerate}[label={\rm(G\arabic*)}]
\item \label{G1} \textit{Regularity:} $\mathcal G$ has $C^1$-paths, with $m\in C^1(I)$ and $r\in C^2\left(I^2\right)$. 
\item \label{G2} \textit{Nondegeneracy}: For all distinct $x, y \in I$, the random vector 
\[(G(x), G(y), G'(x),G'(y))\]
is nondegenerate.
\item \label{G3} \textit{Gaussian structure}: For $x\in I$, $(G(x),G'(x))$ is a nondegenerate Gaussian vector.
\end{enumerate} 
\end{condition}

Here and subsequently, for nonnegative integers $j$ and $\ell$, and any function $f(x,y)$, we write $f_{j\ell}$ instead of $\frac{\partial^{j+\ell} f}{\partial x^j\partial y^\ell}$ whenever it is defined.

The key computation is summarized in the following result.
\begin{theorem}[Two-point correlation function]\label{t.rho2} 

Let Condition \ref{c.GP} hold. Define 
\begin{align}
    \label{e.mu1-s3}    \mu_1 &:= m'(x)+\frac{r(x,y)r_{10}(x,y)}{1-r^2(x,y)} m(x)-\frac{r_{10}(x,y)}{1-r^2(x,y)} m(y),\\
    \label{e.sigma1}    \sigma_1 &:=  \sqrt{r_{11}(x,x)-\frac{r_{10}^2(x,y)}{1-r^2(x,y)}},\\
    \label{e.mu2-s3}    \mu_2 &:= m'(y)+\frac{r(x,y)r_{01}(x,y)}{1-r^2(x,y)} m(y)-\frac{r_{01}(x,y)}{1-r^2(x,y)} m(x),\\
    \label{e.sigma2}    \sigma_2 &:=  \sqrt{r_{11}(y,y)-\frac{r_{01}^2(x,y)}{1-r^2(x,y)}},\\
    \label{e.delta}      \delta &:= \frac{1}{\sigma_1\sigma_2}\bigg(r_{11}(x,y)+\frac{r(x,y)r_{10}(x,y)r_{01}(x,y)}{1-r^2(x,y)}\bigg),
\end{align}
and let $\nu_1=\mu_1/\sigma_1$ and $\nu_2=\mu_2/\sigma_2$. Then, we obtain the following formula for the two-point correlation function $\rho_2(x,y)$ of the real roots of $G$ when $x\ne y$:
\begin{equation} \label{e.rho2}
    \rho_2(x,y)=\frac{E(x,y)}{\pi^2\sqrt{1-r^2(x,y)}}\left(\sum_{i=1}^5 \rho_{2,i}(x,y)\right), 
\end{equation}
where 
\begin{align*}
    E(x,y)&=\exp\left(-\frac{m^2(x)-2r(x,y) m(x)m(y)+m^2(y) }{2(1-r^2(x,y))}\right),\\
    \rho_{2,1}&= \sigma_1\sigma_2 (\sqrt{1-\delta^2}+\delta\arcsin\delta),\\
    \rho_{2,2}&= \mu_1\mu_2 \arcsin \delta,\\
    \rho_{2,3}&= \sigma_1\sigma_2\sqrt{1-\delta^2}\sum_{j=1}^2\int_0^{1}\frac{(1-t)\nu_j^2}{\exp(\nu_j^2t^2/(2-2\delta^2))} dt \\
    &\quad +\sqrt{\frac{\pi}{2}}\sigma_1\sigma_2|\delta| \sum_{j=1}^2\nu_j^2\int_0^{1}\frac{(1-t)|\nu_jt|}{\exp(\nu_j^2t^2/2)}\erf\bigg(\frac{|\delta \nu_j t|}{\sqrt{2(1-\delta^2)}}\bigg)dt,\\
    \rho_{2,4}&= \sqrt{\frac{\pi}{2}}\mu_1\mu_2\sum_{j=1}^2\int_0^{1}\frac{\left(t-1\right)\nu_j}{\exp(\nu_j^2t^2/2)}\erf\bigg(\frac{\delta \nu_jt}{\sqrt{2(1-\delta^2)}}\bigg)dt,
\end{align*} 
and 
\begin{align*}
    \rho_{2,5}&=\frac{|\mu_1\mu_2\nu_1\nu_2|}{\sqrt{1-\delta^2}}\\
    &\quad \times \int_0^{1}\int_0^{1}(1-t)(1-s)\exp\left(-\frac{\nu_1^2t^2-2\delta \nu_1\nu_2 ts+\nu_2^2s^2}{2(1-\delta^2)}\right) dtds.
\end{align*}
\end{theorem}

The analogous formula for the one-point correlation function $\rho_1$ is classical; see Leadbetter and Cryer \cite{LC65} (see also \cite{MBFMA97} for the complex setting). It will also emerge as a byproduct of the proof of Theorem~\ref{t.rho2}.
\begin{lemma}[One-point correlation function]  \label{l.rho1}
Let 
\begin{equation}
    \label{e.rho1}
    \rho_1(x)=\rho_{1,1}(x)+\rho_{1,2}(x),
\end{equation}
where  
\begin{equation}
    \label{e.rho11}
    \rho_{1,1}(x)= \frac{1}{\pi}\sqrt{r_{11}(x,x)}\exp\bigg(-\frac 12\bigg[m^2(x)+\frac{\left(m'(x)\right)^2}{r_{11}(x,x)}\bigg]\bigg),
\end{equation}
    and 
\begin{equation}
    \label{e.rho12}
    \rho_{1,2}(x)=\frac{1}{\sqrt{2\pi}} m'(x)\exp\bigg(-\frac 12 m^2(x)\bigg)\erf\bigg(\frac{m'(x)}{\sqrt{2r_{11}(x,x)}} \bigg).
\end{equation}
Then $\rho_1(x)$ is the density of real roots of $G(x)$, so that 
\begin{equation*}
    \mathbb E[N_G(I)]=\int_I \rho_1(x)dx.
\end{equation*}
\end{lemma}

The variance of the number of real roots of $G$ can then be calculated via the following standard formula (see, e.g., \cite{HKPV09}).

\begin{lemma} \label{l.var}
Let $\rho_1$ and $\rho_2$ be the one-point and two-point correlation functions for the real roots of $G$.  Then
\begin{equation*}
    \Var[N_G(I)]= \iint_{I^2}[\rho_2(x,y)-\rho_1(x)\rho_1(y)]dxdy +\int_I \rho_1(x)dx.
\end{equation*}
\end{lemma}
\subsection{One-point correlation: Proof of Lemma~\ref{l.rho1}}
 
First, it follows from hypothesis \ref{G1} that $G(x)$ has a quadratic mean derivative $G'(x)$, which is continuous on $I$ (see, for example, \cite{AW09}*{\S 1.4}). Therefore, letting \[\lambda(x,y)=\mathbb E[G(x)G(y)]=r(x,y)+m(x)m(y),\]
we see that $\lambda \in C^2\left(I^2\right)$ and $\lambda(x,x)=1+m^2(x)$. This gives
\begin{align*}
    \mathbb E[G(x)G'(x)] &=m(x)m'(x),\\
    \mathbb E[G'(x)G(y)] &=\lambda_{10}(x,y) = r_{10}(x,y)+m'(x)m(y),\\
    \mathbb E[G'(x)G'(y)]&=\lambda_{11}(x,y) = r_{11}(x,y)+m'(x)m'(y).
\end{align*}
From the above computation and hypothesis \ref{G3}, we deduce that for each $x\in I$, $G(x)$ and $G'(x)$ are independent.

Now, we recall the classical Kac--Rice formula \cite{Kac43},
    \[
    \mathbb E[N_G(I)]=\int_I \mathbb E[|G'(x)|\mid G(x)=0]p_{G(x)}\left(0\right)dx,
    \]
where $p_{G(x)}$ is the density of $G(x)$. Since $G(x)\sim \mathcal N\left(m(x),1\right)$, it follows that 
    \[
    p_{G(x)}\left(0\right)=\frac{1}{\sqrt{2\pi}}\exp\left(-\frac 12 m^2(x)\right).
    \]
Because $G'(x)$ and $G(x)$ are independent and $G'(x)\sim \mathcal N\left(m'(x),r_{11}(x,x)\right)$, 
we  have 
\begin{align*}
    &\mathbb E[|G'(x)|\mid G(x)=0] =\mathbb E[|G'(x)|]\\
    &=\frac{1}{\sqrt{2\pi r_{11}(x,x)}}\int_{\mathbb R}|t|\exp\bigg(-\frac{\left(t-m'(x)\right)^2}{2r_{11}(x,x)} \bigg)dt\\
    &=\frac{\sqrt{2r_{11}(x,x)}}{\sqrt{\pi}}\int_{\mathbb R}|t+\frac{m'(x)}{\sqrt{2r_{11}(x,x)}}|e^{-t^2}dt\\
    &=\frac{\sqrt{2r_{11}(x,x)}}{\sqrt{\pi}}\bigg[\exp\bigg(-\frac{\left(m'(x)\right)^2}{2r_{11}(x,x)}\bigg) +\frac{\sqrt{\pi}m'(x)}{\sqrt{2r_{11}(x,x)}}\erf\bigg(\frac{m'(x)}{\sqrt{2r_{11}(x,x)}} \bigg)\bigg],
\end{align*}
where $\erf$ is the error function defined as 
    \[
    \erf(x)=\frac{2}{\sqrt{\pi}}\int_0^x e^{-t^2}dt.
    \]
In particular, we recover the formula \eqref{e.rho1} for the one-point correlation stated in   Lemma~\ref{l.rho1}.  

\subsection{Two-point correlation: Proof of Theorem~\ref{t.rho2}} 
Let $\mathcal C:=\{G(x)=0, G(y)=0\}$. Given that $\mathcal{G}$ satisfies hypotheses \ref{G1} and \ref{G2}, the Rice formula \cite{AW09}*{Theorem 3.2} asserts that
\begin{align*}
    \mathbb E[N_G(I)(N_G(I)-1)]
    =\iint_{I^2}\mathbb E[|G'(x)G'(y)|\big|\mathcal C]p_{(G(x),G(y))}(0,0)dxdy,
\end{align*}
where $p_{(G(x),G(y))}$ is the joint density of $(G(x),G(y))$. Since the diagonal set $\{(x,x): x\in I\}$ has Lebesgue measure zero, in subsequent computations, we may safely assume that $(x, y)\in I^2$ with $x\ne y$. 

Thus, the two-point correlation function $\rho_2(x,y)$ of real roots of $G(x)$ is given by
    \[
    \rho_2(x,y) =\mathbb E[|G'(x)G'(y)|\big|\mathcal C]p_{(G(x),G(y))}(0,0).
    \]

Recall that $r(x,y)$ is defined by \eqref{e.rxy}, so 
\begin{align*}
    & p_{(G(x),G(y))}(0,0)\\
    &=\frac{1}{2\pi \sqrt{1-r^2(x,y)}}\exp \left(-\frac{m^2(x)-2r(x,y) m(x)m(y)+m^2(y)}{2(1-r^2(x,y))}\right).
\end{align*}
    
For $x\ne y$, under hypothesis \ref{G2}, it follows that $r(x,y)\ne 1$, ensuring the well-definedness of $p_{(G(x),G(y))}(0,0)$.

We will use regression to evaluate $\mathbb E[|G'(x)G'(y)|\big| \mathcal C]$. To perform a regression of $G'(x)$ on $\mathcal{C}$, we must carefully select $\theta_1=\theta_1(x,y)$ and $\eta_1=\eta_1(x,y)$ such that 
\[
\mathcal R_1:=G'(x)+\theta_1 G(x)+\eta_1 G(y)
\]
becomes independent of $G(x)$ and $G(y)$. This independence can be achieved by solving the following system of equations for $\left(\theta_1, \eta_1\right)$:
\begin{align*}
    \mathbb E[(G'(x)+\theta_1 G(x)+\eta_1 G(y))G(x)]&=\mathbb E[G'(x)+\theta_1 G(x)+\eta_1 G(y)]\mathbb E[G(x)],\\
    \mathbb E[(G'(x)+\theta_1 G(x)+\eta_1 G(y))G(y)]&=\mathbb E[G'(x)+\theta_1 G(x)+\eta_1 G(y)]\mathbb E[G(y)].
\end{align*}
These equations are equivalent to
\begin{align*}
    \theta_1+\eta_1r(x,y)&=0,\\
    r_{10}(x,y)+\theta_1r(x,y)+\eta_1&=0,
\end{align*}
which leads to 
 \begin{align*}
   \theta_1=\frac{r(x,y)r_{10}(x,y)}{1-r^2(x,y)}\quad\text{and}\quad \eta_1=-\frac{r_{10}(x,y)}{1-r^2(x,y)}.
 \end{align*}

Similarly, a regression  of $G'(y)$ on $\mathcal C$ is 
    \[
    \mathcal{R}_2=G'(y)+\theta_2 G(y)+\eta_2 G(x),
    \]
where 
 \begin{align*}
    \theta_2:=\frac{r(x,y)r_{01}(x,y)}{1-r^2(x,y)}\quad \text{and}\quad \eta_2:=-\frac{r_{01}(x,y)}{1-r^2(x,y)}.
 \end{align*}
But then 
    \[
    \mathbb E[|G'(x)G'(y)|\big| \mathcal C]=\mathbb E[|\mathcal{R}_1\mathcal{R}_2|].
    \]
Given $\mu_1$, $\sigma_1$, $\mu_2$, $\sigma_2$ from \eqref{e.mu1-s3}, \eqref{e.sigma1}, \eqref{e.mu2-s3}, \eqref{e.sigma2}, a computation shows that
 \begin{align*}
\mu_1 &= \mathbb E[\mathcal{R}_1],&
\sigma_1 &= \sqrt{\Var[\mathcal{R}_1]},\\
\mu_2 &= \mathbb E[\mathcal{R}_2],&
\sigma_2 &= \sqrt{\Var[\mathcal{R}_2]}.
 \end{align*}
Indeed, using $\Cov[G'(x),G(x)]=0$ and $\Var[G(x)]=\Var[G(y)]=1$, we get
\begin{align*}
    \Var[\mathcal{R}_1] &= \Var[G'(x)]+\theta_1^2  +\eta_1^2  + 2\theta_1\eta_1\Cov[G(x),G(y)]\\
    &\quad +2\eta_1 \Cov[G'(x),G(y)] \\
    &= r_{11}(x,x) - \frac{r_{10}^2(x,y)}{1-r^2(x,y)}\\
    &= \sigma_1^2. 
\end{align*}
The computation for $\sigma_2$ is entirely similar,  while the computations for $\mu_1$ and $\mu_2$ are even simpler.  

We obtain
\begin{align*}
    \mathbb E[|G'(x)G'(y)|\big| \mathcal C] =\mathbb E[|\mathcal{R}_1\mathcal{R}_2|]= \iint_{\mathbb R^2} |ts|p(t,s)dtds,
\end{align*}
where $p(t,s)$ is the joint density of the bi-variate Gaussian random variable $(\mathcal R_1,\mathcal R_2)$. Note that $p(t,s)$ depends on $x,y$ but for brevity, we omit this dependence in the notation. Below, we will calculate the covariance of $\mathcal R_1$ and $\mathcal R_2$.

Recalling 
\[
\Cov[G'(x),G(x)]=\Cov[G'(y),G(y)]=0 \text{ and }
\Var[G(x)]=\Var[G(y)]=1,
\] 
we have
\begin{align*}
    \Cov[\mathcal{R}_1,\mathcal{R}_2]
    &=\Cov[G'(x), G'(y)]+ \theta_2\Cov[G'(x), G(y)]\\
    &\quad + \theta_1 \Cov[G(x),G'(y)]+ \theta_1\theta_2 \Cov[G(x),G(y)]\\
    &\quad +\theta_1 \eta_2 + \eta_1 \theta_2 + \eta_1 \eta_2 \Cov[G(x),G(y)]\\
    &= r_{11}(x,y)+\theta_2r_{10}(x,y)+\theta_1r_{01}(x,y)+\theta_1\eta_2+\theta_2\eta_1\\
    &\quad +\left(\theta_1\theta_2+\eta_1\eta_2\right)r(x,y)\\
    &=r_{11}(x,y)+\frac{r(x,y)r_{10}(x,y)r_{01}(x,y)}{1-r^2(x,y)}.
\end{align*}

From \eqref{e.delta}, we see that
\begin{equation*}
    \delta=\frac{\Cov[\mathcal{R}_1,\mathcal{R}_2]}{\sigma_1\sigma_2}.
\end{equation*}
Let
    \[
    \mu :=\begin{pmatrix}
    \mu_1\\
    \mu_2
    \end{pmatrix},\quad \Sigma :=\begin{pmatrix}
    \sigma_1^2&\delta \sigma_1\sigma_2\\
    \delta \sigma_1\sigma_2&\sigma_2^2
    \end{pmatrix}.
    \]
Hence,
    \[
    \begin{pmatrix}
    \mathcal{R}_1\\
    \mathcal{R}_2 
    \end{pmatrix} \sim \mathcal N\left(\mu, \Sigma\right),
    \]
and the probability density function $p$ is given by
\begin{align*}
    p(t,s)=\frac{\exp \left(-\frac{1}{2(1-\delta^2)}[(\frac{t-\mu_1}{\sigma_1})^2-2\delta (\frac{t-\mu_1}{\sigma_1})(\frac{s-\mu_2}{\sigma_2})+(\frac{s-\mu_2}{\sigma_2})^2 ]\right)}{2\pi \sigma_1\sigma_2 \sqrt{1-\delta^2}}.
\end{align*}
The nondegeneracy hypothesis \ref{G2} implies that the joint distribution of $\mathcal R_1$ and $\mathcal R_2$ remains non-degenerate, thus ensuring the well-definedness of $p$ for all distinct $x, y\in I$.

We obtain
\begin{align*}
    \mathbb E[|\mathcal{R}_1\mathcal{R}_2|] 
    = \iint_{\mathbb R^2} |ts|p(t,s)dtds= \frac{\sigma_1\sigma_2}{2\pi \sqrt{1-\delta^2}} I_\delta\left(\nu_1,\nu_2\right),
 \end{align*}  
where $\nu_j:=\mu_j/\sigma_j$, and 
    \[
    I_\delta\left(\nu_1,\nu_2\right):=\iint_{\mathbb R^2} |\left(t+\nu_1\right)\left(s+\nu_2\right)|\exp \left(-\frac{t^2-2\delta ts+s^2}{2(1-\delta^2)}\right)dtds.
    \]
Write
    \[
    \exp \left(-\frac{t^2-2\delta ts+s^2}{2(1-\delta^2)}\right)=\exp \left(-\frac{t^2+s^2}{2(1-\delta^2)}\right)\sum_{j=0}^\infty \frac{1}{j!}\left(\frac{\delta ts}{1-\delta^2}\right)^j,
    \]
and hence 
\begin{align*}   
    I_\delta\left(\nu_1,\nu_2\right)
    &=\sum_{j=0}^\infty \frac{\delta^j}{j!(1-\delta^2)^j}\\
    &\quad \times \iint_{\mathbb R^2} |\left(t+\nu_1\right)\left(s+\nu_2\right)|\left(ts\right)^j\exp\left(-\frac{t^2+s^2}{2(1-\delta^2)}\right)dtds.
\end{align*}

For subsequent computations, we will utilize the following power series expansion:
\begin{equation*}
    \frac{\arcsin \delta}{\sqrt{1-\delta^2}} = \frac 1 2 \sum_{j=0}^{\infty} \frac {\left(j!\right)^2}{\left(2j+1\right)!} \left(2\delta\right)^{2j+1}.
\end{equation*}
This expansion can be readily obtained by noting that the left-hand side is an odd function that satisfies the functional equation $f'(\delta)= \frac {1+\delta f(\delta)}{1-\delta^2}$.
We will also employ the following series expansion for the error function:
    \[
    \erf\left(s\right) =\frac{2}{\sqrt \pi} \int_0^s e^{-t^2}dt = \frac 2{\sqrt \pi} e^{-s^2}  \sum_{j=1}^{\infty}  \frac {j!}{\left(2j\right)!} \left(2s\right)^{2j-1},  
    \]
which follows since $\erf\left(s\right)e^{s^2}$ is an odd function that satisfies the functional equation $f'\left(s\right)=\frac {2}{\sqrt{\pi}}+ 2sf\left(s\right)$.

Now, for each $j\ge 0$ and $\nu\in \mathbb R$, one has 
\begin{align*}
    &\int_{\mathbb R} |t+\nu|t^j\exp\left(\frac{-t^2}{2(1-\delta^2)}\right)dt\\
    &=\int_{-\nu}^\infty (t+\nu)t^j\exp\left(\frac{-t^2}{2(1-\delta^2)}\right)dx-\int_{-\infty}^{-\nu}(t+\nu)t^j\exp\left(\frac{-t^2}{2(1-\delta^2)}\right)dt \\
    &=  [2(1-\delta^2)]^{\frac{j}{2}+1} \Gamma\Big(\frac j2 +1\Big) \pmb 1_{j \text{ even}}+\nu [2(1-\delta^2)]^{\frac{j+1}{2}}\Gamma\Big(\frac{j+1}{2}\Big) \pmb 1_{j \text{ odd}}\\
    &\quad +2(-1)^j \nu^2 \int_0^{1}(1-t)(\nu t)^{j}\exp\left(\frac{-\nu^2 t^2}{2(1-\delta^2)}\right)dt,
\end{align*}
where $\Gamma$ is the gamma function defined as
    \[
    \Gamma(z)=\int_0^\infty t^{z-1}e^{-t}dt.
    \]
Hence,
\begin{align*}
   &\iint_{\mathbb R^2}|\left(t+\nu_1\right)\left(s+\nu_2\right)|\left(ts\right)^j\exp\left(-\frac{t^2+s^2}{2(1-\delta^2)}\right)dtds\\
   &= [2(1-\delta^2)]^{j+2}\Gamma^2\left(\frac j2 +1\right) \pmb 1_{j \text{ even}}+\nu_1\nu_2 [2(1-\delta^2)]^{j+1}\Gamma^2\Big(\frac{j+1}{2}\Big)\pmb 1_{j \text{ odd}}\\
   &\quad +2\bigg([2(1-\delta^2)]^{\frac{j}{2}+1}\Gamma\Big(\frac j2 +1 \Big)\sum_{i=1}^2  \int_0^{1}\frac{\nu_i^2(1-t)(\nu_i t)^{j}}{\exp(\nu_i^2 t^2/ (2-2\delta^2))}dt \bigg) \pmb 1_{j \text{ even}}\\
   &\quad +(-2)\bigg([2(1-\delta^2)]^{\frac{j+1}{2}}\Gamma\Big(\frac{j+1}{2}\Big) \sum_{i=1}^2\int_0^{1}\frac{\nu_i^2 \nu_{3-i}(1-t)\left(\nu_i t\right)^{j}}{\exp(\nu_i ^2 t^2/(2-2\delta^2))}dt\bigg)\pmb 1_{j \text{ odd}}\\
    &\quad +4 (\nu_1 \nu_2)^2\int_0^{1}\int_0^{1}(1-t)(1-s)(\nu_1\nu_2 ts)^{j}\exp\left(-\frac{\nu_1^2 t^2+\nu_2^2 s^2}{2(1-\delta^2)}\right)dsdt.
\end{align*}
This gives
    \[
    I_\delta\left(\nu_1,\nu_2\right)=\sum_{i=1}^5V_i,
    \]
where
\begin{align*}
    V_1 &=\sum_{j=0}^\infty \frac{\delta^{2j}}{\left(2j\right)!(1-\delta^2)^{2j}}  [2(1-\delta^2)]^{2j+2} \left(j!\right)^2 \\
     &= 4\sqrt{1-\delta^2}\Big(\sqrt{1-\delta^2}+\delta\arcsin\delta \Big),
\end{align*}
\begin{align*}
    V_2 &=\sum_{j=0}^\infty \frac{\delta^{2j+1}}{\left(2j+1\right)!(1-\delta^2)^{2j+1}}\nu_1\nu_2 [2(1-\delta^2)]^{2j+2}\left(j!\right)^2\\
    &=4\nu_1\nu_2\sqrt{1-\delta^2} \arcsin\delta,
\end{align*}
\begin{align*}
    V_3 &=\sum_{j=0}^\infty \frac{\delta^{2j}}{\left(2j\right)!(1-\delta^2)^{2j}}2[2(1-\delta^2)]^{j+1} j!  \sum_{i=1}^2 \int_0^{1}\frac{\nu_i^2(1-t)(\nu_i t)^{2j}}{\exp(\nu_i^2 t^2/(2-2\delta^2))}dt\\
    &=4(1-\delta^2)\sum_{i=1}^2 \int_0^{1} \frac{(1-t)\nu_j^2 }{\exp(\nu_j^2 t^2/(2-2\delta^2))}\sum_{j=0}^\infty \frac{j!}{(2j)!} \bigg(\frac{2\delta \nu_j t}{\sqrt{2(1-\delta^2)}}\bigg)^{2j}dt\\
    &=4(1-\delta^2)\sum_{i=1}^2\nu_i^2\int_0^{1}(1-t) \exp\left(\frac{-\nu_i^2t^2}{2(1-\delta^2)}\right) dt\\
    &\quad +2|\delta|\sqrt{2\pi(1-\delta^2)}\sum_{i=1}^2\nu_i^2\int_0^{1}\frac{(1-t)|\nu_it|}{\exp(\nu_i^2t^2/2)}\erf\bigg(\frac{|\delta \nu_i t|}{\sqrt{2(1-\delta^2)}}\bigg)dt,
\end{align*}
\begin{align*}
    V_4 &=\sum_{j=0}^\infty \frac{\delta^{2j+1}}{\left(2j+1\right)!(1-\delta^2)^{2j+1}}\\
    &\quad  \times  (-2)[2(1-\delta^2)]^{j+1}j!  \sum_{i=1}^2 \int_0^{1}\frac{ \nu_{3-i} \nu_i^2 (1-t)(\nu_i t)^{2j+1}}{\exp(\nu_i^2 t^2/(2-2\delta^2))}dt\\
    &= \sum_{i=1}^2 \int_0^{1}\frac{(-2) (1-t) \nu_{3-i} \nu_i^2}{\exp(\nu_i^2 t^2/(2-2\delta^2))} \sum_{j=0}^\infty \frac{j!}{(2j+1)!}\frac{2^{j+1}(\delta\nu_i t)^{2j+1}}{(1-\delta^2)^{j}}dt\\
    &=2\sqrt{2\pi(1-\delta^2)}\nu_1\nu_2 \sum_{i=1}^2 \int_0^{1}\frac{\left(t-1\right)\nu_i}{\exp(\nu_i^2t^2/2)}\erf\bigg(\frac{\delta \nu_it}{\sqrt{2(1-\delta^2)}}\bigg)dt,
\end{align*}
and 
\begin{align*}
    V_5 &= \sum_{j=0}^\infty \frac{\delta^j}{j!(1-\delta^2)^j}\\
    &\quad \times 4 \nu_1^2 \nu_2^2\int_0^{1}\int_0^{1}(1-t)(1-s)(\nu_1\nu_2 ts)^{j}\exp\left(-\frac{\nu_1^2 t^2+\nu_2^2 s^2}{2(1-\delta^2)}\right)dsdt\\
    &=4\nu_1^2\nu_2^2\int_0^{1}\int_0^{1}(1-t)(1-s)\exp\left(-\frac{\nu_1^2t^2-2\delta \nu_1\nu_2 ts+\nu_2^2s^2}{2(1-\delta^2)}\right) dtds.
\end{align*}
This completes the proof of Theorem~\ref{t.rho2}.
\begin{remark} The formula for $V_1$ can also be obtained by noting that $V_1 = I_\delta(0,0)$, corresponding to the zero-mean case, which can be computed via the geometry of bivariate normal distributions.  \end{remark}

\section{Preliminary analysis of the correlation functions} \label{s.asym.est}  The proofs of Theorems \ref{t.varcomp} and \ref{t.cltcomp} rely on Lemma \ref{l.var} together with asymptotic formulas for the correlation functions $\rho_1$ and $\rho_2$ in \eqref{e.rho1} and \eqref{e.rho2}. In this section, we analyze the terms appearing in these expressions. 

We first consider some notation. In what follows, let  $0\le b_n<a_n<1$, where 
\begin{align*}
    a_n \ll_A \log^{-A}n \quad \text{for any } A>0.
\end{align*}
Define $I_n =  [1 - a_n, 1 - b_n)$,  $U_n = I_n \cup \left(-I_n\right)$, $U_n^*=I_n^{-1}\cup \left(-I_n^{-1}\right)$, and let $S_n \in \{I_n, -I_n\}$. Throughout our analysis, we assume that $n$ is sufficiently large.

Assume that the coefficients of $P_n$ in \eqref{e.pn} have polynomial growth of order $\tau>-1/2$ and that $(\xi_j)_{j=0}^n$ are independent standard normal random variables. It is straightforward to verify that the normalized Gaussian process 
\[\mathcal G=\bigg\{\frac{P_n(x)}{\sqrt{\Var[P_n(x)]}}\bigg\}_{x\in \mathbb R}\]
satisfies Condition \ref{c.GP}. 

Define
    \[
    k(x):=\sum_{j=0}^n v_j^2x^j.
    \] 
Then
    \[
    \Var[P_n(x)]=k(x^2) \quad \text{and}\quad m(x)=\frac{M_n(x)}{\sqrt{k(x^2)}}.
    \]
Setting 
\[
\lambda(x,y):=\frac{\mathbb E[P_n(x)P_n(y)]}{\sqrt{\Var[P_n(x)]\Var[P_n(y)]}} 
     = m(x)m(y)+\frac{k(xy)}{\sqrt{k(x^2)k(y^2)}},
\]
we have 
\begin{align*}
    r(x,y)=\lambda(x,y)-m(x)m(y)=\frac{k(xy)}{\sqrt{k(x^2)k(y^2)}}.
\end{align*}
Let $\mu_1$, $\mu_2$, $\sigma_1$, $\sigma_2$, and $\delta$ be defined in \eqref{e.mu1-s3}, \eqref{e.mu2-s3}, \eqref{e.sigma1}, \eqref{e.sigma2}, and \eqref{e.delta} respectively. Furthermore, let $\nu_1=\mu_1/\sigma_1$ and $\nu_2=\mu_2/\sigma_2$. We stress that the dependence on $n$ has been suppressed from the notation for brevity.
\subsection{Estimates for the mean and variance functions} We begin with asymptotic estimates for $k$, $m$, and their derivatives. 

Applying the same argument as in the proof of Lemma \ref{l.estvar}, we find that for any constant $x_0\in (0,1)$ and $0\le i\le 4$, 
\begin{equation*}
    k^{(i)}\left(|x|\right) \asymp  \frac{1}{(1-|x|+1/n)^{2\tau+i+1}},\quad x_0\le |x|\le 1,
\end{equation*}
and
\begin{equation*}
     |M_n^{(i)}(x)|\ll  \frac{1}{(1-|x|+1/n)^{\tau+i+1}},\quad x_0\le |x|\le 1.
\end{equation*}
Thus, for each $i=0,1,2$, it holds that 
\begin{equation}
    \label{e.djm}
    |m^{(i)}(x)| \ll  \frac{1}{(1-|x|+1/n)^{i+\frac 1 2}},\quad x_0\le |x|\le 1.
\end{equation}
\subsection{The pseudo-hyperbolic distance on the unit disk} 
We now review the pseudo-hyperbolic metric on the unit disk, which plays a central role in our analysis.

Let $\varrho$ denote the pseudo-hyperbolic distance on $\mathbb D :=\{z\in \mathbb C: |z|<1\}$; that is,
\[\varrho(z,w)=\frac{|z-w|}{|1-\overline{w}z|},\quad (z,w)\in \mathbb D\times \mathbb D.\]
For each $0\le \varepsilon <\frac{1}{\sqrt 5}$, let $\mathbb U_\varepsilon =\{(x,y)\in U_n\times U_n: \varrho(x,y)\le \varepsilon\}$. For $(x,y)\in \mathbb R^2$ with $xy\ne 1$, we define
    \[
    \alpha:=\alpha(x,y):=\frac{(1-x^2)(1-y^2)}{(1-xy)^2}=1-\varrho^2(x,y).
    \]
We recall a property of the pseudo-hyperbolic distance $\varrho$ that will be convenient later.
\begin{lemma}[\cite{DN25}] \label{l.ez} Let $0\le \varepsilon<\frac 1{\sqrt 5}$ be a fixed constant. Suppose that for $x,y \in \left(-1,1\right)$ that have the same sign, we have $\varrho(x,y)  \le \varepsilon$. Then for every $z_1,z_2,z_3,z_4$ between $x$ and $y$ it holds that
    \[
    \frac{1}{1-z_1 z_2}=\frac{1+O(\varrho(x,y))}{1-z_3z_4}
    \]
and the implicit constant may depend on $\varepsilon$. Consequently,
    \[
    \varrho\left(z_1,z_2\right)\le \varrho(x,y)[1+O(\varrho(x,y))].
    \]
\end{lemma}
\subsection{Estimates via known results}
The following lemma was shown in \cite{DN25}*{Section 3.2}. 

\begin{lemma}[\cite{DN25}] \label{l.DN0} Let $b_n=1/\left(na_n\right)$. Suppose that $v_j=(1+j)^\tau$ for all $N_0\le j\le n$. It holds uniformly for $(x,y)\in S_n\times S_n$  that 
\begin{align*}
    r(x,y)&= \alpha^{\tau+1/2}(1+o(1)),\\
    1-r^2(x,y)&=\left(1-\alpha^{2\tau+1}\right)(1+o(1)),\\ 
    \sigma_1(x,y)&=\sqrt{r_{11}(x,x)}\sqrt{1-\frac{(2\tau+1)\left(1-\alpha\right)\alpha^{2\tau+1}}{1-\alpha^{2\tau+1}}}(1+o(1)),\\
    \sigma_2(x,y)&=\sqrt{r_{11}(y,y)}\sqrt{1-\frac{(2\tau+1)\left(1-\alpha\right)\alpha^{2\tau+1}}{1-\alpha^{2\tau+1}}}(1+o(1)),
\end{align*} 
and
    \[
    r_{11}(x,x)=\frac{2\tau+1}{(1-x^2)^2}(1+o(1)),\quad x\in U_n.
    \]
Moreover, if $(x,y)\in \left(-S_n\right)\times S_n$, then $r(x,y)=o(1)$, and hence
\begin{align*}
    \sigma_1(x,y)=\sqrt{r_{11}(x,x)}(1+o(1)) \quad\text{and}\quad
    \sigma_2(x,y)=\sqrt{r_{11}(y,y)}(1+o(1)).
\end{align*} 
\end{lemma}
The proof of Lemma \ref{l.DN0} relies on the mean value theorem and the following uniform estimates:
\[
\frac{d^i}{dx^i}\bigg[\sum_{j=0}^n (1+j)^{2\tau}x^j\bigg] = \frac{\Gamma\left(2\tau+i+1\right)}{(1-x)^{2\tau+i+1}}(1+o(1)),\quad x\in (1-a_n, 1-\frac{1}{na_n}]
\]
for $i=0,\dots, 4$  (see \cite{DN25}*{Corollary 3.2}).

We now extend the above asymptotics to the case $b_n=B/n$, where $B>0$ is a sufficiently large constant. Using the asymptotic for generalized binomial coefficients (see, e.g., \cite{E15}), we find that 
\[
\binom{j+1+2\tau}{j+1}=\frac{(1+j)^{2\tau}}{\Gamma(2\tau+1)}\left(1+O\left(1/j\right)\right).
\]
It was shown in \cite{DONV18}*{Lemma 10.4} that for each $i=0,1, 2$, the bound 
\[
\frac{d^i}{dx^i}\bigg[\sum_{j=0}^\infty \binom{j+1+2\tau}{j+1} x^j - \frac{1}{(1-x)^{2\tau+1}}\bigg]= O\bigg(\frac{\left(1+[n(1-x)]^{2\tau+i}\right)x^{n+1}}{(1-x)^{2\tau+i+1}}\bigg)
\]
holds uniformly over $x\in [x_0,1)$, for some constant $x_0\in (0,1)$. The same proof extends to the cases $i=3$ and $i=4$.

Consequently, for $i=0,\dots, 4$ and $x\in I_n=(x_0,1-B/n]$, where $B>0$ is a sufficiently large constant, we have 
\[
\left(1+[n(1-x)]^{2\tau+i}\right)x^{n+1} \ll e^{-B/2}. 
\]
Therefore, it holds uniformly over $x\in I_n$ that 
\begin{equation}
    \label{e.dik}
    \frac{d^i}{dx^i}\bigg[\sum_{j=0}^n (1+j)^{2\tau}x^j\bigg] = \frac{\Gamma\left(2\tau+i+1\right)}{(1-x)^{2\tau+i+1}}\big(1+O(e^{-B/2})\big).
\end{equation}
Using these estimates, the mean value theorem, and the approximations
\[
1-\alpha^{2\tau+1} \asymp 1-\alpha =\varrho^2(x,y)
\]
and
\[
\sqrt{1-\frac{(2\tau+1)\left(1-\alpha\right)\alpha^{2\tau+1}}{1-\alpha^{2\tau+1}}} \asymp \sqrt{1-\alpha} =\varrho(x,y),
\]
we establish the following lemma.
\begin{lemma} \label{l.DN} Let $b_n=B/n$, where $B>0$ is a sufficiently large constant. Assume $v_j=(1+j)^\tau$ for $j\ge 0$. It holds uniformly for $(x,y)\in S_n\times S_n$  that 
\begin{align*}
    r(x,y)&= \alpha^{\tau+1/2}(1+O(e^{-B/2})),\\
    1-r^2(x,y)&\asymp \varrho^2(x,y),\\ 
    \sigma_1(x,y)&\asymp \sqrt{r_{11}(x,x)}\varrho(x,y),\\
    \sigma_2(x,y)&\asymp \sqrt{r_{11}(y,y)}\varrho(x,y),
\end{align*} 
and 
    \[
    r_{11}(x,x)=\frac{2\tau+1}{(1-x^2)^2}(1+O(e^{-B/2})),\quad x\in U_n.
    \]
Furthermore, if $(x,y)\in \left(-S_n\right)\times S_n$, then $r(x,y)=o(1)$.
\end{lemma}
\subsection{Estimates for the covariance kernel and its partial derivatives} \label{s.8g} In the following, we derive asymptotic estimates for $r(x,y)$ and its partial derivatives in the absence of asymptotic structure for the deterministic coefficients $v_j$. From now on, we assume that $b_n=B/n$ for a sufficiently large constant $B$.

\begin{lemma}
It holds uniformly for $(x,y)\in U_n\times U_n$ that
\begin{equation}
    \label{e.1-rr}
    1-r^2(x,y)\asymp \varrho^2(x,y).
\end{equation}
\end{lemma}
\begin{proof} By elementary computation, for any $n\ge 0$ and any sequences $\left(x_j\right)_{j=0}^n$ and $\left(y_j\right)_{j=0}^n$, we have
\begin{equation}
    \label{e.L}
    \bigg(\sum_{j=0}^nx_j^2\bigg)\bigg(\sum_{j=0}^ny_j^2\bigg)-\bigg(\sum_{j=0}^nx_jy_j\bigg)^2=\frac 12\sum_{i,j=0}^n\left(x_iy_j-x_jy_i\right)^2.
\end{equation}
 Let 
 \begin{equation}
     \label{e.Axy}
     A(x,y):=k(x^2)k(y^2)-k^2(xy).
 \end{equation} 
 By applying \eqref{e.L} with $x_j:=v_jx^j$ and $y_j:=v_jy^j$, we find that
    \[
    A(x,y)=\frac 12\sum_{i,j=0}^nv_i^2v_j^2(x^iy^j-x^jy^i)^2.
    \]
This gives
\begin{align*}
    1-r^2(x,y)=\frac{A(x,y)}{k(x^2)k(y^2)}=\frac{\frac 12\sum_{i,j=0}^nv_i^2v_j^2(x^iy^j-x^jy^i)^2}{\sum_{i,j=0}^nv_i^2v_j^2x^{2i}y^{2j}}.
\end{align*}
Since all summands are non-negative,  it follows from assumptions \ref{A2} and \ref{A3} that
    \[
    1-r^2(x,y) \asymp \frac{\frac 12\sum_{i,j=0}^n(i+1)^{2\tau}(1+j)^{2\tau}(x^iy^j-x^jy^i)^2}{\sum_{i,j=0}^n(i+1)^{2\tau}(1+j)^{2\tau}x^{2i}y^{2j}},
    \]
which yields \eqref{e.1-rr} when combined with Lemma \ref{l.DN}. 

\end{proof}

\begin{lemma} It holds uniformly  that
\begin{equation}
    \label{e.r11}
    r_{11}(x,x) \asymp \frac{1}{(1-|x|)^2},\quad x\in U_n,
\end{equation}
and 
\begin{equation}
    \label{e.1r11}
    r_{11}(x,x) \ll \frac{1}{(1-|x|+1/n)^2},\quad 1-\frac{B}{n} \le |x| \le 1.
\end{equation}
\end{lemma}
\begin{proof} Write 
\begin{equation}
    \label{e.ax}a(x):=k(x^2)[k'(x^2)+x^2k''(x^2)]-(xk'(x^2))^2.
\end{equation} Using \eqref{e.L} with $x_j:=v_jx^j$ and $y_j:=jv_jx^{j-1}$, we obtain
    \[
    a(x)=\frac 12\sum_{i,j=0}^nv_i^2v_j^2x^{2(i+j-1)}(j-i)^2.
    \]
It follows from \ref{A2}, \ref{A3}, and Lemma \ref{l.DN} that
\begin{align*}
    r_{11}(x,x)=\frac{a(x)}{k^2(x^2)}
    &=\frac{\frac 12\sum_{i,j=0}^nv_i^2v_j^2x^{2(i+j-1)}(j-i)^2}{\sum_{i,j=0}^nv_i^2v_j^2x^{2(i+j)}}\\
    &\asymp \frac{\frac 12\sum_{i,j=0}^n(1+i)^{2\tau}(1+j)^{2\tau}x^{2(i+j-1)}(j-i)^2}{\sum_{i,j=0}^n(1+i)^{2\tau}(1+j)^{2\tau}x^{2(i+j)}}.
\end{align*}
Combining with Lemma \ref{l.DN}, we deduce \eqref{e.r11}. For $1-B/n \le |x| \le 1$, using $(j-i)^2\le n^2$, we obtain 
    \[
    r_{11}(x,x) \ll n^2\ll \frac{1}{(1-|x|+1/n)^2},
    \]
proving the lemma.
\end{proof}

\begin{lemma}\label{l.sig} It holds uniformly for $(x,y)\in U_n\times U_n$ that
\begin{equation}
    \label{e.sig1}
    \sigma_1\asymp \sqrt{r_{11}(x,x)} \varrho(x,y) \asymp \frac{\varrho(x,y)}{1-|x|},
\end{equation}
and 
\begin{equation}
    \label{e.sig2}
    \sigma_2\asymp \sqrt{r_{11}(y,y)} \varrho(x,y)\asymp \frac{\varrho(x,y)}{1-|y|}.
\end{equation}
\end{lemma}
\begin{proof}
    Note that
\begin{align*}
    \sigma_1^2&= \frac{(1-r^2(x,y))r_{11}(x,x)- r_{10}^2(x,y)}{1-r^2(x,y)}=\frac{ A(x,y)a(x)-B^2(x,y)}{(1-r^2(x,y))k^3(x^2)k(y^2)},
\end{align*}
where $A(x,y)$ and $a(x)$ are defined    using \eqref{e.Axy} and \eqref{e.ax}, and 
    \[B(x,y):=yk'(xy)k(x^2)-xk'(x^2)k(xy).\]
Let 
    \[
    a_{i,j}:=v_iv_j(x^iy^j-x^jy^i) \quad \text{and}\quad b_{i,j}:=v_iv_jx^{i+j-1}(j-i).
    \]
We have shown that
    \[
    A(x,y)=\frac 12 \sum_{i,j=0}^na_{i,j}^2 \quad\text{and}\quad a(x)= \frac 12 \sum_{i,j=0}^nb_{i,j}^2.
    \]
Moreover, 
\begin{align*}
    B(x,y)&=\sum_{i,j=0}^n v_i^2v_j^2 jx^{i+j-1}(x^iy^j-x^jy^i)\\
    &=\frac 12 \sum_{i,j=0}^n v_i^2v_j^2 (j-i)x^{i+j-1}(x^iy^j-x^jy^i)\\
    &=\frac 12\sum_{i,j=0}^n a_{i,j}b_{i,j}.
\end{align*}
Therefore, using \eqref{e.L} again, we obtain
\begin{align*}
    A(x,y)a(x)-B^2(x,y)&=\frac 14 \bigg[\sum_{i,j=0}^na_{i,j}^2 \sum_{i,j=0}^nb_{i,j}^2 -\bigg(\sum_{i,j=0}^n a_{i,j}b_{i,j}\bigg)^2\bigg]\\
    &=\frac{1}{8}\sum_{i,j,l, q=0}^n(a_{i,j}b_{l,q}-a_{l,q}b_{i,j})^2\\
    &=\frac{1}{8}\sum_{i,j,l, q=0}^nv_i^2v_j^2v_l^2v_q^2x^{2\left(i+j+l+q-1\right)}e_{i,j,l,q}^2\left(\frac yx\right),
\end{align*}
where 
    \[
    e_{i,j,l,q}(t):=(t^j-t^i)(q-l)-(t^q-t^l)(j-i).
    \]
By employing assumptions \ref{A2} and \ref{A3}, and noting that each term in the final sum is non-negative, we can replace $v_j$ with $(1+j)^\tau$, resulting in an expression comparable to $A(x,y)a(x)-B^2(x,y)$. A similar conclusion can be drawn for $(1-r^2(x,y))k^3(x^2)k(y^2)$. Thus, we can apply Lemma \ref{l.DN} to obtain \eqref{e.sig1}. 

In the same manner, we also have \eqref{e.sig2}.
\end{proof}

Our next goal is to estimate $r_{10}$ and $r_{01}$.
\begin{lemma} \label{l.dr} It holds uniformly for $(x,y)\in U_n\times U_n$ that 
\begin{equation}
    \label{e.r10}
    |r_{10}(x,y)| \ll  \frac{\varrho(x,y)}{1-x^2},
\end{equation}
and 
\begin{equation}
    \label{e.r01}
    |r_{01}(x,y)| \ll  \frac{\varrho(x,y)}{1-y^2}.
\end{equation}
\end{lemma}
\begin{proof} Using \eqref{e.sigma1}, \eqref{e.1-rr}, and \eqref{e.r11}, we obtain
\begin{align*}
    r_{10}^2(x,y)\le (1-r^2(x,y))r_{11}(x,x) 
    &\asymp \varrho^2(x,y)\frac{1}{(1-x^2)^2},
\end{align*}
which implies \eqref{e.r10}.
    
Similar arguments apply to \eqref{e.r01}. 
\end{proof}

\begin{lemma}\label{l.1-dr} Fix $0<\varepsilon < \frac{1}{\sqrt 5}$ and let $\mathbb U_\varepsilon=\{(x,y)\in U_n\times U_n: \varrho(x,y) \le \varepsilon\}$. It holds uniformly for $(x,y)\in \mathbb U_\varepsilon$ that 
\begin{equation}
    \label{e.1r10}
    1-(y-x)\frac{r_{10}(x,y)}{1-r^2(x,y)} \ll  \varrho(x,y)
\end{equation}
and 
\begin{equation}
    \label{e.1r01}
    1-(x-y)\frac{r_{01}(x,y)}{1-r^2(x,y)} \ll  \varrho(x,y).
\end{equation}
\end{lemma}
\begin{proof} By \eqref{e.1-rr}, inequality \eqref{e.1r10} will be proved once we prove that
\begin{equation}
    \label{e.3.7}
    1-r^2(x,y)-(y-x)r_{10}(x,y) \ll  \varrho^3(x,y), \quad (x,y)\in \mathbb U_\varepsilon.
\end{equation}
Fix $x\in U_n$ and let $\varphi(y)= 1-r^2(x,y)-(y-x)r_{10}(x,y)$, viewed as a function of $y\in U_n$. By the mean value theorem, there is some $z$ between $x$ and $y$,
    \[
    \varphi(y)=\varphi(x)+(y-x)\varphi'(x)+\frac{(y-x)^2}{2}\varphi''(x)+\frac{(y-x)^3}{6}\varphi'''(z).
    \]
Direct computation shows 
    \[
    \varphi(x)=0,\quad \varphi'(x)=0,\quad \varphi''(x)=0,
    \]
and 
    \[
    \varphi'''(z)= -6r_{01}(x,z)r_{02}(x,z)-2r(x,z)r_{03}(x,z)-3r_{12}(x,z)-(z-x)r_{13}(x,z).
    \]
Note that if $(x,y)\in \mathbb U_\varepsilon$, then $x$ and $y$ have the same sign. Thus, we can apply Lemma \ref{l.ez} to deduce that for each $0\le i\le 4$, 
    \[
    k^{(i)}(x^2) \asymp k^{(i)}\left(xz\right) \asymp k^{(i)}(z^2) \asymp k^{(i)}(y^2) \asymp k^{(i)}(xy)
    \]
for $(x,y)\in \mathbb U_\varepsilon$ and every $z$ between $x$ and $y$. But then 
\begin{align*}
    r_{01}(x,z)&\ll  \frac{1}{1-xy}, &r_{02}(x,z)&\ll  \frac{1}{(1-xy)^2},\\
    r_{03}(x,z)&\ll  \frac{1}{\left(1-xy\right)^3}, &r_{12}(x,z)&\ll  \frac{1}{\left(1-xy\right)^3},
\end{align*}
and 
    \[
    r_{13}(x,z) \ll  \frac{1}{\left(1-xy\right)^4}.
    \]
Therefore,
    \[
    \varphi'''(z) \ll  \frac{1}{\left(1-xy\right)^3}+\frac{|z-x|}{\left(1-xy\right)^4} \ll  \frac{1}{\left(1-xy\right)^3}.
    \]
Summarizing, we have 
    \[
    \varphi(y) =\frac{(y-x)^3}{6}\varphi'''(z) \ll  \frac{|y-x|^3}{\left(1-xy\right)^3} =\varrho^3(x,y),
    \]
proving \eqref{e.3.7}.

The inequality \eqref{e.1r01} can be proved in much the same way.
\end{proof}
\subsection{Estimates for the mean functions}
Recall that $\mu_1$ and $\mu_2$ are defined in \eqref{e.mu1-s3} and \eqref{e.mu2-s3}, respectively.
    
\begin{lemma}\label{l.mu}
It holds uniformly for $(x,y)\in U_n\times U_n$ that 
\begin{equation}
    \label{e.mu1}
    \mu_1 \ll  \frac{\varrho(x,y)}{1-|x|}\bigg(\frac{1}{\sqrt{1-|x|}}+\frac{1}{\sqrt{1-|y|}}\bigg)
\end{equation}
and 
\begin{equation}
    \label{e.mu2}
    \mu_2 \ll  \frac{\varrho(x,y)}{1-|y|}\bigg(\frac{1}{\sqrt{1-|x|}}+\frac{1}{\sqrt{1-|y|}}\bigg).
\end{equation}
\end{lemma}
\begin{proof}  Assume first that $(x,y)\in \mathbb U_\varepsilon$. Write 
\begin{align*}
    \mu_1 &=m'(x)+\frac{r(x,y)r_{10}(x,y)}{1-r^2(x,y)}m(x)-\frac{r_{10}(x,y)}{1-r^2(x,y)}m(y)\\
    &=  \left[1-\frac{(y-x) r_{10}(x,y)}{1-r^2(x,y)}\right]m'(x)-\frac{r_{10}(x,y)}{1+r(x,y)}m(x)\\
    &\qquad -\frac{r_{10}(x,y)}{1-r^2(x,y)}[m(y)-m(x)-(y-x)m'(x)].
\end{align*} 
Using \eqref{e.1r10} and \eqref{e.djm}, we have 
    \[
    \bigg(1-\frac{(y-x) r_{10}(x,y)}{1-r^2(x,y)}\bigg)m'(x) \ll  \varrho(x,y)\frac{1}{(1-|x|)^{3/2}}.
    \]
By \eqref{e.r10} and \eqref{e.djm},
    \[
    \frac{r_{10}(x,y)}{1+r(x,y)}m(x) \ll  \frac{\varrho(x,y)}{1-|x|}\frac{1}{\sqrt{1-|x|}}.
    \]
By the mean value theorem, \eqref{e.djm}, and Lemma \ref{l.ez}, 
\begin{align*}
    m(y)-m(x)-(y-x)m'(x)&=\frac{(y-x)^2}{2}m''(z)\\
    &\ll  \frac{(y-x)^2}{2}\frac{1}{(1-|z|)^{5/2}}\\
    &\ll \varrho^2(x,y)\bigg(\frac{1}{\sqrt{1-|x|}}+\frac{1}{\sqrt{1-|y|}}\bigg)
\end{align*}
for some $z$ between $x$ and $y$. Combining \eqref{e.1-rr} and \eqref{e.r10} yields
\begin{align*}
    &\frac{r_{10}(x,y)}{1-r^2(x,y)}\left[m(y)-m(x)-(y-x)m'(x)\right]\\
    &\ll  \frac{\frac{\varrho(x,y)}{1-x^2}}{\varrho^2(x,y)} \varrho^2(x,y)\bigg(\frac{1}{\sqrt{1-|x|}}+\frac{1}{\sqrt{1-|y|}}\bigg)\\
    &\ll  \frac{\varrho(x,y)}{1-|x|}\bigg(\frac{1}{\sqrt{1-|x|}}+\frac{1}{\sqrt{1-|y|}}\bigg).
\end{align*}
Therefore, we obtain \eqref{e.mu1} for $(x,y)\in \mathbb U_\varepsilon$.

If $(x,y)\in U_n\times U_n \backslash  \mathbb U_\varepsilon$, then $\varrho(x,y)\asymp 1$ and hence
   \[
   \frac{1}{1-r^2(x,y)} \ll  \frac{1}{\varrho^2(x,y)} \ll  1.
   \]
Combining with \eqref{e.djm} and \eqref{e.r10}, we find that 
\begin{align*}
    \mu_1 & \ll  |m'(x)|+|\frac{r(x,y)r_{10}(x,y)}{1-r^2(x,y)}m(x)|+|\frac{r_{10}(x,y)}{1-r^2(x,y)}m(y)|\\
    &\ll  \frac{1}{(1-|x|)^{3/2}}+\frac{1}{1-x^2}\frac{1}{\sqrt{1-|x|}}+\frac{1}{1-x^2}\frac{1}{\sqrt{1-|y|}}\\
    &\ll  \frac{\varrho(x,y)}{1-|x|}\bigg(\frac{1}{\sqrt{1-|x|}}+\frac{1}{\sqrt{1-|y|}}\bigg),
\end{align*} 
which gives \eqref{e.mu1}. 
    
Similar arguments apply to \eqref{e.mu2}.
\end{proof}

The subsequent lemma demonstrates that refining the bounds for the normalized mean $m(x)$ and its derivatives leads to enhanced estimates for $\mu_1$ and $\mu_2$.

\begin{lemma} \label{l.cmu} Fix constants $C,\theta>0$, and let $\phi(x)=Cx^{\theta}$ for $x\in [0,1]$. Assume that uniformly for $x\in U_n$,
\begin{equation}
    \label{e.cdm}
    |m^{(i)}(x)| \ll  \frac{\phi(1-|x|)}{(1-|x|)^i},\quad i=0,1,2.
\end{equation}
Then it holds uniformly for $(x,y)\in U_n\times U_n$ that 
\begin{equation}
    \label{e.cmu1}
    \mu_1 \ll  \frac{\varrho(x,y)}{1-|x|}[\phi(1-|x|)+\phi(1-|y|)] 
\end{equation}
and 
\begin{equation}
    \label{e.cmu2}
    \mu_2 \ll  \frac{\varrho(x,y)}{1-|y|}[\phi(1-|x|)+\phi(1-|y|)].
    \end{equation}
\end{lemma}
\begin{proof} Assume first that $(x,y)\in \mathbb U_\varepsilon$. By Lemma \ref{l.1-dr} and \eqref{e.cdm}, 
\begin{align*}
    \bigg(1-\frac{(y-x) r_{10}(x,y)}{1-r^2(x,y)}\bigg)m'(x)  \ll  \varrho(x,y) \frac{\phi(1-|x|)}{1-|x|} 
\end{align*}
and 
    \[
    \frac{r_{10}(x,y)}{1+r(x,y)}m(x) \ll  \frac{\varrho(x,y)}{1-|x|} \phi(1-|x|).
    \]
Using the mean value theorem, \eqref{e.cdm}, Lemma~\ref{l.ez}, and the monotonicity of $\phi$, we obtain, for some $z$ between $x$ and $y$,
\begin{align*}
    m(y)-m(x)-(y-x)m'(x)&\ll  \frac{(y-x)^2}{2}\frac{\phi(1-|z|)}{(1-|z|)^2}\\
    &\ll  \frac{(y-x)^2}{(1-xy)^2}\left[\phi(1-|x|)+\phi(1-|y|)\right]\\
    &= \varrho^2(x,y)\left[\phi(1-|x|)+\phi(1-|y|)\right].
\end{align*}
Together with \eqref{e.r10} and \eqref{e.1-rr}, we get
\begin{align*}
    &\frac{r_{10}(x,y)}{1-r^2(x,y)}[m(y)-m(x)-(y-x)m'(x)]\\
    &\ll  \frac{\varrho(x,y)}{1-|x|}[\phi(1-|x|)+\phi(1-|y|)].
\end{align*}
Combining the above estimates yields
    \[
    \mu_1 \ll  \frac{\varrho(x,y)}{1-|x|}[\phi(1-|x|)+\phi(1-|y|)],\quad (x,y)\in \mathbb U_\varepsilon.
    \]
    
We now consider $(x,y)\in U_n\times U_n\backslash  \mathbb U_\varepsilon$. Then 
   \[
   \varrho(x,y)\asymp 1 \quad \text{and}\quad \frac{1}{1-r^2(x,y)} \ll  \frac{1}{\varrho^2(x,y)} \ll  1.
   \]
Together with Lemma \ref{l.dr} and \eqref{e.cdm}, we conclude that
\begin{align*}
    \mu_1 & \ll  |m'(x)|+|r_{10}(x,y)||m(x)|+|r_{10}(x,y)||m(y)|\\
    &\ll  \frac{\phi(1-|x|)}{1-|x|}+\frac{\varrho(x,y)}{1-|x|}\phi(1-|x|)+\frac{\varrho(x,y)}{1-|x|}\phi(1-|y|)\\
    &\ll  \frac{\varrho(x,y)}{1-|x|}[\phi(1-|x|)+\phi(1-|y|)],
\end{align*} 
and \eqref{e.cmu1} is proved.
    
Similar arguments apply to \eqref{e.cmu2}.
\end{proof}

\section{Comparison principles: Proof for the Gaussian case} \label{s.Gaussian.case}

In this section, we prove the Gaussian case of Theorems \ref{t.varcomp} and \ref{t.cltcomp}. The arguments rely on the Kac--Rice formulas from Section \ref{s.Kac-Rice} and the asymptotic estimates established in Section \ref{s.asym.est}.
 
Throughout, we assume that the coefficients of $P_n$ in \eqref{e.pn} have polynomial growth of order $\tau>-1/2$ and that $(\xi_j)_{j=0}^n$ are independent standard normal random variables. We further assume that $0\le b_n<a_n<1$, where $a_n \ll_A (\log n)^{-A}$ for any $A>0$, and $b_n=B/n$ for a sufficiently large constant $B$. Here, $B$ may depend on $\tau$ and $C_0$, the constants appearing in Condition~\ref{c.A}. Set $I_n=[1-a_n, 1-b_n)$ and define
\begin{align*}
U_n =I_n\cup \left(-I_n\right),\quad 
U_n^* = I_n^{-1}\cup \left(-I_n^{-1}\right),\quad \text{and}\quad
\mathcal I_n=U_n\cup U_n^*. 
\end{align*}
\subsection{Expected number of real roots}
We first establish the Gaussian results for the expected number of real roots, which form the foundation for the proofs of Theorem \ref{t.outside-In} and Lemma \ref{l.estimate-Nj} in Section \ref{s.moment} and Appendix \ref{a.l5.2}, respectively.
\begin{lemma} \label{l.mean.Gauss}
    Let $I \subset [1-a_n, 1)$.  
\begin{enumerate}
    \item There exists a constant $C>0$ such that, if $M_n$ dominates $R_n$ on $I$ with factor function $C|\log x|^{1/2}$, then 
        \[
        \mathbb E[N_{P_n}(I)] =O(a_n).
        \]
    \item Fix $C>0$. If $M_n$ is dominated by $R_n$ up to order 1 on $I$ with a constant factor $C$, then the one-point correlation function $\rho_1$ of the real roots of $P_n$ satisfies
    \begin{equation}
        \label{e.rho.est}
        \rho_1(x) \ll \frac{1}{1-x+1/n} \quad\text{uniformly on } I.
    \end{equation}
    In particular, for any constant $B>0$, 
    \begin{align*}
        \mathbb E[N_{P_n}\left( [1-a_n/n,1]\cap I\right)] = O(a_n),
    \end{align*}
    and
    \begin{align*}
       \mathbb E[N_{P_n} \left([1-B/n, 1-a_n/n)\cap I\right)]=O(1).
    \end{align*}
\end{enumerate}
    Analogous bounds hold on $(-1,-1+a_n]$ and for the reciprocal polynomial $P_n^*$.
\end{lemma}
\begin{proof}
Recall from Lemma \ref{l.rho1} that 
    \[
     \rho_1(x) =\rho_{1,1}(x) +\rho_{1,2}(x),
    \]
    where $\rho_{1,1}$ and $\rho_{1,2}$ are defined in \eqref{e.rho11} and \eqref{e.rho12}, respectively.

    Assume first that $M_n$ dominates $R_n$ on $I$ with factor function $C|\log x|^{1/2}$. It follows from \eqref{e.rho11} and \eqref{e.1r11} that uniformly for $x\in I$, 
    \[
    \rho_{1,1}(x) \le \sqrt{r_{11}(x,x)}\exp\left(-\frac 12 m^2(x)\right) \ll (1-x+1/n)^{-1+C^2/2}.
    \]
    Next, using \eqref{e.rho12} and \eqref{e.djm}, we see that uniformly for $x\in I$, 
    \[
        \rho_{1,2}(x) \le  |m'(x)|\exp\left(-\frac 12 m^2(x)\right)\ll (1-x+1/n)^{(-3+C^2)/2}.
    \]
    Therefore, for $C> \sqrt 3$, we have 
    \[
    \mathbb E[N_{P_n}(I)]= \int_I \rho_1(x)dx \ll \int_{1-a_n}^1 (1-x+1/n)^{(-3+C^2)/2}dx \ll a_n.
    \]

    Assume now that $M_n$ is dominated by $R_n$ up to order 1 on $I$ with a constant factor $C$, so that
    \[
    |m(x)|\le C \quad \text{and} \quad |m'(x)| \le \frac{C}{1-x+1/n},\quad x\in I.
    \]
    Then, uniformly for $x\in I$, 
    \[
    \rho_1(x)=\rho_{1,1}(x)+ \rho_{1,2}(x) \le \sqrt{r_{11}(x,x)}+|m'(x)| \ll \frac{1}{1-x+1/n}, 
    \]
  which implies \eqref{e.rho.est}.

    Consequently, 
    \begin{align*}
     \mathbb E[N_{P_n}\left( [1-a_n/n,1]\cap I\right)]  &\ll  \int_{1-\frac{a_n}{n}}^1 \frac{dx}{1-x+1/n} = O(a_n),\\
      \mathbb E[N_{P_n} \left([1-B/n, 1-a_n/n)\cap I\right)] &\ll  \int_{1-\frac{B}{n}}^{1-\frac{a_n}{n}} \frac{dx}{1-x+1/n} = O(1),
    \end{align*}
    completing the proof.
\end{proof}

By Theorem \ref{t.outside-In}, we only need to focus on the real roots of $P_n$ inside $\mathcal I_n= U_n\cup U_n^*$, where $U_n=I_n\cup \left(-I_n\right)$ and $U_n^* = I_n^{-1}\cup \left(-I_n^{-1}\right)$.

\subsection{Fluctuations of real roots in the core region, part (1)} We now discuss the first part of Theorem \ref{t.varcomp}. We show that when $M_n$ dominates $R_n$, the polynomial $P_n$ has few real roots.
\begin{theorem} \label{t.Cvar} Let $I\subset \mathbb R$ be an interval. There exists a constant $C>0$ such that if
\begin{equation}
    \label{e.cm}
    |m(x)|\ge C|\log(1-|x|)|^{1/2},\quad x\in I\cap U_n,
\end{equation}
then 
\begin{equation}
    \label{e.Cvar}
    \Var[N_{P_n}(I\cap U_n)] =o(1).
\end{equation}
\end{theorem}
\begin{proof}
By Lemma \ref{l.var}, we have
    \[
    \Var[N_{P_n}(I\cap U_n)]=\iint_{(I\cap U_n)^2}[\rho_2(x,y)-\rho_1(x)\rho_1(y)]dxdy+\int_{I\cap U_n}\rho_1(x)dx,
    \]
where $\rho_1$ and $\rho_2$ are the one-point and two-point correlation functions defined in \eqref{e.rho1} and \eqref{e.rho2}, respectively. 

We first show that 
    \[
    \int_{I\cap U_n}\rho_1(x)dx=o(1),
    \]
and hence,
    \[
    \iint_{(I\cap U_n)^2}\rho_1(x)\rho_1(y)dxdy = o(1).
    \]
Recall that $\rho_1(x)=\rho_{1,1}(x)+\rho_{1,2}(x)$, where $\rho_{1,1}$ and $\rho_{1,2}$ are given in \eqref{e.rho11} and \eqref{e.rho12}, respectively. Together with \eqref{e.djm}, \eqref{e.r11}, and \eqref{e.cm}, we see that 
\begin{align*}
    |\rho_{1,1}(x)| \le \frac{1}{\pi}\sqrt{r_{11}(x,x)}\exp\Big(-\frac 12 m^2(x)\Big)
    \ll  \frac{1}{1-|x|} (1-|x|)^{C^2/2},
\end{align*}
and 
\begin{align*}
    |\rho_{1,2}(x)| \ll  |m'(x)|\exp\Big(-\frac 12 m^2(x)\Big) \ll  \frac{1}{(1-|x|)^{3/2}}(1-|x|)^{C^2/2}.
\end{align*}
Thus, for $C>1$, we have 
\begin{align*}
    \int_{I\cap U_n}\rho_1(x)dx &=\int_{I\cap U_n}[\rho_{1,1}(x)+\rho_{1,2}(x)]dx
    \ll  \int_{I\cap U_n} (1-|x|)^{\frac{C^2-3}{2}}dx=o(1).
\end{align*}

The proof is completed by showing that 
    \[
    \iint_{(I\cap U_n)^2}\rho_2(x,y)dxdy =o(1).
    \]
It follows from Theorem \ref{t.rho2} that 
   \[
   \rho_2(x,y)=\frac{E(x,y)}{\pi^2\sqrt{1-r^2(x,y)}}\sum_{i=1}^5 \rho_{2,i}(x,y).
   \]

To estimate $E(x,y)$, we utilize the assumption \eqref{e.cm} and find that
\begin{align*}
    E(x,y)&\le \exp\left(-\frac{|r(x,y)|(m(x)-m(y))^2}{2(1-r^2(x,y))}-\frac{m^2(x)+m^2(y)}{2(1+|r(x,y)|)}\right)\\
    & \le \exp\left(-\frac{m^2(x)+m^2(y)}{4}\right)\\
    & \le (1-|x|)^{C^2/4}(1-|y|)^{C^2/4}.
\end{align*}

We now estimate $\rho_{2,i}$, for $1\le i\le 5$. First, by Lemma \ref{l.sig},
    \[
    |\rho_{2,1}|=\sigma_1\sigma_2(\sqrt{1-\delta^2}+\delta\arcsin\delta) \le \frac{\pi}{2}\sigma_1\sigma_2 \ll  \frac{\varrho^2(x,y)}{(1-|x|)(1-|y|)}.
    \]
According to Lemma \ref{l.mu}, we have
    \[
    |\rho_{2,2}|= |\mu_1\mu_2 \arcsin \delta| \le \frac{\pi}{2}|\mu_1\mu_2| \ll  \frac{\varrho^2(x,y)}{(1-|x|)(1-|y|)}\bigg(\frac{1}{1-|x|}+\frac{1}{1-|y|}\bigg).
    \]
For $j=1,2$, it holds that 
   \[
   0\le \int_0^{1}\frac{(1-t)}{\exp(\nu_j^2t^2/2(1-\delta^2))} dt \le \int_0^{1}dt  = 1,
   \]
and
\begin{align*}
    0\le \int_0^{1}\frac{(1-t)|\nu_jt|}{\exp(\nu_j^2t^2/2)}\erf\bigg(\frac{|\delta \nu_j t|}{\sqrt{2(1-\delta^2)}}\bigg)dt\le \int_0^{1}\frac{|\nu_j|}{\exp(\nu_j^2t^2/2)}dt \le \sqrt{\frac \pi 2}.
\end{align*}
Moreover, it follows from \eqref{e.sig1} and \eqref{e.mu1} that 
    \[
    \nu_1^2 =\frac{\mu_1^2}{\sigma_1^2} \ll  \frac{\frac{\varrho^2(x,y)}{(1-|x|)^2}\Big(\frac{1}{\sqrt{1-|x|}}+\frac{1}{\sqrt{1-|y|}}\Big)^2}{\frac{\varrho^2(x,y)}{(1-|x|)^2}} \ll  \frac{1}{1-|x|}+\frac{1}{1-|y|}.
    \]
Similarly, 
    \[
    \nu_2^2 \ll  \frac{1}{1-|x|}+\frac{1}{1-|y|}.
    \]
Therefore, using the fact that $|\delta|\le 1$, we have 
\begin{align*}
    |\rho_{2,3}|&\le \sigma_1\sigma_2\sqrt{1-\delta^2}\left(\nu_1^2+\nu_2^2\right)+\frac{\pi}{2}\sigma_1\sigma_2|\delta| \left(\nu_1^2+\nu_2^2\right)\\
    &\ll  \sigma_1\sigma_2[\nu_1^2+\nu_2^2] \\
    &\ll  \frac{\varrho^2(x,y)}{(1-|x|)(1-|y|)}\bigg(\frac{1}{1-|x|}+\frac{1}{1-|y|}\bigg).
  \end{align*}
Likewise, 
\begin{align*}
    |\rho_{2,4}|
    \le \frac{\pi}{2} |\mu_1\mu_2|
    \ll  \frac{\varrho^2(x,y)}{(1-|x|)(1-|y|)}\bigg(\frac{1}{1-|x|}+\frac{1}{1-|y|}\bigg).
\end{align*}
Note that 
\begin{align*}
    0&\le \frac{|\nu_1\nu_2|}{\sqrt{1-\delta^2}}\int_0^{1}\int_0^1(1-t)(1-s)\exp\left(-\frac{\nu_1^2t^2-2\delta \nu_1\nu_2 ts+\nu_2^2s^2}{2(1-\delta^2)}\right)dtds\\
    &\le \frac{|\nu_1\nu_2|}{\sqrt{1-\delta^2}}\int_0^{1}\int_0^1\exp\left(-\frac{\nu_1^2t^2-2\delta \nu_1\nu_2 ts+\nu_2^2s^2}{2(1-\delta^2)}\right)dtds\\
    &\le \frac{\pi}{2}+\arcsin\delta  \le \pi.
\end{align*}
Hence, by \eqref{e.cmu1} and \eqref{e.cmu2},  
\begin{align*}
    |\rho_{2,5}|\le \pi|\mu_1\mu_2| \ll  \frac{\varrho^2(x,y)}{(1-|x|)(1-|y|)}\bigg(\frac{1}{1-|x|}+\frac{1}{1-|y|}\bigg).
\end{align*}
Combining these, we obtain
\begin{align*}
    \rho_2(x,y)&=\frac{E(x,y)}{\pi^2\sqrt{1-r^2(x,y)}}\left(\sum_{i=1}^5 \rho_{2,i}(x,y)\right)\\
    &\ll  \frac{(1-|x|)^{C^2/4}(1-|y|)^{C^2/4}}{(1-|x|)(1-|y|)}\bigg(\frac{1}{1-|x|}+\frac{1}{1-|y|}\bigg)\\
    &=(1-|x|)^{\frac{C^2}{4}-2}(1-|y|)^{\frac{C^2}{4}-1}+(1-|x|)^{\frac{C^2}{4}-1}(1-|y|)^{\frac{C^2}{4}-2}.
\end{align*}
Thus, for $C\ge 2$, we have 
    \[
    \iint_{(I\cap U_n)^2}\rho_2(x,y)dxdy \ll  \int_{I\cap U_n} (1-|x|)^{\frac{C^2}{4}-2}dx\int_{I\cap U_n}(1-|y|)^{\frac{C^2}{4}-1}dy=o(1),
    \]
and the theorem follows.
\end{proof}

\subsection{Fluctuations of real roots in the core region, part (2)} We now turn to the second part of Theorem \ref{t.varcomp} as well as the central limit theorem stated in Theorem \ref{t.cltcomp}. We show that when $M_n$ is dominated by $R_n$ through an appropriate factor function $\phi$, the number of real roots of $P_n$ exhibits the same asymptotic behavior as that of its centered counterpart  $R_n$.

\begin{theorem} \label{t.comvar} Let $I$ be an interval and let $J$ be an enlargement of $I$. Fix constants $C,\theta>0$, and let $\phi(x)=Cx^{\theta}$ for $x\in [0,1]$. Assume that uniformly for $x\in J$,
\begin{equation}
    \label{e.cdm-in}
    |m^{(i)}(x)| \ll  \frac{\phi(1-|x|)}{(1-|x|)^i},\quad i=0,1,2.
\end{equation}
Then
\begin{equation}
    \label{e.comvar}
    \Var[N_{P_n}(I\cap U_n)] = \Var[N_{R_n}(I\cap U_n)]+o(1). 
\end{equation}
Moreover, if $\Var[N_{R_n}(I\cap U_n)]\ge \epsilon \log n$ for some constant $\epsilon >0$, then $N_{P_n}(I\cap U_n)$ satisfies the CLT; that is, 
\begin{equation} \label{e.CLT.Un}
    \frac{N_{P_n}(I\cap U_n)-\mathbb E[N_{P_n}(I\cap U_n)]}{\Var[N_{P_n}(I\cap U_n)]} \xrightarrow{d} \mathcal N(0,1).
\end{equation}
\end{theorem}

To prove \eqref{e.comvar}, we compare the one-point and two-point correlation functions of the real roots of $P_n$ and  $R_n$. For this purpose, we first recall the corresponding correlation functions for $R_n$, as established in \cite{DN25}. 

Observe that 
    \[
    \Cov[R_n(x),R_n(y)]=\mathbb E[R_n(x)R_n(y)]=\frac{k(xy)}{\sqrt{k(x^2)k(y^2)}}=r(x,y).
    \]
Therefore, by letting 
    \begin{equation}\label{e.rho1c}
     \widetilde{\rho}_1(x)=\frac{1}{\pi}\sqrt{r_{11}(x,x)}
    \end{equation}
and 
    \begin{equation}\label{e.rho2c}
     \widetilde{\rho}_2(x,y) = \frac{\sigma_1\sigma_2}{\pi^2\sqrt{1-r^2(x,y)}}\left(\sqrt{1-\delta^2}+\delta\arcsin \delta\right),
    \end{equation}
we have
    \[
    \mathbb E[N_{R_n}(I\cap U_n)] = \int_{I\cap U_n}\widetilde{\rho}_1(x)dx
    \]
and 
    \[
    \Var[N_{R_n}(I\cap U_n)]=\iint_{(I\cap U_n)^2}[\widetilde{\rho}_2(x,y)-\widetilde{\rho}_1(x)\widetilde{\rho}_1(y)]dxdy +\int_{I\cap U_n}\widetilde{\rho}_1(x)dx.
    \]
    
Note that we can derive the formulas for $\widetilde{\rho}_1$ and $\widetilde{\rho}_2$ from Lemma \ref{l.rho1} and Theorem \ref{t.rho2}, respectively, by setting $m\equiv 0$.

To prove Theorem \ref{t.comvar}, we begin with three auxiliary lemmas.

\begin{lemma}\label{l.cmean} Under the assumptions of Theorem \ref{t.comvar}, it holds uniformly for $x\in J$ that
\begin{equation}
    \label{e.crho1}
    |\rho_1(x)-\widetilde{\rho}_1(x)| \ll  \frac{\phi^2(1-|x|)}{1-|x|}.
\end{equation}
Consequently, 
    \[
    \mathbb E[N_{P_n}(I\cap U_n)]=\mathbb E[N_{R_n}(I\cap U_n)]+o(1).
    \]
\end{lemma}

\begin{lemma} \label{l.cvar} Under the assumptions of Theorem \ref{t.comvar}, it holds uniformly for $(x,y)\in J\times J$ that
\begin{equation}
    \label{e.crho2}
    |\rho_2(x,y)-\widetilde{\rho}_2(x,y)| \ll  \frac{\phi^2(1-|x|)+\phi^2(1-|y|)}{(1-|x|)(1-|y|)}.  
\end{equation}
Consequently, 
\begin{equation}
    \label{e.cirho2}
    \mathbb \iint_{(I\cap U_n)^2}\rho_2(x,y)dxdy=\iint_{(I\cap U_n)^2}\widetilde \rho_2(x,y)dxdy+o(1).  
\end{equation}
\end{lemma}

\begin{lemma} \label{l.cCLT} Under the assumptions of Theorem \ref{t.comvar}, as $n\to \infty$, we have
\begin{equation}
    \label{e.compareroots}
    \frac{N_{P_n}(I\cap U_n)-N_{R_n}(I\cap U_n)}{\sqrt{\log n}} \xrightarrow{d} 0.
\end{equation}
\end{lemma}
   
Granting these lemmas, we now proceed to the proof of Theorem \ref{t.comvar}.

\begin{proof}[Proof of Theorem \ref{t.comvar}] 
We first prove \eqref{e.comvar}. By Lemma~\ref{l.cmean} and the definition of $\phi$,
    \[
    \int_{I\cap U_n} \rho_1(x)dx=\int_{I\cap U_n}\widetilde{\rho}_1(x)dx+o(1)
    \]
and  
\begin{align*}
    \iint_{(I\cap U_n)^2}\rho_1(x)\rho_1(y)dxdy=\iint_{(I\cap U_n)^2}\widetilde{\rho}_1(x)\widetilde{\rho}_1(y)dxdy+o(1).
\end{align*} 
Combining with \eqref{e.cirho2}, we obtain
\begin{align*}
    \Var[N_{P_n}(I\cap U_n)]&= \iint_{(I\cap U_n)^2}[\rho_2(x,y)-\rho_1(x)\rho_1(y)]dxdy+\int_{I\cap U_n}\rho_1(x)dx \\
    &= \iint_{(I\cap U_n)^2}[\widetilde{\rho}_2(x,y)-\widetilde{\rho}_1(x)\widetilde{\rho}_1(y)]dxdy\\
    &\quad +\int_{I\cap U_n}\widetilde{\rho}_1(x)dx +o(1)\\
    &= \Var[N_{R_n}(I\cap U_n)]+o(1),
\end{align*}
proving the claim.

To establish \eqref{e.CLT.Un}, we recall that in \cite{NV22D}*{\S 6}, Nguyen and Vu proved that if $\Var[N_{R_n}(I\cap U_n)]\ge \epsilon \log n$ for some  $\epsilon >0$, then $N_{R_n}(I\cap U_n)$ satisfies the CLT; that is, 
    \[
    \frac{N_{R_n}(I\cap U_n)-\mathbb E[N_{R_n}(I\cap U_n)]}{\sqrt{\Var[N_{R_n}(I\cap U_n)]}}\xrightarrow{d}\mathcal N(0,1).
    \]
Assuming Lemmas \ref{l.cmean} and \ref{l.cvar}, we have
    \[
    \mathbb E[N_{P_n}(I\cap U_n)]=\mathbb E[N_{R_n}(I\cap U_n)]+o(1) \asymp \log n,
    \]
and 
    \[
    \Var[N_{P_n}(I\cap U_n)]=\Var[N_{R_n}(I\cap U_n)]+o(1) \gg  \log n.
    \]
Therefore, as $n\to \infty$, 
\begin{align*}
    &\frac{N_{R_n}(I\cap U_n)-\mathbb E[N_{P_n}(I\cap U_n)]}{\sqrt{\Var[N_{P_n}(I\cap U_n)]}}\\
    &=\frac{N_{R_n}(I\cap U_n)-\mathbb E[N_{R_n}(I\cap U_n)]}{\sqrt{\Var[N_{R_n}(I\cap U_n)]}}(1+o(1))+o(1) \xrightarrow{d} \mathcal N(0,1).
\end{align*}
Thus, \eqref{e.CLT.Un}  follows from Lemma \ref{l.cCLT} and Slutsky's theorem.
\end{proof}

Our next task is to prove Lemmas \ref{l.cmean}, \ref{l.cvar}, and \ref{l.cCLT}.

\begin{proof}[Proof of Lemma \ref{l.cmean}]
Note that
    \[
    \exp\left(-\frac 12 m^2(x)\right)\le 1 \quad \text{and}\quad |\erf\bigg(\frac{m'(x)}{\sqrt{2r_{11}(x,x)}} \bigg)|\le \frac{2}{\sqrt{\pi}}\frac{|m'(x)|}{\sqrt{2r_{11}(x,x)}}.
    \]
So, by \eqref{e.r11} and \eqref{e.cdm-in}, 
\begin{align*}
    |\rho_{1,2}(x)|&=\frac{1}{\sqrt{2\pi}} |m'(x)|\exp\bigg(-\frac 12 m^2(x)\bigg)|\erf\bigg(\frac{m'(x)}{\sqrt{2r_{11}(x,x)}} \bigg)|\\
    &\le \frac{1}{2\pi} |m'(x)| \frac{|m'(x)|}{\sqrt{r_{11}(x,x)}}\\
    &\ll  \frac{\phi^2(1-|x|)}{1-|x|}.
\end{align*}
Now, using $1-x\le e^{-x}\le 1$ for $x\ge 0$, \eqref{e.r11}, and \eqref{e.cdm-in}, we find that
\begin{align*}
    |\rho_{1,1}(x)-\widetilde{\rho}_1(x)|&=\frac{1}{\pi}\sqrt{r_{11}(x,x)}\bigg|1-\exp\bigg(-\frac 12\bigg[m^2(x)+\frac{\left(m'(x)\right)^2}{r_{11}(x,x)}\bigg]\bigg)\bigg|\\
    &\le \frac{1}{2\pi}\sqrt{r_{11}(x,x)} \bigg(m^2(x)+\frac{\left(m'(x)\right)^2}{r_{11}(x,x)}\bigg)\\
    &\ll  \frac{\phi^2(1-|x|)}{1-|x|}.
\end{align*}
Therefore, 
    \[
    |\rho_1(x)-\widetilde{\rho}_1(x)| \le |\rho_{1,1}(x)-\widetilde{\rho}_1(x)|+|\rho_{1,2}(x)| \ll  \frac{\phi^2(1-|x|)}{1-|x|},
    \]
which proves \eqref{e.crho1}. 
        
As a consequence,
\begin{align*}
    \left|\mathbb E[N_{P_n}(I\cap U_n)]-\mathbb E[N_{R_n}(I\cap U_n)]\right| &\le \int_{I\cap U_n}|\rho_1(x)-\widetilde{\rho}_1(x)|dx=o(1).
\end{align*}
proving the assertion.
\end{proof}
    
To prove Lemma \ref{l.cvar}, we first need to estimate $E(x,y)$.
    
\begin{lemma} \label{l.conE} We have
\begin{equation}
    \label{e.conE}
    |E(x,y)-1| \ll  \phi^2(1-|x|)+\phi^2(1-|y|),\quad (x,y)\in J\times J.
\end{equation}
\end{lemma}
\begin{proof} We first assume that $\varrho(x,y)\le 1/3$. Using $1\ge e^{-x}\ge 1-x$ for $x\ge 0$, we obtain 
\begin{align*}
    |E(x,y)-1|&=\bigg|\exp\left(-\frac{m^2(x)-2r(x,y)m(x)m(y)+m^2(y)}{2(1-r^2(x,y))}\right)-1\bigg|\\
    & \le \frac{m^2(x)-2r(x,y)m(x)m(y)+m^2(y)}{2(1-r^2(x,y))}\\
    &\le \frac{\left(m(y)-m(x)\right)^2}{2(1-r^2(x,y))} + \frac{|m(x)m(y)|}{1+r(x,y)}.
\end{align*}
By the mean value theorem, \eqref{e.cdm-in}, Lemma \ref{l.ez}, and the monotonicity of $\phi$, we see that, for some $z$ between $x$ and $y$,
\begin{align*}
    |m(y)-m(x)|& =|y-x||m'(z)| \\
    &\ll  |y-x| \frac{\phi(1-|z|)}{1-|z|}\\
    &\ll  \frac{|y-x|}{|1-xy|}\left[\phi(1-|x|)+\phi(1-|y|)\right]\\
    &= \varrho(x,y)\left[\phi(1-|x|)+\phi(1-|y|)\right].
\end{align*}
Combining with \eqref{e.1-rr}, we get
    \[
    \frac{\left(m(y)-m(x)\right)^2}{2(1-r^2(x,y))} \ll  \left[\phi(1-|x|)+\phi(1-|y|)\right]^2 \ll  \phi^2(1-|x|)+\phi^2(1-|y|).
    \]
Since $1+r(x,y)\asymp 1$, it follows from \eqref{e.cdm-in} that
    \[
    \frac{|m(x)m(y)|}{1+r(x,y)} \ll  \phi(1-|x|)\phi(1-|y|)\ll  \phi^2(1-|x|)+\phi^2(1-|y|),
    \]
and hence that
\begin{align*}
        |E(x,y)-1| \ll  \phi^2(1-|x|)+\phi^2(1-|y|).
\end{align*}
    
Let us now consider $\varrho(x,y)>1/3$. Then 
   \[
   \varrho(x,y)\asymp 1 \quad \text{and}\quad \frac{1}{1-r^2(x,y)} \ll  \frac{1}{\varrho^2(x,y)} \asymp 1.
   \]
Therefore, 
\begin{align*}
    |E(x,y)-1|&\le \frac{m^2(x)-2r(x,y)m(x)m(y)+m^2(y)}{2(1-r^2(x,y))}\\
    &\ll  m^2(x)+m^2(y)\\
    &\ll  \phi^2(1-|x|)+\phi^2(1-|y|),
\end{align*}
proving \eqref{e.conE}. 
\end{proof}
    
\begin{proof}[Proof of Lemma \ref{l.cvar}] We first prove \eqref{e.crho2}. For this purpose, we need to estimate $|\rho_2(x,y)-\widetilde{\rho}_2(x,y)|$, where
   \[
   \rho_2(x,y)=\frac{E(x,y)}{\pi^2\sqrt{1-r^2(x,y)}}\sum_{i=1}^5 \rho_{2,i}(x,y)
   \]
and 
   \[
    \widetilde{\rho}_2(x,y) = \frac{\sigma_1\sigma_2}{\pi^2\sqrt{1-r^2(x,y)}}\left(\sqrt{1-\delta^2}+\delta\arcsin \delta\right).
    \]
According to \eqref{e.1-rr} and Lemma \ref{l.sig}, we have 
   \[
   \widetilde{\rho}_2(x,y)\le \frac{\sigma_1\sigma_2}{2\pi \sqrt{1-r^2(x,y)}} \ll  \frac{\varrho(x,y)}{(1-|x|)(1-|y|)}.
   \]
Together with Lemma \ref{l.conE}, we obtain 
\begin{align*}
    \bigg|\frac{E(x,y)}{\pi^2\sqrt{1-r^2(x,y)}}\rho_{2,1}(x,y)-\widetilde{\rho}_2(x,y)\bigg| &\le  \widetilde{\rho}_2(x,y)|E(x,y)-1|\\
    & \ll  \frac{\phi^2(1-|x|)+\phi^2(1-|y|)}{(1-|x|)(1-|y|)}.
\end{align*}
To establish \eqref{e.crho2}, it remains to show that 
\begin{equation}
    \label{e.rho2i} 
    \frac{|E(x,y)|}{\pi^2\sqrt{1-r^2(x,y)}}\sum_{i=2}^5|\rho_{2,i}(x,y)| \ll  \frac{\phi^2(1-|x|)+\phi^2(1-|y|)}{(1-|x|)(1-|y|)}.
\end{equation}
It follows from the proof of Theorem \ref{t.Cvar} that 
    \[
    \sum_{i=2}^5|\rho_{2,i}| \ll  |\mu_1\mu_2| +\sigma_1\sigma_2\left(\nu_1^2+\nu_2^2\right).
    \]
On account of Lemma \ref{l.sig} and Lemma \ref{l.cmu}, we obtain 
   \[
   \nu_j^2 =\frac{\mu_j^2}{\sigma_j^2} \ll  \phi^2(1-|x|)+\phi^2(1-|y|),\quad j=1,2,
   \]
and hence 
   \[
   |\mu_1\mu_2| +\sigma_1\sigma_2\left(\nu_1^2+\nu_2^2\right) \ll  \frac{\varrho^2(x,y)}{(1-|x|)(1-|y|)}\left[\phi^2(1-|x|)+\phi^2(1-|y|)\right].
   \]
Because $|E(x,y)|\le 1$ and $1-r^2(x,y)\asymp \varrho^2(x,y)$,  we deduce that
\begin{align*}
    \frac{|E(x,y)|}{\pi^2\sqrt{1-r^2(x,y)}}\sum_{i=2}^5|\rho_{2,i}(x,y)|&\ll  \frac{\frac{\varrho^2(x,y)}{(1-|x|)(1-|y|)}\left[\phi^2(1-|x|)+\phi^2(1-|y|)\right]}{\varrho(x,y)}\\
    &\ll  \frac{\phi^2(1-|x|)+\phi^2(1-|y|)}{(1-|x|)(1-|y|)},
\end{align*}
and \eqref{e.rho2i} is verified.  

It follows from \eqref{e.crho2} that 
\begin{align*}
      &\iint_{(I\cap U_n)^2}|\rho_2(x,y)-\widetilde{\rho}_2(x,y)|dxdy\\
      &\ll  \iint_{(I\cap U_n)^2}\frac{\phi^2(1-|x|)+\phi^2(1-|y|)}{(1-|x|)(1-|y|)}dxdy\\
      &=o(1),
  \end{align*}
which proves \eqref{e.cirho2}.
\end{proof}

We now turn to the proof of Lemma \ref{l.cCLT}. 
\begin{proof}[Proof of Lemma \ref{l.cCLT}] We first outline the proof. The main strategy is to compare the statistics of $N_{P_n}(I\cap U_n)$ with those of $N_{R_n}(I\cap U_n)$, and then invoke Markov’s inequality. To this end, we partition $I\cap U_n$ into $O(\log n)$ subintervals and show that, on each subinterval, the number of real roots is well approximated by the number of sign changes of the polynomial. We further control the discrepancy between the sign changes of $P_n$ and $R_n$ in terms of the normalized mean $m$, which is small under the assumption that $M_n$ is dominated by $R_n$.

Before proceeding, we introduce the necessary notation. 

Let $\delta_j=a_n/2^j$, for $j=1, \dots, T$, where $T$ is the smallest positive integer such that $a_n/2^T \le B/n$. Then, $\delta_j\ge B/n$ for $j=1,..., T$, and $T\ll \log n$. 

For $j=1,...,T$, let $I_{1j}= [1-2\delta_j, 1-\delta_j)$ and $I_{2j}=-I_{1j}$. Let $N_{ij}$ and $R_{ij}$ represent the number of real roots of $P_n$ and $R_n$, respectively, in the intersection $I\cap I_{ij}$, where $i=1,2$ and $j=1, \dots, T$. Since the family $\{I_{ij}: i=1,2, j=1,...,T\}$ is pairwise disjoint and covers $U_n$, we have 
   \[
   N_{P_n}(I\cap U_n)=\sum_{j=1}^T\sum_{i=1}^2N_{ij} \quad\text{and}\quad N_{R_n}(I\cap U_n)=\sum_{j=1}^T\sum_{i=1}^2R_{ij}.
   \]
   
Let $a_{ij}$ and $b_{ij}$ denote the endpoints of the interval $I\cap I_{ij}$ whenever it is nonempty. If $P_n$ has at most one root in $I\cap I_{ij}$ and does not vanish at $a_{ij}$ and $b_{ij}$, then $N_{ij}=1$ provided that $P_n(a_{ij})$ and $P_n\left(b_{ij}\right)$ have different signs, and $N_{ij}=0$ otherwise. Thus, on $I\cap I_{ij}$, it is reasonable to approximate $N_{ij}$ by the number of sign changes $N_{ij}^\dagger$, defined as follows. Let
\begin{align*}
    N_{ij}^\dagger&=\begin{cases}\frac 12 -\frac 12\sgn\left(P_n(a_{ij})P_n\left(b_{ij}\right)\right)&\text{if } I\cap I_{ij}\ne \emptyset,\\
    0&\text{otherwise},
    \end{cases}\\
    R_{ij}^\dagger&=\begin{cases}\frac 12 -\frac 12\sgn\left(R_n(a_{ij})R_n\left(b_{ij}\right)\right)&\text{if } I\cap I_{ij}\ne \emptyset,\\
    0&\text{otherwise},
    \end{cases}
\end{align*}
    \[
    N^\dagger=\sum_{j=1}^T\sum_{i=1}^2 N_{ij}^\dagger, \quad \text{\and}\quad R^\dagger=\sum_{j=1}^T\sum_{i=1}^2 R_{ij}^\dagger,
    \]
where
    \[
    \sgn(x):=\begin{cases}
    1&\text{if}\quad x>0,\\
    0&\text{if}\quad x=0,\\
    -1&\text{if}\quad x<0.
    \end{cases}
    \]
    
We divide the proof into three steps.

\setcounter{step}{0}
\begin{step}[Approximating the number of real roots by the number of sign changes] The main result in this step is the following lemma.
\begin{lemma} \label{l.root-sign} If $|m''(x)|\ll  (1-|x|)^{-2}$ uniformly on $J$, then there exist positive constants $C_1$ and $c$ such that 
    \[
    \mathbb E [|N_{P_n}(I\cap U_n)-N^\dagger|]\le C_1 a_n^c.
    \]
Consequently, by letting $m\equiv 0$, we have 
    \[
    \mathbb E [|N_{R_n}(I\cap U_n)-R^\dagger|]\le C_1 a_n^c.
    \]
\end{lemma}
\begin{proof} We first show that there exists a positive constant $\varepsilon$ such that for $i=1,2$ and $j=1, \dots, T$, 
\begin{equation}
    \label{e.sign-approx}
    \mathbb E[|N_{ij}-N_{ij}^\dagger|] \ll  \delta_j^{\varepsilon}.
\end{equation}
If $I\cap I_{ij}=\emptyset$, then $N_{ij}=N_{ij}^\dagger=0$, hence \eqref{e.sign-approx} is trivial. Otherwise, since $N_{ij}^\dagger\in \{0,1/2,1\}$, it follows that 
\begin{align*}
   \mathbb E[|N_{ij}-N_{ij}^\dagger|] \le  \mathbb E[|N_{ij}-N_{ij}^\dagger|\pmb 1_{E_{ij}^c}] +\mathbb E[N_{ij}\pmb 1_{E_{ij}}],
\end{align*}
where $E_{ij}:=\{N_{ij}\ge 2\}$ and $E_{ij}^c=\{N_{ij}\le 1\}$. Since $P_n(x)$ is a Gaussian random variable for any $x$, we have 
    \[\mathbb P\left(P_n(a_{ij})P_n\left(b_{ij}\right)=0\right)=0.\]
If $P_n(a_{ij})P_n\left(b_{ij}\right)\ne 0$ and $N_{ij}\le 1$, then $N_{ij}=N_{ij}^\dagger$. Thus, 
    \[
    \mathbb E[|N_{ij}-N_{ij}^\dagger|\pmb 1_{E_{ij}^c}]=0.
    \]
Since $E_{ij}\subset \{N_{P_n}(J\cap I_{ij})\ge 2\}$, Lemma~\ref{l.nearR} yields the estimate 
    \[
    \mathbb P(E_{ij})\le \mathbb P(N_{P_n}(J\cap I_{ij})\ge 2) \le C_1' \delta_j^{1.5\varepsilon}
    \]
for some constants $C_1'>0$ and $\varepsilon>0$. Combining this with Theorem \ref{t.localcount}, we conclude that
    \[
    \mathbb E[N_{ij}\pmb 1_{E_{ij}}] \le C_1\left(\delta_j^M+|\log \delta_j|^{C_2}\delta_j^{1.5\varepsilon}\right)
    \]
for any $M>0$ and some $C_1, C_2>0$. By choosing $M>\varepsilon$, we deduce that 
    \[
    \mathbb E[N_{ij}\pmb 1_{E_{ij}}] \ll  \delta_j^{\varepsilon}.
    \]
Bringing these estimates together yields \eqref{e.sign-approx}.

Finally, using the triangle inequality and \eqref{e.sign-approx}, we find that 
\begin{align*}
    \mathbb E [|N_{P_n}(I\cap U_n)-N^\dagger|] \le \sum_{j=1}^{T}\sum_{i=1}^2\mathbb E[|N_{ij}-N_{ij}^\dagger|]
    \ll  T a_n^{\varepsilon}
    \ll  a_n^{c},
\end{align*}
for some $c\in \left(0, \varepsilon\right)$. This completes the proof of Lemma \ref{l.root-sign}.
    \end{proof}
\end{step}

\begin{step}[Approximating the sign changes of $P_n$ with those of $R_n$] This step aims to prove the following lemma.
\begin{lemma} \label{l.sign-sign} Under the assumptions of Theorem \ref{t.comvar}, there exist positive constants $C_1$ and $c$ such that 
   \[
   \mathbb E[|N^\dagger-R^\dagger|] \le C_1a_n^c.
   \]
\end{lemma}

\begin{proof} Without loss of generality, we may assume that $I\cap I_{ij}\ne \emptyset$ for all $j=1,...,T$ and $i=1,2$. We begin by proving that for $j=1,...,T$ and $i=1,2$,
\begin{equation}\label{e.sign-Ij}
    \mathbb E[|N_{ij}^\dagger-R_{ij}^\dagger|]\ll  |m(a_{ij})|+|m(b_{ij})|.
\end{equation}
Let 
\begin{align*}
    X_1&=\frac{P_n(a_{ij})}{\sqrt{\Var[P_n(a_{ij})]}}, & X_2&=\frac{P_n\left(b_{ij}\right)}{\sqrt{\Var[P_n\left(b_{ij}\right)]}},\\
    Y_1&=\frac{R_n(a_{ij})}{\sqrt{\Var[R_n(a_{ij})]}}, & Y_2&=\frac{R_n\left(b_{ij}\right)}{\sqrt{\Var[R_n\left(b_{ij}\right)]}}.
\end{align*}
Since $|\sgn(x)|\le 1$ for all $x$ and $\sgn(x)$ is multiplicative, it follows that  
\begin{align*}
    |N_{ij}^\dagger-R_{ij}^\dagger|&=\frac{1}{2}|\sgn\left(X_1X_2\right)-\sgn\left(Y_1Y_2\right)|\\
     &\le \frac{1}{2}|\sgn\left(X_1\right)-\sgn\left(Y_1\right)|+\frac{1}{2}|\sgn\left(X_2\right)-\sgn\left(Y_2\right)|.
\end{align*}

Note that $X_1=Y_1+m(a_{ij})$ and that $|\sgn\left(X_1\right)-\sgn\left(Y_1\right)| \in \{0, 1, 2\}$. If $m(a_{ij})=0$, then $X_1=Y_1$. Hence, we may assume  $m(a_{ij})\ne 0$. In this case, $|\sgn\left(X_1\right)-\sgn\left(Y_1\right)|>0$ only if $X_1Y_1\le 0$, and thus this occurs with probability $\mathbb P(X_1Y_1\le 0)$. Since $Y_1\sim \mathcal N(0,1)$, it follows that 
\[
\mathbb P(X_1Y_1\le 0) \le \mathbb P(-m(a_{ij})\le Y_1 \le 0)+\mathbb P(0\le Y_1 \le -m(a_{ij}))\ll |m(a_{ij})|.
\]
Consequently, 
    \[
    \mathbb E[|\sgn\left(X_1\right)-\sgn\left(Y_1\right)|]\ll  |m(a_{ij})|.
    \]
Similarly, 
    \[
    \mathbb E\left[|\sgn\left(X_2\right)-\sgn\left(Y_2\right)|\right] \ll  |m(b_{ij})|.
    \]
Therefore, 
\begin{align*}
    \mathbb E[|N_{ij}^\dagger-R_{ij}^\dagger|] &\le \frac{1}{2}\mathbb E[|\sgn\left(X_1\right)-\sgn\left(Y_1\right)|]+\frac{1}{2}\mathbb E[|\sgn\left(X_2\right)-\sgn\left(Y_2\right)|]\\
    & \ll |m(a_{ij})|+|m(b_{ij})|,
\end{align*}
proving \eqref{e.sign-Ij}.

Finally, by the triangle inequality together with \eqref{e.sign-Ij}, we obtain
\begin{align*}
    \mathbb E[|N^\dagger-R^\dagger|]& \le \sum_{j=1}^{T}\sum_{i=1}^2\mathbb E[|N_{ij}^\dagger-R_{ij}^\dagger|] \ll \sum_{j=1}^{T}\sum_{i=1}^2 (|m(a_{ij})|+|m(b_{ij})|).
\end{align*}
It follows from the assumptions that 
\[
|m(a_{ij})|\le \phi(1-|a_{ij}|),\quad \int_{U_n}\frac{\phi(1-|x|)}{1-|x|}dx \ll a_n^\theta,
\]
and that $\frac{\phi(1-|x|)}{1-|x|}$ is monotone on $U_n$. Writing the sum as a Riemann sum with mesh sizes $\delta_j$ and using monotonicity to compare with the corresponding integral, we obtain
\begin{align*}
    \sum_{j=1}^{T}\sum_{i=1}^2|m(a_{ij})|&\ll  \sum_{j=1}^{T}\sum_{i=1}^2\frac{|\phi(1-|a_{ij}|)|}{1-|a_{ij}|}\delta_j\\
    &\ll  \int_{U_n}\frac{\phi(1-|x|)}{1-|x|}dx+O(a_n)\\
    &\ll a_n^\theta+O(a_n).
\end{align*}
Similarly, 
    \[
    \sum_{j=1}^{T}\sum_{i=1}^2|m\left(b_{ij}\right)|\ll a_n^\theta+O(a_n),
    \]
and the lemma follows.
\end{proof}
\end{step}

\begin{step}[Proving \eqref{e.compareroots}]  By the triangle inequality, Lemma \ref{l.root-sign}, and Lemma \ref{l.sign-sign},
\begin{align*}
    \mathbb E[|N_{P_n}(I\cap U_n)-N_{R_n}(I\cap U_n)|] &\ll  \mathbb E[|N_{P_n}(I\cap U_n)-N^\dagger|]+\mathbb E[|N^\dagger-R^\dagger|]\\
    &\quad +\mathbb E[|R^\dagger-N_{R_n}(I\cap U_n)|]\\
    &\ll a_n^c.
\end{align*}
Using Markov's inequality, we see that for any $\varepsilon>0$,
    \[
    \mathbb P\left(\frac{|N_{P_n}(I\cap U_n)-N_{R_n}(I\cap U_n)|}{\sqrt{\log n}}\ge \varepsilon \right) \ll \frac{a_n^c}{\varepsilon \sqrt{\log n}} \to 0,
    \]
which implies \eqref{e.compareroots}.
\end{step}
This completes the proof of Lemma \ref{l.cCLT}.
\end{proof}
\subsection{Fluctuations of real roots via reciprocal polynomials} This subsection is devoted to the proofs of Theorems \ref{t.varcomp} and \ref{t.cltcomp} for $I\backslash [-1,1]$. In particular, we study the real roots of $P_n$ in the core regions outside the interval $[-1,1]$. 

To study the real roots of $P_n$ in $U_n^*$, we will work with the reciprocal polynomial 
    \[P_n^*(x):=x^nP_n(1/x),\] 
which transforms roots of $P_n$ in $\left(-\infty,-1\right)\cup \left(1,\infty\right)$ into roots of $P_n^*$ in $\left(-1,1\right)$. 

We begin with asymptotic estimates for the mean and the variance of $P_n^*$. Note that 
    \[
    P_n^*(x)=\sum_{j=0}^n \left(m_{n-j}+v_{n-j}\xi_{n-j}\right)x^j=M_n^*(x)+R_n^*(x).
    \]
Let $k^*(x)$ denote the corresponding variance function of $R_n^*$, 
    \[
    k^*(x)=\sum_{j=0}^n v_{n-j}^2x^j=x^nk\left(1/x\right). 
    \]
By conditions \ref{A2} and \ref{A3},
    \[
    \frac{v_{n-j}^2}{n^{2\tau}}\asymp 1+O(j/n),\quad j\ge 0.
    \]
Thus, it holds uniformly for $x\in U_n$,
    \begin{align*}
        k^*(x)=n^{2\tau}\sum_{j=0}^n \frac{v_{n-j}^2}{n^{2\tau}}x^j
        &\asymp n^{2\tau}\sum_{j=0}^n \big(1+O(j/n)\big)x^j\\
        &\asymp \frac{n^{2\tau}}{1-x}\bigg(1+O\bigg(\frac{1}{n(1-x)}\bigg)\bigg).
    \end{align*}
Since $n(1-x)\ge B$, it follows that 
\begin{equation}\label{e.k*}
    k^*(x) \asymp \frac{n^{2\tau}}{1-x},\quad x\in U_n.  
\end{equation}
Similarly, for $1\le i\le 4$, it holds that 
\begin{equation}\label{e.djk*}
    {k^*}^{(i)}(x) \asymp \frac{n^{2\tau}}{(1-x)^{i+1}},\quad x\in U_n.
\end{equation}
By \ref{A2}, we also have 
    \[
    {M_n^*}^{(i)}(x) \ll  \frac{n^{\tau}}{(1-x)^{i+1}},\quad x\in U_n.
    \]
Then
    \[
    m^*(x)=\frac{\mathbb E[P_n^*(x)]}{\sqrt{\Var[P_n^*(x)]}}=\frac{M_n^*(x)}{\sqrt{k^*(x^2)}}, 
    \]
and hence, for $0\le i\le 2$,
\begin{equation*}
    {m^*}^{(i)}(x) \ll  \frac{1}{(1-|x|)^{i+\frac{1}{2}}},\quad x\in U_n.
\end{equation*}
We will denote by $r^*$,  $\sigma_1^*$, $\sigma_2^*$, $\mu_1^*$, $\mu_2^*$,  the analogous quantities. We have 
    \[
    r^*(x,y) =\frac{k^*(xy)}{\sqrt{k^*(x^2)k^*(y^2)}}\asymp \sqrt{\alpha},\quad (x,y)\in S_n\times S_n.
    \]
An argument similar to the previous treatment for $P_n$ (see Section~\ref{s.8g}), specialized to the case $\tau=0$, shows that for $(x,y)\in U_n\times U_n$,
\begin{align*}
    1-\left(r^*(x,y)\right)^2 &\asymp \varrho^2(x,y), &r_{11}^*(x,x) &\asymp \frac{1}{(1-|x|)^2},\\
    \sigma_1^* &\asymp \frac{\varrho(x,y)}{1-|x|}, &\sigma_2^* &\asymp \frac{\varrho(x,y)}{1-|y|}.
\end{align*}
The same conclusion can be drawn for $\mu_1^*$ and $\mu_2^*$. That is, it holds uniformly for $(x,y)\in U_n\times U_n$ that 
\begin{equation*}
    \mu_1^* \ll  \frac{\varrho(x,y)}{1-|x|}\bigg(\frac{1}{\sqrt{1-|x|}}+\frac{1}{\sqrt{1-|y|}}\bigg)
\end{equation*}
and 
\begin{equation*}
    \mu_2^* \ll  \frac{\varrho(x,y)}{1-|y|}\bigg(\frac{1}{\sqrt{1-|x|}}+\frac{1}{\sqrt{1-|y|}}\bigg).
\end{equation*}
Assume, in addition, that uniformly for $x\in U_n$,
\begin{equation*}
    |{m^*}^{\left(j\right)}(x)| \ll  \frac{\phi(1-|x|)}{(1-|x|)^j},\quad i=0,1,2.
\end{equation*}
Then, uniformly for $(x,y)\in U_n\times U_n$,
\begin{equation*}
    \mu_1^* \ll  \frac{\varrho(x,y)}{1-|x|}[\phi(1-|x|)+\phi(1-|y|)],
\end{equation*}
and 
\begin{equation*}
    \mu_2^* \ll  \frac{\varrho(x,y)}{1-|y|}[\phi(1-|x|)+\phi(1-|y|)].
\end{equation*}

Observing that $N_{P_n^*}(U_n)=N_{P_n}(U_n^*)$ and $N_{R_n^*}(U_n)=N_{R_n}(U_n^*)$, we obtain the following theorem. Its proof follows verbatim the arguments used in the proofs of Theorem \ref{t.Cvar} and Theorem \ref{t.comvar}, and is therefore omitted.  

\begin{theorem}\label{t.Cvar*} Let $I$ be an interval and let $J$ be an enlargement of $I$. 
\begin{enumerate}
    \item There exists a constant $C>0$ such that if
\begin{equation}
    \label{e.cm*}
    |m^*\left(1/y\right)|\ge C|\log(1-|1/y|)|^{1/2},\quad y\in I\cap U_n^*,
\end{equation}
then 
\begin{equation}
    \label{e.Cvar*}
    \Var[N_{P_n}(I\cap U_n^*)] =o(1).
\end{equation}
\item Fix constants $C,\theta>0$, and let $\phi(x)=Cx^{\theta}$ for $x\in [0,1]$. Assume that uniformly for $y\in J$, 
\begin{equation}
    \label{e.djm-out}
    {m^*}^{(i)}\left(1/y\right) \ll  \frac{\phi(1-|1/y|)}{(1-|1/y|)^i},\quad i=0,1,2.
\end{equation}
Then
\begin{equation}
    \label{e.comvar*}
    \Var[N_{P_n}(I\cap U_n^*)] = \Var[N_{R_n}(I\cap U_n^*)]+o(1). 
\end{equation}
Furthermore, if $\Var[N_{R_n}(I\cap U_n^*)]\ge \epsilon \log n$ for some constant $\epsilon >0$, then $N_{P_n}(I\cap U_n^*)$ satisfies the CLT.
\end{enumerate}
\end{theorem}

\subsection{Completion of the proofs of Theorems \ref{t.varcomp} and \ref{t.cltcomp}} We conclude the proofs of Theorems \ref{t.varcomp} and \ref{t.cltcomp} by studying the contribution of real roots of $P_n$ lying in the region $\mathcal I_n=U_n\cup U_n^*$, as formalized in the following theorem.
\begin{theorem}\label{t.comvarclt} Let $I$ be an interval and let $J$ be an enlargement of $I$.
\begin{enumerate}
    \item There exists a constant $C>0$ such that, if conditions \eqref{e.cm} and \eqref{e.cm*} hold, then 
\begin{equation} \label{e.varo1}
    \Var[N_{P_n}(I\cap \mathcal I_n)]=o(1). 
\end{equation}
\item Fix constants $C,\theta>0$, and let $\phi(x)=Cx^{\theta}$ for $x\in [0,1]$. If \eqref{e.cdm-in} holds uniformly for $x\in J\cap [-1,1]$ and \eqref{e.djm-out} holds uniformly for $y\in J\backslash [-1,1]$, then 
\begin{equation} \label{e.varPR}
    \Var[N_{P_n}\left(I\cap \mathcal I_n\right)] = \Var[N_{R_n}\left(I\cap \mathcal I_n\right)]+o(1). 
\end{equation}
If, in addition, there exists a constant $\epsilon > 0$ such that 
\[
\Var[N_{R_n}\left(I\cap \mathcal I_n\right)]\ge \epsilon \log n,
\]
then $N_{P_n}\left(I\cap \mathcal I_n\right)$ satisfies the CLT.
\end{enumerate}
\end{theorem}
Combining Theorems \ref{t.Cvar}, \ref{t.comvar}, and \ref{t.Cvar*}, we see that to prove Theorem \ref{t.comvarclt}, it suffices to estimate the correlations between real roots lying inside and outside $[-1,1]$. More precisely, we have the following lemmas.
\begin{lemma}\label{l.crho1-mix} 
If \eqref{e.djm-out} holds, then 
\begin{equation}
    \label{e.crho1-mix}
    |\rho_1(y)-\widetilde{\rho}_1(y)| \ll  \frac{\phi^2(1-|1/y|)}{1-|1/y|},\quad y\in I\cap U_n^*,
\end{equation}
where $\rho_1$ and $\widetilde{\rho_1}$ are given in \eqref{e.rho1} and \eqref{e.rho1c}, respectively.
\end{lemma}
\begin{lemma} \label{l.cvar-mix}
Assume that \eqref{e.cdm-in} and \eqref{e.djm-out} hold. Then it holds uniformly for $(x,y)\in (I\cap U_n)\times (I\cap U_n^*)$ that
\begin{equation}
    \label{e.crho2-mix}
    |\rho_2(x,y)-\widetilde{\rho}_2(x,y)| \ll  \frac{\phi^2(1-|x|)+\phi^2(1-|1/y|)}{(1-|x|)(1-|1/y|)},
\end{equation}
where $\rho_2$ and $\widetilde{\rho_2}$ are given in \eqref{e.rho2} and \eqref{e.rho2c}, respectively.
Consequently, 
\begin{equation}
    \label{e.cirho2-mix}
    \int_{I\cap U_n}dx\int_{I\cap U_n^*}\rho_2(x,y)dy=\int_{I\cap U_n}dx\int_{I\cap U_n^*}\widetilde{\rho}_2(x,y)dy+o(1).  
\end{equation}
\end{lemma}
Assuming Lemmas \ref{l.crho1-mix} and \ref{l.cvar-mix} hold, we now proceed to prove Theorem \ref{t.comvarclt}.
\begin{proof}[Proof of Theorem \ref{t.comvarclt}] Since 
    \[
    \Cov[N_{P_n}(I\cap U_n),N_{P_n}(I\cap U_n^*)] \le \sqrt{\Var[N_{P_n}(I\cap U_n)]\Var[N_{P_n}(I\cap U_n^*)]},
    \]
the asymptotic estimate \eqref{e.varo1} follows immediately from \eqref{e.Cvar} and \eqref{e.Cvar*}.

By \eqref{e.comvar} and \eqref{e.comvar*}, the estimate \eqref{e.varPR} will be proved once we prove that 
\begin{align*}
    \Cov[N_{P_n}(I\cap U_n),N_{P_n}(I\cap U_n^*)]
    &=\Cov[N_{R_n}(I\cap U_n),N_{R_n}(I\cap U_n^*)]+o(1).
\end{align*}
Indeed, it follows from \eqref{e.crho1} and \eqref{e.crho1-mix} that 
    \[
    \int_{I\cap U_n}dx\int_{I\cap U_n^*}\rho_1(x)\rho_1(y)dy = \int_{I\cap U_n}dx\int_{I\cap U_n^*}\widetilde{\rho}_1(x)\widetilde{\rho}_1(y)dy+o(1).
    \]
Together with \eqref{e.cirho2-mix}, we deduce that 
\begin{align*}
    &\Cov[N_{P_n}(I\cap U_n),N_{P_n}(I\cap U_n^*)]\\
    &= \int_{I\cap U_n}dx\int_{I\cap U_n^*}[\rho_2(x,y)-\rho_1(x)\rho_1(y)]dy\\
    &= \int_{I\cap U_n}dx\int_{I\cap U_n^*}[\widetilde{\rho}_2(x,y)-\widetilde{\rho}_1(x)\widetilde{\rho}_1(y)]dy +o(1)\\
    &=\Cov[N_{R_n}(I\cap U_n),N_{R_n}(I\cap U_n^*)]+o(1),
\end{align*}
which is the desired conclusion.

The CLT for $N_{P_n}(I\cap \mathcal I_n)$ follows by the same argument as in the proof of the CLT for $N_{P_n}(I\cap U_n)$ in Theorem~\ref{t.comvar}, with $I\cap \mathcal I_n$ in place of $I\cap U_n$. Namely, we partition $I\cap \mathcal I_n$ into $O(\log n)$ subintervals and show that, on each subinterval, the number of real zeros is well approximated by the number of sign changes of the polynomial. The only difference arises for subintervals outside $[-1,1]$, where we pass to the reciprocal polynomial $P_n^*$. In this case, we control the discrepancy between the sign changes of $P_n^*$ and $R_n^*$ via $m^*$, which remains negligible under the assumption that $M_n$ is dominated by $R_n$. As the modifications are straightforward, we omit the details.
\end{proof}

It remains to prove Lemmas \ref{l.crho1-mix} and \ref{l.cvar-mix}. For this, we need estimates for $r$, its partial derivatives, $\mu_1$, and $\mu_2$.
\begin{lemma} \label{l.cor-in-out}
For any sufficiently large constant $B$, the following inequalities hold for all sufficiently large $n$ and all $(x,y)\in U_n\times U_n^*$, with $b_n=B/n$:
\begin{equation}
    \label{e.r-mix}
    |r(x,y)| \le e^{-B/2},
\end{equation}
and 
\begin{equation}
    \label{e.dr-mix}
    |r_{10}(x,y)| \le e^{-B/3}\sqrt{r_{11}(x,x)},\quad  |r_{01}(x,y)|\le e^{-B/3}\sqrt{r_{11}(y,y)}.
\end{equation}
Consequently, uniformly for $(x,y)\in U_n\times U_n^*$, 
\begin{equation}\label{e.sigmix}
    \sigma_1 \asymp \sqrt{r_{11}(x,x)} \quad \text{and}\quad \sigma_2 \asymp \sqrt{r_{11}(y,y)},
\end{equation}
where the asymptotic behavior of $r_{11}(x,x)$ is given in \eqref{e.r11}, and
\begin{equation}
    \label{e.r11-out}
    r_{11}(y,y) \asymp \frac{1}{1-|1/y|},\quad y\in U_n^*.
\end{equation}
Assume further that for some function $\phi$, uniformly for $x\in U_n$,
\begin{equation}
    \label{e.mphi}
    m^{(i)}(x) \ll \frac{\phi(1-|x|)}{(1-|x|)^i},\quad i=0,1,
\end{equation}
and uniformly for $y\in U_n^*$,
\begin{equation}
    \label{e.m*phi}
    {m^*}^{(i)}(1/y) \ll \frac{\phi(1-|1/y|)}{(1-|1/i|)^i},\quad i=0,1.
\end{equation}
Then it holds uniformly for $(x,y)\in U_n\times U_n^*$ that
\begin{equation}
    \label{e.mu1mix}
    \mu_1 \ll  \frac{1}{1-|x|}[\phi(1-|x|)+\phi(1-|1/y|)]
\end{equation}
and 
\begin{equation}
    \label{e.mu2mix}
    \mu_2 \ll \frac{1}{1-|1/y|}[\phi(1-|x|)+\phi(1-|1/y|)].
\end{equation}
\end{lemma}
The proof of Lemma \ref{l.cor-in-out} is deferred to Appendix \ref{a.l8.24}.

We now proceed to prove Lemmas \ref{l.crho1-mix} and \ref{l.cvar-mix}.
\begin{proof}[Proof of Lemma \ref{l.crho1-mix}]
Using \eqref{e.r11-out} and \eqref{e.djm-out}, we see that
\begin{align*}
    |\rho_{1,2}(y)| &\le \frac{1}{2\pi} |m'(y)| \frac{|m'(y)|}{\sqrt{r_{11}(y,y)}}
    \ll  \frac{\phi^2(1-|1/y|)}{1-|1/y|},
\end{align*}
and 
\begin{align*}
    |\rho_{1,1}(y)-\widetilde{\rho}_1(y)| &\le \frac{1}{2\pi}\sqrt{r_{11}(y,y)} \bigg(m^2(y)+\frac{\left(m'(y)\right)^2}{r_{11}(y,y)}\bigg)
    \ll  \frac{\phi^2(1-|1/y|)}{1-|1/y|}.
\end{align*}
Therefore, 
    \[
    |\rho_1(y)-\widetilde{\rho}_1(y)| \le |\rho_{1,1}(y)-\widetilde{\rho}_1(y)|+|\rho_{1,2}(y)| \ll  \frac{\phi^2(1-|1/y|)}{1-|1/y|},
    \]
which proves \eqref{e.crho1-mix}. 
\end{proof}

\begin{proof}[Proof of Lemma \ref{l.cvar-mix}] Recall that 
    \[
    \rho_2(x,y)=\frac{E(x,y)}{\pi^2\sqrt{1-r^2(x,y)}}\sum_{i=1}^5 \rho_{2,i}(x,y)
    \]
and 
    \[
    \widetilde{\rho}_2(x,y) = \frac{\sigma_1\sigma_2}{\pi^2\sqrt{1-r^2(x,y)}}\left(\sqrt{1-\delta^2}+\delta\arcsin \delta\right).
    \]
We have
\begin{align*}
    |E(x,y)-1|&\le \frac{m^2(x)-2r(x,y)m(x)m(y)+m^2(y)}{2(1-r^2(x,y))}\\
    &\ll  m^2(x)+m^2(y)\\
    &\ll  \phi^2(1-|x|)+\phi^2(1-|1/y|).
\end{align*}
According to Lemma \ref{l.cor-in-out}, we have 
   \[
   \widetilde{\rho}_2(x,y)\ll  \sigma_1\sigma_2 \ll  \frac{1}{(1-|x|)(1-|1/y|)},
   \]
whence
\begin{align*}
    \bigg|\frac{E(x,y)}{\pi^2\sqrt{1-r^2(x,y)}}\rho_{2,1}(x,y)-\widetilde{\rho}_2(x,y)\bigg| &\le  \widetilde{\rho}_2(x,y)|E(x,y)-1|\\
    & \ll  \frac{\phi^2(1-|x|)+\phi^2(1-|1/y|)}{(1-|x|)(1-|1/y|)}.
\end{align*}
It follows from the proof of Theorem \ref{t.Cvar} that 
    \[
    \sum_{i=2}^5|\rho_{2,i}| \ll  |\mu_1\mu_2| +\sigma_1\sigma_2\left(\nu_1^2+\nu_2^2\right).
    \]
On account of Lemma \ref{l.cor-in-out}, we obtain 
   \[
   |\mu_1\mu_2| \ll  \frac{\phi^2(1-|x|)+\phi^2(1-|1/y|)}{(1-|x|)(1-|1/y|)}
   \]
and 
   \[
   \sigma_1\sigma_2\left(\nu_1^2+\nu_2^2\right) \ll  \frac{\phi^2(1-|x|)+\phi^2(1-|1/y|)}{(1-|x|)(1-|1/y|)}.
   \]
Therefore,
    \[
    \frac{|E(x,y)|}{\pi^2\sqrt{1-r^2(x,y)}}\sum_{i=2}^5|\rho_{2,i}(x,y)|\ll  \frac{\phi^2(1-|x|)+\phi^2(1-|1/y|)}{(1-|x|)(1-|1/y|)}.
    \]
Now, 
\begin{align*}
    |\rho_2(x,y)-\widetilde{\rho}_2(x,y)|&\le \bigg|\frac{E(x,y)}{\pi^2\sqrt{1-r^2(x,y)}}\rho_{2,1}(x,y)-\widetilde{\rho}_2(x,y)\bigg|\\
    &\quad + \frac{|E(x,y)|}{\pi^2\sqrt{1-r^2(x,y)}}\sum_{i=2}^5|\rho_{2,i}(x,y)| \\
    & \ll  \frac{\phi^2(1-|x|)+\phi^2(1-|1/y|)}{(1-|x|)(1-|1/y|)},
\end{align*}
and \eqref{e.crho2-mix} is proved.
    
It follows from \eqref{e.crho2-mix} that 
\begin{align*}
    &\int_{I\cap U_n}dx\int_{I\cap U_n^*}|\rho_2(x,y)-\widetilde{\rho}_2(x,y)|dy\\
    &\ll  \int_{I\cap U_n}dx\int_{I\cap U_n^*} \frac{\phi^2(1-|x|)+\phi^2(1-|1/y|)}{(1-|x|)(1-|1/y|)} dy\\
    &=o(1),
\end{align*}
which proves \eqref{e.cirho2-mix}.
\end{proof}

\section{Proof of the main theorems using the comparison principles} \label{s.proof.main}
In this section, we employ Theorems \ref{t.varcomp}, \ref{t.cltcomp}, and Lemma \ref{t.outside-In} to establish Theorems \ref{t.kac}, \ref{t.kac-der}, and \ref{t.hyperkac}. 

To begin, we recall a result from \cite{DN25} concerning the variance of the number of real roots of centered random polynomials whose coefficients have polynomial asymptotics.
\begin{proposition}[\cite{DN25}] \label{p.DN}
Let $R_n(x)=\sum_{j=0}^n \xi_j v_j x^j$, where  $(\xi_j)_{j=0}^n$ satisfy \ref{A1} and  $|v_j|=C_0j^{\tau}\left(1+o_j\left(1\right)\right)$ for some positive constant $C_0$. Let $\kappa_\tau$ be defined by \eqref{kappa}. Then, for $I=(0,\infty)$, as $n\to \infty$, 
\begin{align*}
    \Var[N_{R_n}(\pm I)]=\left[\kappa_\tau+\frac 1\pi \left(1-\frac 2\pi\right)+o(1)\right]\log n,
\end{align*}
and 
    \[
    \Var[N_{R_n}]=\left[2\kappa_\tau+\frac 2\pi \left(1-\frac 2\pi\right)+o(1)\right]\log n.
    \]
\end{proposition}

Since the Kac polynomials belong to the class of hyperbolic polynomials, Theorems \ref{t.kac} and \ref{t.kac-der} follow directly from Theorem~\ref{t.hyperkac}.

\begin{proof}[Proof of Theorem \ref{t.hyperkac}]
It suffices to consider the case $\ell=0$, as the remaining cases are analogous.

Let $P_{n,L}(x)$ be given by \eqref{e.hyperbolic}, and decompose
\[
M_{n,L}(x):=\mathbb E[P_{n,L}(x)] \quad\text{and}\quad R_{n,L}(x):=P_{n,L}(x)-M_{n,L}(x).
\]
For each $i=0,1,2$, we have 
    \[
   M_{n,L}^{(i)}(x) \sim \mu \sum_{j=1}^n \sqrt{\frac{L\left(L+1\right)\cdots\left(L+j-1\right)}{j!}} j^i x^j,
    \]
and 
    \[
    \Var[R_{n,L}^{(i)}(x)] \sim \sum_{j=1}^n \frac{L\left(L+1\right)\cdots\left(L+j-1\right)}{j!} j^{2i} x^{2j}.
    \]
Since 
\[
\frac{L\left(L+1\right)\cdots\left(L+j-1\right)}{j!} \asymp j^{L-1},
\]
it follows that, for each $i=0,1,2$,
     \[
    M_{n,L}^{(i)}(x) \asymp \frac{1}{(1-x+1/n)^{i+\left(L+1\right)/2}},\quad x\in [-1,1],
    \]
and 
\[
\Var[R_{n,L}^{(i)}(x)] \asymp \frac{1}{(1-x^2+1/n)^{L+2i}},\quad x\in [-1,1].
\]
To treat the region $\mathbb R\backslash[-1,1]$, consider the normalized reciprocal polynomial
\[P_{n,L}^*(x):=\frac{x^n}{v_{n,L}}P_{n, L}\left(1/x\right),\] 
where
\[v_{n,L}=\sqrt{\frac{L\left(L+1\right)\cdots\left(L+n-1\right)}{n!}}.\]
We also define
\[
M_{n,L}^*(x):=\mathbb E[P_{n,L}^*(x)] \quad\text{and}\quad R_{n,L}^*(x):=P_{n,L}^*(x)-M_{n,L}^*(x).
\]
Proceeding as above (noting that the reciprocal ensemble behaves like the case $L=1$), we obtain for $i=0,1,2$,
     \[
    {M_{n,L}^*}^{(i)}(x) \asymp \frac{1}{(1-x+1/n)^{i+1}},\quad x\in [-1,1],
    \]
and 
\[
\Var[{R_{n,L}^*}^{(i)}(x)] \asymp \frac{1}{(1-x^2+1/n)^{2i+1}},\quad x\in [-1,1].
\]
Let $I=[0, \infty)$. Then for $i=0,1,2$,
\[
 M_{n,L}^{(i)}(x) \asymp \frac{1}{(1-x+1/n)^{1/2}} \sqrt{\Var[R_{n,L}(x)]},\quad x\in I\cap [0,1],
\]
and 
\[
 {M_{n,L}^*}^{(i)}(1/y) \asymp \frac{1}{(1-1/y+1/n)^{1/2}} \sqrt{\Var[{R_{n,L}^*}(1/y)]},\quad y\in I\backslash [0,1].
\]
Moreover, 
    \[
   \frac{1}{(1-x+1/n)^{1/2}} \ge C\sqrt{|\log(1-x+1/n)|},\quad x\in [0,1],
    \]
for any fixed constant $C>0$ and all sufficiently large $n$. Hence, $M_{n,L}$ dominates $R_{n,L}$ on $I$ with factor function $C|\log x|^{1/2}$. Applying Theorem~\ref{t.varcomp} with $d=1/2$, we obtain
    \[
    \Var[N_{n,0}(I)]=O(\sqrt{\log n})=o(\log n).
    \]

Next, for each $i=0,1,2$, 
    \[
    M_{n,L}^{(i)}(x)  \ll (1-|x|+1/n)^{L/2} \sqrt{\Var[R_{n,L}^{(i)}(x)]},\quad x\in (-I)\cap [-1,0],
    \]
and 
     \[
    {M_{n,L}^*}^{(i)}(1/y)  \ll (1-|1/y|+1/n)^{1/2} \sqrt{\Var[{R_{n,L}^*}^{(i)}(1/y)]},\quad y\in (-I)\backslash [-1,0].
    \]
Thus, $M_{n,L}$ is dominated by $R_{n,L}$ on $-I$ up to order $2$ with factor function $|x|^\theta$, where $\theta=\min\{L/2, 1/2\}>0$ and $C>0$ is a constant depending on $L$. Therefore, Theorem~\ref{t.varcomp} together with Proposition~\ref{p.DN} yields
    \[
    \Var[N_{n,0}(-I)]=\left[\kappa_\tau+\frac 1\pi \left(1-\frac 2\pi\right)+o(1)\right]\log n,
    \]
    where $\kappa_\tau$ is defined in \eqref{kappa} and $\tau=(L-1)/2$.
Consequently, 
    \begin{align*}
        \Var[N_{n,0}]&=\Var[N_{n,0}(-I)]+\Var[N_{n,0}(I)]+2\Cov[N_{n,0}(-I), N_{n,0}(I)]\\
        &=\left[\kappa_\tau+\frac 1\pi \left(1-\frac 2\pi\right)+o(1)\right]\log n+o(\log n)+o(\log n)\\
        &=\left[\kappa_\tau+\frac 1\pi \left(1-\frac 2\pi\right)+o(1)\right]\log n.
    \end{align*}
    
To prove the CLT for $N_{n,0}$, we write 
    \[
    \frac{N_{n,0}-\mathbb E[N_{n,0}]}{\sqrt{\Var[N_{n,0}]}} = \frac{N_{n,0}(-I)-\mathbb E[N_{n,0}(-I)]}{\sqrt{\Var[N_{n,0}]}} +\frac{N_{n,0}(I)-\mathbb E[N_{n,0}(I)]}{\sqrt{\Var[N_{n,0}]}}.
    \]
By Theorem~\ref{t.cltcomp}, $N_{n,0}(-I)$ satisfies the CLT, and hence
    \[
    \frac{N_{n,0}(-I)-\mathbb E[N_{n,0}(-I)]}{\sqrt{\Var[N_{n,0}]}}=\frac{N_{n,0}(-I)-\mathbb E[N_{n,0}(-I)]}{\sqrt{\Var[N_{n,0}(-I)]}}(1+o(1)) \xrightarrow{d} \mathcal N(0,1).
    \]
On the other hand, by Markov’s inequality, for any $\varepsilon>0$,
\[
\mathbb P\bigg(\bigg|\frac{N_{n,0}(I)-\mathbb E[N_{n,0}(I)]}{\sqrt{\Var[N_{n,0}]}}\bigg|\ge \varepsilon \bigg) \le \frac{\Var[N_{n,0}(I)]}{\varepsilon^2\Var[N_{n,0}]}=o(1),
\]
which implies that 
\[
\frac{N_{n,0}(I)-\mathbb E[N_{n,0}(I)]}{\sqrt{\Var[N_{n,0}]}}\xrightarrow{\mathbb P} 0.
\]
The conclusion follows from Slutsky’s theorem.
\end{proof}

\begin{appendix}
\section{Proof of Lemma \ref{l.estimate-Nj}} \label{a.l5.2}

Note that $T\ll \log n$, and $a_n$ decays faster than any power of $\log n$.  Thus, by the mean value theorem and the triangle inequality, it suffices to show that
\begin{equation} 
\label{e.E[Nj]}
    \mathbb E[|N_{j}-\varphi_{j}|] \ll  a_n^{\varepsilon/4}+n^{-c},
\end{equation}
uniformly over $1\le j\le T$. 

For an interval $K$ and a number $t>0$, let $K(t)=K+[-t,t]$ denote the $t$-neighborhood of $K$ on the real line. Let $J_j^-$ and $J_j^+$ be the left and right connected components of the set difference $J_j(\delta_j^{1+\varepsilon})\backslash J_j$.

Let $Z_{j}$ denote the number of complex roots of $P_n$ with their real parts in $J_j(\delta_j^{1+\varepsilon})$ and imaginary parts in $(-\delta_j^{1+\varepsilon}, \delta_j^{1+\varepsilon})\backslash\{0\}$. Then
\begin{align*}
    \mathbb E[|N_{j}-\varphi_{j}|] \ll  \mathbb E[N_{P_n}(J_j^{-})]+ \mathbb E[N_{P_n}(J_j^{+})] +\mathbb E[Z_{j}].
\end{align*}

Using the universality estimates in Theorem~\ref{t.corr} for a suitable bump function supported on $J_j^-(\delta_{j}^{1+\varepsilon})$  (that equals $1$ on $J_j^-$), we obtain, for some $\alpha_1>0$,
\begin{align*}
    \mathbb E[N_{P_n}(J_j^{-})] &\le  \mathbb E[N_{P_{n,G}}(J_j^-(\delta_{j}^{1+\varepsilon}))]+O(\delta_j^{-3\varepsilon}\delta_j^{\alpha_1}), 
\end{align*}
where $P_{n,G}$ denotes the Gaussian analog of $P_n$. By the Kac-Rice formula  (see Lemmas \ref{l.rho1} and \ref{l.mean.Gauss}), we derive
\begin{align*}
    \mathbb E[N_{P_n}(J_j^{-})]  
    \ll   \int_{J_j^-(\delta_{j}^{1+\varepsilon})} \frac{dx}{1-x+\frac1 n} + O(\delta_j^{\alpha_1-3\varepsilon})
    \ll \delta_j^{\varepsilon},
\end{align*}
provided that $\varepsilon$ is sufficiently small. 

Similarly, we also have $\mathbb E[N_{P_n}(J_j^{+})] \ll \delta_j^{\varepsilon}$.

Note that the roots of the real polynomials $P_n$ counted in $Z_j$ will form conjugate pairs, therefore using the union bound and Lemma \ref{l.nearR}, we obtain 
\[
\mathbb P\left(Z_{j} \ge 1\right) \ll \delta_j^{-\varepsilon}\delta_j^{3\varepsilon/2}=\delta_j^{\varepsilon/2}.
\]
Combining with Theorem \ref{t.localcount}, we deduce that
    \[
    \mathbb E[Z_{j}]=\mathbb E[Z_{j} \pmb 1_{\{Z_{j}\ge 1\}}]\ll  \delta_j^{\varepsilon/4}.
    \]
Collecting estimates,  we obtain
\[
    \mathbb E[|N_{j}-\varphi_{j}|] \ll  \delta_{j}^{\varepsilon/4} \ll a_n^{\varepsilon/4}+n^{-\varepsilon/4},
\]
proving \eqref{e.E[Nj]}.

\section{Proof of Lemma \ref{l.estimate-varphi}} \label{a.l5.3}
The proof will be divided into two steps.

\setcounter{step}{0}
\begin{step}[Discrete sampling] Using Green’s  identity, we have
\begin{align*}
    \varphi_{j}=\sum_{\ell=1}^n\varphi_j\left(\zeta_\ell\right)
   =\frac{1}{2\pi}\int_{\mathbb C}\log|P_n(w)| \Delta \varphi_j(w)dw.
\end{align*}
Note that by \eqref{e.varphi-j} $\Delta \varphi_j$ is supported inside the (complex) ball $B_j$ centered at the midpoint of $J_j$, with radius $\frac 1 2|J_j|+2\delta_j^{1+\varepsilon}$, and $\|\Delta \varphi_j\|_{\sup}=O(\delta_j^{-2(1+\varepsilon)})$. 

Let $q_j\ge 1$ be an integer to be chosen later. By discrete  sampling, we may approximate $\varphi_j$, with high probability,  by a discrete sum of the form
\begin{align*}
     L_{j}&:=\frac{1}{q_j}\sum_{\ell=1}^{q_j} a_{j,\ell}\log|P_n\left(w_{j,\ell}\right)|,\;\; \text{where } 
     a_{j,\ell} :=\frac{1}{2\pi}|B_j|\Delta \varphi_{j}\left(w_{j,\ell}\right) =   O(\delta_j^{-2\varepsilon}).
\end{align*}

For simplicity, let $\widetilde L_j$ denote the Gaussian analogue of $L_j$ (with the same coefficients $a_{j,\ell}$). Let $\{w_{j,\ell}\}_{1\le \ell \le q_j}$ be independent samples drawn uniformly from $B_j$, also independent of the coefficients of $P_n$ and $P_{n,G}$. By standard Monte Carlo sampling (see, e.g., \cite{TV15}*{Lemma 6.1}), 
\begin{equation}
\label{e.varphi-Pn}
   \mathbb P_{(w_{j,\ell})}\left(\left|\varphi_{j}-L_j\right|>\lambda\right) \ll \frac{|B_j|}{q_j\lambda^2}\int_{B_j}  |\log  |P_n(w)||^2 |\Delta \varphi_j(w)|^2 dw.
\end{equation}
By Theorem \ref{t.locallogint}, it holds with probability $1-O( \delta_j^{\varepsilon})$ that
\begin{align*}
    \frac{1}{|B_j|}\int_{B_j}  |\log  |P_n(w)||^2 dw \ll  |\log \delta_j|^4.
\end{align*}
Conditioning on such an event and choosing $q_j\asymp \delta_j^{-8\varepsilon}$ and $\lambda=\delta_j^{\varepsilon}$,  it follows that the right-hand side of \eqref{e.varphi-Pn}, denoted by RHS, satisfies
    \[
    \text{RHS}\ll q_j^{-1}\lambda^{-2}\delta_j^{-4\varepsilon}|\log \delta_j|^4 \ll \delta_j^{\varepsilon}.
    \]
 We conclude that with probability $1-O(\delta_j^{\varepsilon})$,
    \[
    \left|\varphi_{j}-L_j  \right| \ll  \delta_j^{\varepsilon}.
    \]

Given that $F$ and its partial derivatives are bounded, we may apply the triangle inequality and the mean value theorem to obtain
\begin{align*}
    \mathbb E[F\left(\varphi_{1}, \dots, \varphi_{T}\right)]
    &=\mathbb E[F\left(L_{1}, \dots, L_{T}\right)]+O\Big(\sum_{j=1}^T \delta_j^\varepsilon\Big)\\
    &=\mathbb E[F(L_{1}, \dots, L_{T})]+O(a_n^{\varepsilon}+n^{-\varepsilon}).
\end{align*}
Thus, it remains to show that
\begin{equation}\label{e.FL}
    \mathbb E[F(L_{1}, \dots, L_{T})] - \mathbb E[F(\widetilde L_{1}, \dots, \widetilde L_{T})] \ll  a_n^{\varepsilon}.
\end{equation}
Here, we stress that the expectation is over the joint distribution of the sampling points and the polynomial coefficients.

Let $D_j=\{\left(j,\ell\right):  1\le \ell \le q_j\}$, $D:=\cup_{j=1}^T D_j$, and $q=|D|=\sum_{j=1}^T q_j$. Define $\hat F: \mathbb R^{q}\to \mathbb R$ as
    \[
    \hat F\left(y_{s}\right)_{s\in D}:=F\bigg(\frac{1}{q_1}\sum_{\ell=1}^{q_1}a_{1,\ell}y_{1,\ell}, \dots, \frac{1}{q_T}\sum_{\ell=1}^{q_T}a_{T,\ell}y_{T,\ell}\bigg).
    \]

To prove \eqref{e.FL}, it suffices to show that
\begin{align}\label{e.u-log}
    \mathbb E[\hat F\left(\log|P_n\left(w_s\right)|\right)_{s\in D}]-\mathbb E[\hat F\left(\log|P_{n, G}\left(w_s\right)|\right)_{s\in D}] \ll  a_n^{\varepsilon}.
\end{align}
    
 \end{step}
\begin{step}[Lindeberg swapping and proof of \eqref{e.u-log}]

Let $\widetilde F: \mathbb R^{q}\to \mathbb R$ be defined by
    \[
    \widetilde F\left(y_{s}\right)_{s\in D}=\hat F\left(y_{s}+\log \sigma\left(w_{s}\right)\right)_{s\in D}.
    \]
For $\Delta_j:=\log(\delta_j^{-10\varepsilon})$, let
    \[
    Y_1=\left\{\left(y_s\right)_{s\in D}\in \mathbb R^{q}: y_s\le -\Delta_j \text{ for some } j \text{ and } s\in D_j\right\},
    \]
    \[
    Y_2=\left\{\left(y_s\right)_{s\in D}\in \mathbb R^{q}: y_s\ge -\Delta_j-1 \text{ for all } j \text{ and } s\in D_j\right\}.
    \]
We write
\[
    \widetilde F=F_1+F_2 \equiv \left(1-f\right)\widetilde f + f\widetilde F,
\]
where $f:\mathbb R^q \to [0, 1]$ smooth and supported in $Y_2$, $f=1$ on the complement $Y_1^c$, and $\|f^{\left(\beta\right)}\|_{\infty}=o(1)$ for all $0\le |\beta|\le 3$. More precisely,
    \[
    f\left(y_s\right)_{s\in D}=\prod_{j=1}^T\prod_{s\in D_j}f_j\left(y_{s}\right),
    \]
where $f_j:\mathbb R \to \mathbb R$ are  smooth  such that
    \[
    \pmb 1_{[-\Delta_j, \infty)}\le f_j \le \pmb 1_{[-\Delta_j-1, \infty)},
    \]
and for $0\le \beta \le 3$,
    \[
    \|f^{\left(\beta\right)}_j\|_{\infty} \ll  1.
    \]
Clearly, $\supp F_1 \subset Y_1$ and $\supp F_2\subset Y_2$.

We now recall a version of Tao-Vu's Lindeberg swapping estimate (\cite{TV15}).

\begin{lemma} [\cite{Do21}] \label{l.lindeberg}
Assume that $(\xi_j)_{j=0}^n$ are independent real-valued with zero mean and unit variance and satisfy  \ref{A1}. Let $(G_j)_{j=0}^n$ be independent Gaussian and independent from $\xi_j$'s. Consider $H: \mathbb{C}^{n+1} \to \mathbb{C}$ such that $H$ is in $C^3$ when regarded as a function on $\mathbb{R}^{2n+2}$. There exists some finite positive constant $C=C\left(\varepsilon_0, C_0\right)$ such that 
    \[|\mathbb E[H\left(\xi_0, \dots, \xi_n\right)]  - \mathbb E[H\left(G_0,\dots,  G_n\right)]|   \le C \sum_{k=0}^nH_{2,k}^{1-\varepsilon_0} H_{3,k}^{\varepsilon_0},\]
    where 
    \[H_{l,k}:= \sum_{\beta=0}^{l}\left\|\frac{\partial^lH\left(z_0,\dots, z_n\right)}{\partial \re\left(z_k \right)^{\beta}\partial \im\left(z_k\right)^{l-\beta}}\right\|_{\infty},\quad l= 2, 3.\]
\end{lemma}

Using the chain rules, we obtain the following lemma.

\begin{lemma} \label{l.u-log-smooth}
Let $K: \mathbb R^{q}\to \mathbb R$ be a $C^3$ function that satisfies 
\begin{align}\label{e.K-bounds}
    \|K\|_{\sup}&=O(1), &\|\partial_{s}K\|_{\sup}&=O(\delta_j^{-10\varepsilon}),\\
    \notag    \|\partial_{s}\partial _{t}K\|_{\sup}&=O(\delta_j^{-10\varepsilon}\delta_{j'}^{-10\varepsilon}),&\|\partial_{s}\partial_{t}\partial_{u}K\|_{\sup}&=O(\delta_j^{-10\varepsilon}\delta_{j'}^{-10\varepsilon}\delta_{j''}^{-10\varepsilon}),
\end{align}
for all $s\in D_j$, $t\in D_{j'}$, $u\in D_{j''}$, and $1\le j, j', j'' \le T$. Then, for every $w_{ij\ell}$ in the ball $B\left(x_{ij},2\delta_j/3\right)$, we have 
    \[
    \mathbb E[K\bigg(\frac{P_n\left(w_s\right)}{\sqrt{\Var[P_n\left(w_s\right)]}}\bigg)_{s\in D}]-\mathbb E[K\bigg(\frac{P_{n, G}\left(w_s\right)}{\sqrt{\Var[P_{n,G}\left(w_s\right)]}}\bigg)_{s\in D}] \ll  a_n^{\varepsilon}.
    \]
\end{lemma}

\begin{proof}[Proof of Lemma \ref{l.u-log-smooth}]
Let $\sigma(z) = \sqrt{\Var[P_n(z)]}$. We apply Lemma \ref{l.lindeberg} to 
\[
H\left(\xi_0, \dots, \xi_n\right) := K\left(P_n\left(w_s\right)/\sigma\left(w_s\right)\right)_{s \in D}.
\]
For each $k = 0, \dots, n$, via explicit computations utilizing \eqref{e.K-bounds}, we derive for each $m\le 3$,
\begin{align*}
    H_{m,k} \ll  \bigg(\sum_{j=1}^T\delta_j^{-10\varepsilon}\sum_{s\in D_j} \frac {|v_{k} w_{s}^{k}|}{\sigma\left(w_{s}\right)}\bigg)^m.
\end{align*}
For $s\in D_j$, we have $w_s\in B_j$, so $1-|w_s|\asymp \delta_j$. Using conditions \ref{A2} and \ref{A3}, we see that
    \[
    \sigma^2(w_s) 
    \gg  \frac{1}{(1-|w_s|^2)^{2\tau+1}} \gg  \delta_j^{-2\tau-1}.
    \]
Therefore, by analyzing   $(1+x)^{2\tau} (1-\delta_j)^{2k}$ as a function of $k \in [0,\infty)$,
    \[
    \frac{|v_k w_{s}^k|^2}{\sigma^2(w_{s})} \ll \frac{(1+ k)^{2\tau} (1-\delta_j)^{2k}}{\sigma^2(w_{s})} \ll  \delta_j^{c},\quad s\in D_j,
    \]
 where $c:=\min(2\tau+1, 1)>0$. Using repeated applications of H\"older's inequality, for $\varepsilon>0$ sufficiently small, we obtain
    \[
    H_{2,k}^{1-\varepsilon_0}H_{3,k}^{\varepsilon_0} \ll  T^{1+\varepsilon_0}\sum_{j=1}^T\delta_j^{10\varepsilon}\sum_{s\in D_j}\frac {|v_{k} w_{s}^{k}|^2}{\sigma^2(w_{s})}.
    \]
Applying Lemma \ref{l.lindeberg}, we obtain
\begin{align*}
    \mathbb E[K\left(\frac{P_n\left(w_s\right)}{\sigma\left(w_s\right)}\right)_{s\in D}]-\mathbb E[K\left(\frac{P_{n, G}\left(w_s\right)}{\sigma\left(w_s\right)}\right)_{s\in D}]
    \ll  T^{1+\varepsilon_0}\sum_{j=1}^T \delta_j^{2\varepsilon}
    \ll a_n^{\varepsilon},
\end{align*}
recalling that $T\ll  \log n$ and $a_n=o((\log n)^{-A})$ for any $A>0$.
\end{proof}

We now prove \eqref{e.u-log}. For the contribution of $F_1$,  we set $\widetilde F_1=\|\widetilde F\|_{\infty}\left(1-f\right)$ and 
    \[
    K_1\left(y_{s}\right)_{s\in D}=\widetilde F_1\left(\log|y_s|\right)_{s\in D}.
    \]
We can verify that $K_1$ satisfies \eqref{e.K-bounds} and $|F_1(\log|y_s|)_{s\in D}|\le K_1(y_{s})_{s\in D}$. Applying Lemma \ref{l.u-log-smooth} to $K_1$, we obtain 
\begin{align*}
    &\mathbb E[F_1\bigg(\log\frac{|P_n(w_s)|}{\sigma(w_s)}\bigg)_{s\in D}]\\
    &\ll \mathbb E[K_1\bigg(\frac{P_n(w_s)}{\sigma(w_s)}\bigg)_{s\in D}]\\
    &=\mathbb E[K_1\left(\frac{P_{n, G}\left(w_s\right)}{\sigma_G\left(w_s\right)}\right)_{s\in D}]+O(a_n^{\varepsilon})\\
    &\ll  \|K_1\|_{\sup} \sum_{j=1}^T \mathbb P\bigg(\exists s\in D_j: \frac{|P_{n, G}\left(w_s\right)|}{\sigma_G\left(w_s\right)}\le e^{-\Delta_j}\bigg)+O(a_n^{\varepsilon})\\
    &\ll  \sum_{j=1}^T\sum_{s\in D_j}\delta_j^{10\varepsilon} +O(a_n^{\varepsilon}) \\
   \qquad &\ll  a_n^{\varepsilon}.
\end{align*}

For the contribution of $F_2$, let $K_2: \mathbb R^{q}\to \mathbb R$ be defined by
    \[
    K_2(y_{s})_{s\in D}=F_2(\log|y_s|)_{s\in D}.
    \]
We can verify that $K_2$ satisfies \eqref{e.K-bounds}. Applying Lemma \ref{l.u-log-smooth} yields
\begin{align*}
    &\mathbb E[F_2\bigg(\log\frac{|P_n(w_s)|}{\sigma(w_s)}\bigg)_{s\in D}]-\mathbb E[F_2\bigg(\log\frac{|P_{n, G}(w_s)|}{\sigma_G(w_s)}\bigg)_{s\in D}]\\
    &= \mathbb E[K_2\bigg(\frac{P_n(w_s)}{\sigma(w_s)}\bigg)_{s\in D}]-\mathbb E[K_2\bigg(\frac{P_{n, G}(w_s)}{\sigma_G(w_s)}\bigg)_{s\in D}]\\
    &\ll  a_n^{\varepsilon},
\end{align*}
which completes the proof.
\end{step}
\section{Proof of Lemma \ref{l.cor-in-out}} \label{a.l8.24}
By the reduction in Section~\ref{s.8g}, we may assume $v_j=(1+j)^\tau$ for all $j\ge 0$.  Recall that 
    \[
    r(x,y)=\frac{k(xy)}{\sqrt{k(x^2)k(y^2)}} \quad \text{and}\quad k(x)=\sum_{j=0}^n v_j^2x^j.
    \] 

We begin with the bound \eqref{e.r-mix}.
From \eqref{e.dik}, for $B$ sufficiently large,
\[
k(x^2)\ge \frac{\Gamma(2\tau+1)}{2(1-x^2)^{2\tau+1}},\quad x\in U_n,
\]
while, arguing as in \eqref{e.k*},
    \[
    k(y^2)= y^{2n}k^*(1/y^2) \ge \frac{n^{2\tau} y^{2n}}{2(1-1/y^2)}, \quad y\in U_n^*.
    \]
Since $|k(xy)|\le k(|xy|)$, it suffices to assume $x, y>0$. We distinguish two cases.

\textit{Case 1:} Suppose that $xy\le 1$. As in Lemma \ref{l.estvar}, 
\[
k(xy) \le  \frac{2\Gamma(2\tau+1)}{(1-xy+1/n)^{2\tau+1}}.
\]
Hence, for some constant $C_\tau>0$,
\begin{equation}
    \label{e.|r|}
    |r(x,y)| \le C_\tau \frac{(1-x)^{\tau+\frac 12}(1-1/y)^{\frac 12}}{(1-xy+1/n)^{2\tau+1}n^\tau y^n}.
\end{equation}
Let $z=1/y$, so that $x\le z \in U_n$. Then 
\begin{align*}
    |r(x,y)| \le C_\tau \frac{(1-x)^{\tau+\frac 12}(1-z)^{\frac 12}z^n}{(1-x/z+1/n)^{2\tau+1}n^\tau}. 
\end{align*}
We split according to the size of the denominator.

If $(1-x/z+1/n)^{-1} \le z/(2(1-x))$, then 
\[
|r(x,y)| \le \frac{C_\tau}{2^{2\tau+1}} \frac{[n(1-z)]^{1/2}z^{n+1}}{[n(1-x)]^{\tau+\frac 12}}\le \frac{C_\tau}{2^{2\tau+1}B^{\tau+\frac 12}} [n(1-z)]^{1/2}z^{n},
\]
since $n(1-x)\ge B$. Otherwise, the function 
\[
x\mapsto \frac{(1-x)^{\tau+\frac 12}}{(1-x/z+1/n)^{2\tau+1}}
\]
is increasing on $[1-a_n, z]$, hence maximized at $x=z$, yielding 
\[
|r(x,y)| \le C_\tau [n(1-z)]^{\tau+1}z^n.
\]
Since $n(1-z)\ge B$, both functions $z\mapsto (1-z)^{1/2}z^n$ and $z\mapsto (1-z)^{\tau+1}z^n$ attain their maximum on $U_n$ at $z=1-B/n$. Evaluating at this point gives 
\[
|r(x,y)| \le C_\tau\bigg(\frac{1}{2^{2\tau+1} B^{\tau}}+B^{\tau+1}\bigg) \left(1-\frac{B}{n}\right)^n \le e^{-B/2}
\]
for all sufficiently large $n$ and $B$. This proves the claim in Case 1.

\textit{Case 2:} Suppose that $xy>1$. As in \eqref{e.k*},
\[
0\le k(xy)=(xy)^nk^*\bigg(\frac{1}{xy}\bigg) \le  \frac{2n^{2\tau}(xy)^n}{1-\frac{1}{xy}+1/n}.
\]
Hence, for some constant $C_\tau'>0$, 
\[
|r(x,y)| \le C_\tau' \frac{n^{\tau}x^n(1-x)^{\tau+\frac 12}(1-z)^{\frac 12}}{1-z/x+1/n}, 
\]
where $z:=1/y\in U_n$ and $z\le x$. If 
\[\frac{1}{1-z/x+1/n}\le \frac{x}{2(1-z)},\]
then 
\[
|r(x,y)| \le \frac{C_\tau'}{2}\frac{x^{n+1}[n(1-x)]^{\tau+\frac 12}}{[n(1-z)]^{\frac 12}}\le \frac{C'_\tau}{2}B^{\tau} \left(1-\frac{B}{n}\right)^n \le e^{-B/2}.
\]
Otherwise, monotonicity in $z$ yields
\[
|r(x,y)| \le C_\tau' x^n[n(1-x)]^{\tau+1}\le C_\tau' B^{\tau+1} \left(1-\frac{B}{n}\right)^n \le e^{-B/2}.
\]
Combining the two cases establishes \eqref{e.r-mix}.

We now prove \eqref{e.dr-mix}. Differentiating $r$, we have 
\begin{equation}
    \label{e.|r01|}
    |r_{10}(x,y)|\le |r(x,y)|\bigg(\frac{|k'(xy)|}{|k(xy)|}+\frac{k'(x^2)}{k(x^2)}\bigg).
\end{equation}
Again, we may assume that $x, y>0$. Recall that for $x\in U_n$,
\[
\frac{k'(x^2)}{k(x^2)} \sim \frac{2\tau+1}{1-x^2} \quad \text{and}\quad r_{11}(x,x)\sim \frac{2\tau+1}{(1-x^2)^2}.
\]
We distinguish two regimes. 

If $xy\le 1$, then 
\[
\frac{|k'(xy)|}{|k(xy)|}=\frac{k'(xy)}{k(xy)}\sim \frac{2\tau+1}{1-xy+1/n}.
\]
For $1-xy+\frac{1}{n}>1-x^2$, there exists some constant $D_\tau>0$ such that 
\[
\frac{|k'(xy)|}{|k(xy)|}+\frac{k'(x^2)}{k(x^2)} \le D_\tau \sqrt{r_{11}(x,x)}, 
\]
which implies \eqref{e.dr-mix} when combined with \eqref{e.r-mix} and \eqref{e.|r01|}.  If $1-xy+\frac{1}{n} \le 1-x^2$, then for some constant $D_\tau>0$,
\[
\frac{|k'(xy)|}{|k(xy)|}+\frac{k'(x^2)}{k(x^2)} \le D_\tau\frac{(1-x)}{(1-xy+1/n)}\sqrt{r_{11}(x,x)}.
\]
Combining with \eqref{e.|r|} and \eqref{e.|r01|}, we get 
\begin{align*}
   |r_{10}(x,y)| \le C_\tau D_\tau \frac{(1-x)^{\tau+\frac 32}(1-1/y)^{\frac 12}}{(1-xy+1/n)^{2\tau+2}n^\tau y^n}\sqrt{r_{11}(x,x)}.
\end{align*}
Proceeding as in the proof of \eqref{e.r-mix}, we deduce that, for $B$ sufficiently large, 
\[
|r_{10}(x,y)|\le e^{-B/3}\sqrt{r_{11}(x,x)}.
\]

If $xy\ge 1$, we have 
\[
\frac{k'(xy)}{k(xy)}\sim \frac{1}{1-\frac{1}{xy}+1/n}.
\]
Arguing as above, we also deduce that 
\[
|r_{10}(x,y)|\le e^{-B/3}\sqrt{r_{11}(x,x)}.
\]
The estimate for $|r_{01}(x,y)|$ follows symmetrically. This proves \eqref{e.dr-mix}.

From \eqref{e.sigma1},
    \[
    \sigma_1^2 =r_{11}(x,x)-\frac{r_{10}^2(x,y)}{1-r^2(x,y)}.
    \]
From \eqref{e.r-mix} and \eqref{e.dr-mix}, the correction term is bounded by $e^{-B/2}r_{11}(x,x)$, yielding 
\[
\sigma_1 \asymp \sqrt{r_{11}(x,x)}.
\]
The same argument applies to $\sigma_2$, and \eqref{e.sigmix} is verified.

The estimate \eqref{e.r11-out} follows from \eqref{e.k*} and \eqref{e.djk*} by the same argument used to prove \eqref{e.r11}.

To prove \eqref{e.mu1mix}, observe that 
    \[
    |m(y)|=\frac{|M_n(y)|}{\sqrt{k(y^2)}}=\frac{|y^nM_n^*\left(1/y\right)|}{\sqrt{y^{2n}k^*\left(1/y^2\right)}} =|m^*\left(1/y\right)|.
    \]
Together with \eqref{e.m*phi}, we get
    \[
    |m(y)| \ll  \phi(1-|1/y|),\quad y\in U_n^*.
    \]
Thus, by using \eqref{e.mphi}, \eqref{e.m*phi}, and Lemma \ref{l.cor-in-out}, we conclude that
\begin{align*}
    \mu_1 & \ll  |m'(x)|+|\frac{r(x,y)r_{10}(x,y)}{1-r^2(x,y)}m(x)|+|\frac{r_{10}(x,y)}{1-r^2(x,y)}m(y)|\\
    &\ll  \frac{\phi(1-|x|)}{1-|x|}+\frac{\phi(1-|x|)}{1-|x|}+\frac{\phi(1-|1/y|)}{1-|x|},
\end{align*} 
which gives \eqref{e.mu1mix}.

The proof of \eqref{e.mu2mix} is similar.
\end{appendix}

\end{document}